\documentclass[10pt,notitlepage,reqno]{amsart}

\usepackage{comment}
\usepackage{amssymb,amsmath,amsthm,mathrsfs,mathtools}
\usepackage{amscd}
\usepackage[breaklinks=true]{hyperref}
\usepackage{url}
\usepackage{units}
\usepackage[usenames,dvipsnames]{xcolor}
\usepackage{graphicx}
\usepackage[all]{xy}
\usepackage{amsmath}
\usepackage{cleveref}
\usepackage{xy}
\hypersetup{
  colorlinks,
  linkcolor	=	Gray,
  urlcolor 	= Gray,
  citecolor	=	Gray,
  pdfauthor={Mirco Richter},
  pdfkeywords={Lie, infinity, Rinehart, Schouten, symplectic, multisymplectic},
  pdftitle={Higher Symplectic Structures on torsionless Lie Rinehart pairs II},
  pdfsubject={Lie, Rinehart, symplectic},
  pdfpagemode=UseNone
}

 \excludecomment{comment}

\let\originalleft\left
\let\originalright\right
\renewcommand{\left}{\mathopen{}\mathclose\bgroup\originalleft}
\renewcommand{\right}{\aftergroup\egroup\originalright}

\newcommand{\bwedge}{\raisebox{0.3ex}{${\scriptstyle\bigwedge\,}$}}

\newcommand{\smwedge}{{\scriptstyle \;\wedge\;}}

\newcommand{\ordinal}[1]{[$\;\!#1\,$]}

\newcommand{\sgn}[1]{\ensuremath{{\scriptstyle (-1)^{\scriptstyle #1}}}}
\newcommand{\R}{\mathbb{R}}
\newcommand{\Z}{\mathbb{Z}}
\newcommand{\N}{\mathbb{N}}

\newcommand{\I}{\infty}
\newcommand{\defeq}{\vcentcolon=}

\newcommand{\BIGOP}[1]{\mathop{\mathchoice
{\raise-0.22em\hbox{\huge $#1$}}
{\raise-0.05em\hbox{\Large $#1$}}{\hbox{\large $#1$}}{#1}}}

\def\clap#1{\hbox to 0pt{\hss#1\hss}}

\theoremstyle{plain}
\newtheorem{theorem}{Theorem}[section]
\newtheorem{proposition}[theorem]{Proposition}

\newtheorem{corollary}[theorem]{Corollary}

\newtheorem{definition}[theorem]{Definition}

\theoremstyle{remark}
\newtheorem*{remark}{Remark}
\newtheorem{example}{Example}

\title{Homotopy Poisson-n Algebras from N-plectic Structures}
\author{Mirco Richter}
\email{mirco.richter@email.de}
\date {\today}
\thanks{
email: \href{mailto:mirco.richter@email.de}{\tt mirco.richter@email.de}}

\begin{document}

\newlength{\mylength}
\setlength{\mylength}{\linewidth-2\multlinegap}

\begin{abstract}We associate a homotopy Poisson-n algebra to any higher symplectic structure, which generalizes the common symplectic Poisson algebra of smooth functions. This provides robust n-plectic prequantum data for most approaches to quantization.

UPDATE: It has been brought to my attention that the exterior product does not close on the exterior cotensors called Poisson cotensors in the present paper. Therefore the set of Poisson cotensors as presented, is not a homotopy Poisson-n algebra, quite yet. However for those Poisson cotensor that do multiply into Poisson cotensors under the exterior product, the computation remains valid and interesting, nonetheless.

Therefor this paper might better be seen as a hint towards homotopy Poisson-n algebras in higher symplectic geometry, rather then a finished theory. Despite this flaw, I choose not to withdrawal the paper, as I think the direction is still pretty interesting.
\end{abstract}
\maketitle
\section{Introduction}
The basic notion of higher symplectic geometry is very simple: Instead of
symplectic $2$-forms, consider closed forms of arbitrary degree. 

This  idea, however, has a remarkable long history. It 
only gradually emerged over the last few decades, mostly 
from an attempt to find a covariant Hamilton formalism 
for physical fields. In fact first origins can be traced back to the work of 
Vito Volterra \cite{V}, which was published as early as 1890.
Many different flavors appeared ever since, most of 
them with a strong motivation from real world applications.

Despite all the effort, one has to admit, that many such theories have a 
certain anachronistic touch. Often no full analog of the symplectic
Poisson algebra is known and this renders most of them
very different from their symplectic counterparts. In fact it 
effectively prohibits allot symplectic techniques, let alone the known 
approaches to quantization.

The present work aims to fill that gap. We show that the natural structure
to expect from higher symplectic data is not a Poisson, but a 
\textit{homotopy Poisson-n} algebra. 

Homotopy Poisson-$n$ algebras are relatively modern structures and an 
explicit description was only found recently \cite{GTV}.
Nevertheless they are very pleasant in the sense, that they
provide robust prequantum data. The abstract theory says, that they are
the homology of certain $\mathbb{E}_n$-algebras and in this 
sense at least a deformation quantization is always guaranteed to exist. 
In addition they fit into the folklore, that observables should be the homology 
of something.
\\\\
\noindent{\it Acknowledgements.} 
I would like to thank Igor Kanatchikov for a fruitful discussion on
products in multisymplectic geometry and Bruno Valette for helping me towards
an understanding of homotopy Poisson-$n$ algebras.
\section{Summary}
We suggest an $n$-plectic structure $(A,\mathfrak{g},\omega)$
to be a torsionless (\ref{torsionless}) Lie Rinehart pair $(A,\mathfrak{g})$ with
 a degree $(n+1)$ cocycle $\omega$,	chosen from the de Rham complex $\Omega^\bullet(\mathfrak{g},A)$
of exterior cotensors (\ref{n-plectic_struct}). 

The symplectic Poisson algebra of differentiable functions is then 
refined to the set of all de Rham cotensors 
$f\in \Omega^\bullet(\mathfrak{g},A)$, which 
 satisfy the constrained first order differential equation 
$$
\begin{array}{lcr}
i_x\omega=df & \text{ and }& i_y\omega =f, 
\end{array}
$$ 
for some, not necessarily unique, Hamilton tensor $x$ and Poisson 
constraint $y$. 

We write $\mathcal{P}ois(A,\mathfrak{g},\omega)$ for this solution set 
(\ref{Poisson_cotensors}) and call it the $n$-plectic Poisson algebra, since 
it is a homotopy Poisson-n algebra (\ref{homotopy_poisson_n}) with respect to the following operations:

A differential graded commutative structure (\ref{theorem_dga_algebra}), 
inherited from the de Rham algebra $\Omega^\bullet(\mathfrak{g},A)$.
This means, that the differential restricts to a map
$$
d: \mathcal{P}ois(A,\mathfrak{g},\omega) \to \mathcal{P}ois(A,\mathfrak{g},\omega)
$$
on Poisson cotensors and that the ordinary exterior cotensor product  
restricts to an $n$-plectic multiplication
$$
\wedge : \mathcal{P}ois(A,\mathfrak{g},\omega) \times
 \mathcal{P}ois(A,\mathfrak{g},\omega) \to \mathcal{P}ois(A,\mathfrak{g},\omega)\;.
$$
If $x_1$ and $x_2$ are Hamilton tensors associated to the (homogeneous)
Poisson cotensors 
$f_1$ and $f_2$ by the fundamental equation $i_x\omega=df$, the 
exterior tensor
$$
x_{f_1\wedge f_2}:=(-1)^{|f_1|}j_{j_1}x_2 + (-1)^{(|f_1|-1)|f_2|}j_{f_2}x_1
$$
is a Hamilton tensor, associated to $f_1\smwedge f_2$ by the same equation. 
Moreover, if $y_2$ is a Poisson constraint associated to $f_2$ by the constraint 
equation $i_y\omega=f$, the exterior tensor
$$
y_{f_1\wedge f_2}\defeq  j_{f_1}y_2
$$
is a Poisson constraint, associated to $f_1\smwedge f_2$. (Here 
$j_{(\cdot)}\scriptstyle{(\cdot})$ is the
left contraction of a tensor by a cotensor (\ref{left contration})) 
This makes Poisson cotensers into a differential graded commutative algebra
(\ref{theorem_dga_algebra}).

The first step beyond this commutative structure appears as
the homotopy Poisson $2$-bracket, which is 
defined for any Poisson cotensors $f_1$ and $f_2$ and
associated (homogeneous) Hamilton tensors $x_1$ and $x_2$ by
$$
\{f_1,f_2\}=-L_{x_1}f_2+(-1)^{(|x_1|-1)(|x_2|-1)}L_{x_2}f_1\;.
$$
This bracket does not depend on the actual choices (\ref{well_def_2}) 
of $x_1$ or $x_2$ and is $(n-1)$-fold graded antisymmetric with respect to the
tensor degree of the $f_i$'s.
The de Rham differential acts as a derivation (\ref{jacobi-2}) and 
an associated Hamilton tensor can be computed by the 
Schouten-Nijenhuis bracket
$$
x_{\{f_1,f_2\}}=\, 2[x_2,x_1]\;.
$$
If $y_1$ resp. $y_2$ are Poisson constraints associated to $f_1$ resp. $f_2$,
a Poisson constraint of their Poisson bracket is given by
$$
y_{\{f_1,f_2\}}= \textstyle\left[x_2,y_1\right] -
(-1)^{(|x_1|-1)(|x_2|-1)}\left[x_1,y_2\right]\;.
$$
In case $n=1$ this bracket is equal to the usual 
symplectic Poisson bracket, up to the factor '2' and for $n\geq 2$ it refines
the $n$-pletic Lie bracket found by Rogers \cite{CR}, with an 
additional coboundary term (\ref{relation_to_Rogers}). 

Unlike the symplectic case, the Jacobi equation does not hold 'on 
the nose' as long as $n>1$. Instead it is satisfied up to 
certain coboundary terms (\ref{jacobi_id}) only. These terms are controlled by 
what we call the homotopy Poisson $3$-bracket, which is defined for any 
three Poisson cotensors $f_1$, $f_2$ resp. $f_3$ and associated Hamilton tensors 
$x_1$, $x_2$ resp. $x_3$ by the 'shuffle contractions'
\begin{multline*}
\{f_1,f_2,f_3\}\defeq 
i_{[x_2,x_1]}f_3 -(-1)^{(|x_2|-1)(|x_3|-1)}i_{[x_3,x_1]}f_2\\
+(-1)^{(|x_1|-1)(|x_2|-1)+(|x_1|-1)(|x_3|-1)}i_{[x_3,x_2]}f_1\;.
\end{multline*}
Again this does not depend on the actual choices of 
neither $x_1$, $x_2$ nor $x_3$ and associated Hamilton or Poisson 
constraint tensors can be computed explicit (\ref{homotopy_3_bracket}). 

With additional operators however come additional Jacobi-like 
equations and to control them we need additional operators. The general 
pattern is then an infinite series of $k$-ary brackets for arbitrary 
integers $k\in\N$. We call these operators the homotopy Poisson $k$-brackets and
define them inductively:

If $f_1,\ldots,f_{k+1}$ are Poisson  cotensors and
$x_{\{f_1,\ldots,f_k\}}$ is a Hamilton tensor associated to their 
homotopy Poisson $k$-bracket for some $k$, the homotopy Poisson $(k+1)$-bracket is  
the shuffled contraction
$$
\{f_1,\ldots,f_{k+1}\}\defeq \textstyle\sum_{\sigma\in Sh(k,1)}
\sgn{\sigma+k}e(s;sx_1,\ldots,sx_{k+1})
i_{x_{\{f_{\sigma(1)},\ldots,f_{\sigma(k)}\}}}f_{\sigma(k+1)}\;,
$$
where $|sx|$ just means $|x|+1$, $Sh(i,j)$ is the set of $(i,j)$-shuffle
permutations (\ref{shuffle_permutation}) and  
$\sgn{\sigma}e(\sigma;\cdot,\ldots,\cdot)$ is the antisymmetric Koszul sign 
(\ref{Koszul_sign_rules}) of $\sigma$.

Again this map does not depend on the choice of any $x_i$
and an associated Hamilton tensor as well as a Poisson constrain
can be computed inductively (\ref{main_theorem_1}). 

All these Poisson brackets interact with each other and the de Rham 
differential in terms 
of a so called $(n-1)$-fold shifted homotopy Lie algebra 
(\ref{homotopy_Lie_algebra}). The 
$(n-1)$-fold shifting is due to the degree relation $|x|= |f|+n$ 
between a Poisson cotensor $f$ and any associated Hamilton tensor $x$. This
in turn originates from the fundamental equation $i_x\omega=df$.

The interaction between the Poisson $2$-bracket 
and the exterior product is controlled by the Leibniz equation, which 
does not hold 'strictly' as long as $n>1$, 
but only up to certain correction terms.

These correction terms in turn are controlled by what we call the first Leibniz operator.
It is defined for any three Poisson cotensors $f_1$, $f_2$ resp. $f_3$ and 
associated Hamilton tensors $x_1$, $x_2$ resp. $x_3$ by 
\begin{multline*}
\{f_1 \| f_2, f_3\}\defeq
 -i_{x_{1}}(f_{2}\smwedge f_{3})
  +\sgn{(|x_{1}|-1)(|x_{f_{2}\wedge f_{3}}|-1)}
   i_{x_{f_{2}\wedge f_{3}}}f_{1}\\
\phantom{mmmmmmm}+\sgn{|x_{1}||f_{2}|} f_{2}\smwedge\left(i_{x_{1}}f_{3}
  -\sgn{(|x_{1}|-1)(|x_{3}|-1)} i_{x_{3}}f_{1}\right)\\
+\left(i_{x_{1}}f_{2}
 -\sgn{(|x_{1}|-1)(|x_{2}|-1)} i_{x_{2}}f_{1}\right)\smwedge f_{3}
\end{multline*}
and does not depend on the actual choices of any $x_i$.
An associated Hamilton tensor as well as a Poisson constrain
can be computed explicit. It is graded symmetric with respect to
the arguments on the right and interacts with the Poisson bracket and the
de Rham differential in terms of the $n$-plectic Leibniz equation:
\begin{multline*}
\{f_{1},f_{2}\smwedge f_{3}\}=
\{f_{1},f_{2}\}\smwedge f_{3}
+(-1)^{(|x_{1}|-1)|f_{2}|}f_{2}\smwedge\{f_{1},f_{3}\}\\
+d\{f_{1}\|f_{2},f_{3}\}
+\{df_{1}\|f_{2},f_{3}\}
-(-1)^{|x_{1}|}\{f_{1}\|df_{2},f_{3}\}\\
-(-1)^{|f_{2}|+|x_{1}|}\{f_{1}\|f_{2},df_{3}\}\;.
\end{multline*}
The interaction between the higher Poisson brackets and the exterior product has
no analog in the symplectic case. Nevertheless, higher 'Leibniz-like' equations
appear (\ref{higher_first_Leibniz}).
We say that such an equation holds 'strictly' or 'on the nose' if
\begin{multline*}
\{f_{1},\ldots,f_{k},f_{k+1}\smwedge f_{k+2}\}=
  \{f_{1},\ldots,f_{k},f_{k+1}\}\smwedge f_{k+2}\\
+\sgn{(\sum_{i=1}^{k}|x_{i}|-1)|f_{k+1}|}
 f_{k+1}\smwedge\{f_{1},\ldots,f_{k},f_{k+2}\}\;.
\end{multline*}
However, in general this equation does not hold strictly, but again only up to 
certain correction terms, which we call the higher Leibniz operators.

Typesetting these operators (\ref{higher_leibniz_op}) and all the 
additional equations (\ref{second_higher_leibniz}), (\ref{third_higher_leibniz}) 
is challenging, since explicit expression get quite complicated for $k\gg 1$.

The complete picture is that the $(n-1)$-fold shifted Lie algebra of higher Poisson brackets interacts
with the differential graded structure and the various Leibniz operators in
terms of a homotopy Poisson-$n$ algebra.

Finally it should be noted, that in an actual 
$n$-plectic setting all operators are trivial beyond a 
certain bound $\mathcal{O}(n)$, which means that for small $n$, only 
'a few' operators are different from zero. 
Loosely speaking we can say that the more $n$ deviates from
$1$, the more the structure deviates from being a Poisson algebra.

\section{Higher Symplectic Structures}
\subsection{N-plectic structures}
We define higher symplectic structures as torsionless Lie Rinehart pairs 
together with a distinguished cocycle from their de Rham complex. 
For an introduction to theses terms look at appendix 
(\ref{Lie_Rinehart_section}), or the references therein.
\begin{definition}\label{n-plectic_struct}
Let $(A,\mathfrak{g})$ be a torsionless 
(semi-reflexive) Lie Rinehart pair and 
$\omega\in \Omega^\bullet(\mathfrak{g},A)$ a degree $(n+1)$ cocycle 
with respect to the de Rham differential. Then
$(A,\mathfrak{g},\omega)$ is called an \textbf{n-plectic} 
\textbf{structure} and $\omega$ is called its 
\textbf{n-plectic cocycle}.
\end{definition}
We do not distinguish between $n$-plectic cocycles that are degenerate on
vectors and those that are not. For a \textit{homotopy Poisson-n algebra} 
to exist here, the fundamental pairing between functions and vector fields has to 
generalize to tensors and cotensors \textit{in a range of degrees} and almost all 
$n$-plectic cocycles have a non trivial kernel on higher tensors. 
Whether they are degenerated on vectors or not.
Unique pairings are an exception in the general $n$-plectic setting.
\begin{remark}Of course there are $n$-plectic structures, which are non
degenerate on vectors. If this becomes important, we could call them
\textit{non-degenerate}. For the general theory however, 
this distinction is irrelevant.
\end{remark}

The following example shows that
any torsionless Lie Rinehart pair gives rise to an $n$-plectic
structure, no matter how degenerate its de Rham complex might be:  
\begin{example}[Trivial $n$-plectic structure]Let 
$(A,\mathfrak{g})$ be a torsionless Lie Rinehart pair and 
$\omega\in\Omega^{n+1}(\mathfrak{g},A)$ the zero cocycle. 
Then $(A,\mathfrak{g},\omega)$ is called the 
\textbf{trivial} $n$-plectic structure.
\end{example}
Since the idea of "$n$-plectic" has such a long historical prologue, it appears 
in a lot of flavors. The following examples just list a few:
\begin{example}[Symplectic manifold]
Any symplectic manifold $(M,\omega)$ is equivalently an
$1$-plectic structure $(C^\I(M),\mathfrak{X}(M),\omega)$ on the
Lie Rinehart pair of smooth functions and vector fields over $M$.
\end{example}
\begin{example}[Presymplectic manifold]
Any presymplectic manifold $(M,\omega)$ is equivalently an
$1$-plectic structure $(C^\I(M),\mathfrak{X}(M),\omega)$ on the
Lie Rinehart pair of smooth functions and vector fields over $M$.
\end{example}
\begin{example}[Multisymplectic fiber bundle]
Let $p:P \to M$ be a smooth fiber bundle of rank $N$ 
over an $n$-dimensional manifold equipped with a closed and non-degenerate 
$(n+1)$-form $\omega$ defined on the total space which is $(n-1)$-\textit{horizontal} 
and admits an involutive and isotropic vector subbundle of the vertical bundle $VP$ 
of codimension $N$ and dimension $Nn+1$. 
(In \cite{FO1}, this is called a \textit{multisymplectic fiber bundle}). Then the
triple $(C^\I(P),\mathfrak{X}(P),\omega)$ is an $n$-plectic structure.
\end{example}
\subsection{Poisson cotensors}Every symplectic manifold has a Lie
bracket, which interacts with the 'dot'-product of smooth functions in
terms of a Poisson algebra. The origin of this bracket is the fundamental and unique pairing 
\begin{equation}\label{fake_fundamental_1}
i_x \omega = df
\end{equation}
between functions and vector fields, defined in terms of the 
non-degenerate symplectic 2-form.

In this section, we generalize the idea to the $n$-plectic 
setting and show that a certain constraint condition can replace the requirement 
of the pairing to be unique. 
This insight goes back to the work of Forger, Paufler and R\"omer \cite{FPR3}. 

Before we start, lets note the following important observations:
For a general n-plectic structure a pairing like (\ref{fake_fundamental_1})
makes sense in a range of cotensor degrees, not only in degree $(n-1)$. 
However the kernel of $\omega$ is  potentially 
non-trivial on tensors of arbitrary degrees and consequently the association 
\begin{center}
tensor $\Leftrightarrow$ cotensor 
\end{center}
via equation (\ref{fake_fundamental_1}) is neither always defined
nor unique. To some extend, this already happens in presymplectic 
geometry \cite{JLB}. 

One way to handle this would be to restrict the theory exclusively to 
n-plectic cocycles that are non-degenerate on vectors and the pairing 
(\ref{fake_fundamental_1}) only to
vectors and appropriate $(n-1)$ cotensors. This however comes with a lot of disadvantages:

If we restrict to non-degenerate cocycles only, we
exclude important applications like contact- or presymplectic-manifolds. Moreover not every torsionless Lie Rinehart pair has non-degenerate
cocycles and therefore the theory wouldn't be natural anymore.

If we allow degeneracy, but restrict the pairing to certain vectors and $(n-1)$ cotensors, a homotopy
\textbf{Poisson} structure can not arise from the de Rham
differential and the exterior product. Therefore such a 
theory is eventually restricted to a bare homotopy Lie structure, which is
not enough for some approaches to quantization.

Another consideration is that the Jacobi identity of the symplectic Lie 
bracket depend on properties that can't be expected in a general n-plecic 
setting. If the Poisson structure has to be replaced by a more general 
\textit{homotopy} Poisson-$n$ structure, a combination of the defining  
structure equations (\ref{main_hom_pois_struc_eq}) and the fundamental pairing 
(\ref{fake_fundamental_1}) leads to the equation
$$
i_y\omega = 
\textstyle\sum^{i>1}_{i+j=k+1}\sum_{s\in Sh(j,k-j)}
\pm\{\{f_{s(1)}, \ldots, f_{s(j)}\}, f_{s_(j+1)}, \ldots, f_{s(k)}\}
$$
and to ensure the existence of a solution $y$ to this equation,
the important additional constraint equation
\begin{equation}\label{fake_fundamental_2}
i_y \omega = f
\end{equation}
is required to hold on all arguments in addition to the first pairing in 
(\ref{fake_fundamental_1}).

We call (\ref{fake_fundamental_2}) the 
\textbf{Poisson constraint} of the fundamental pairing (\ref{fake_fundamental_1}),
since it is the only constraint that has to
be made on the solutions of (\ref{fake_fundamental_1}) in order 
to form a homotopy Poisson-$n$ algebra.
\begin{remark}This additional constraint equation 
is completely invisible in symplectic geometry, as it only appears implicitly.
To see that, let $\eta$ be the Poisson bivector field,
 associated to the symplectic form $\omega$. Then the bivector 
 $f\cdot\eta$ is a solution to (\ref{fake_fundamental_2}) 
for any function $f$ since 
$$
i_{\left(f\cdot \eta\right)}\omega =f\;.
$$ 
\end{remark}
\begin{remark}
To my knowledge, such a Poisson constraint appeared
first in the work of Forger and R\"omer \cite{FPR3}, where they introduced it
to handle the previously 
mentioned ambiguity inherent in the fundamental pairing (\ref{fake_fundamental_1}). 
In \cite{FPR3}, solutions to both equations (\ref{fake_fundamental_1}) and 
(\ref{fake_fundamental_2}) are called \textit{Poisson forms}.
\end{remark}
Summarizing all this, we define the n-plectic equivalent to
the Poisson algebra of smooth functions as the cotensor solution set of the
fundamental pairing with its Poisson constraint:
\begin{definition}
[The fundamental equation and its Poisson constraint]\label{Poisson_cotensors}
Let $(A,\mathfrak{g},\omega)$ be an $n$-plectic structure,
$X^\bullet(\mathfrak{g},A)$ the exterior tensor power
and $\Omega^\bullet(A,\mathfrak{g})$ the exterior cotensor power. The
\textbf{fundamental equation} of the $n$-plectic structure is the first order (algebraic) differential equation
\begin{equation}\label{fundamental_eq}
i_x\omega =df,
\end{equation}
defined on the product
$X^\bullet(\mathfrak{g},A)\times \Omega^\bullet(\mathfrak{g},A)$.
If the pair $(x,f)$ is a solution, $x$ is called a \textbf{Hamilton tensor},
$f$ is called a \textbf{Hamilton cotensor} and both are called 
\textbf{associated} to each other. 

We write $\mathcal{H}am(A,\mathfrak{g},\omega)$ for the set 
of all Hamilton tensors and $\mathcal{H}am^\ast(A,\mathfrak{g},\omega)$ 
for the set of all Hamiltonian cotensors. 

Moreover the \textbf{Poisson constraint} of the $n$-plectic structure
is the algebraic equation
\begin{equation}\label{Poisson_constraint_eq}
i_y\omega =f,
\end{equation}
defined on the product 
$X^\bullet(\mathfrak{g},A)\times \mathcal{H}am^\ast(A,\mathfrak{g},\omega)$. 
If a pair $(y,f)$ is a solution,
$f$ is called a \textbf{Poisson cotensor} and  $y$ is called a \textbf{Poisson constraint} associated to $f$.
We write $\mathcal{P}ois(A,\mathfrak{g},\omega)$ for the set of all Poisson 
cotensors. 
\end{definition}
Neither of the previously defined sets is empty:
Both, the fundamental equation and its Poisson constraint 
are linear in each argument and therefore the
kernel of $\omega$ and the zero cotensor are always a solution.

Each Poisson cotensor $f\in \mathcal{P}ois(A,\mathfrak{g},\omega)$ has a not 
necessarily unique 
Hamilton tensor $x$ associated to $f$ by the equation $i_x\omega=df$ 
and a not necessarily unique 
Poisson tensor $y$ associated to $f$ by the equation $i_y\omega=f$. 
We usually use the symbol $x$ to indicate that the tensor is Hamilton and the 
symbol $y$ to indicate that the tensor is a Poisson constraint.
\begin{corollary}
Let $(A,\mathfrak{g},\omega)$ be an $n$-plectic structure and $x\in \mathcal{H}am(A,\mathfrak{g},\omega)$ a Hamilton tensor.
Then $L_x\omega=0$.
\end{corollary}
\begin{proof}
Since $i_x\omega=df$ for some cotensor $f$ we get $di_x\omega=0$ and therefore
$L_x\omega=0$ from Cartans infinitesimal homotopy formular, since $\omega$ is a cocycle.
\end{proof}
\begin{corollary}
Let $(A,\mathfrak{g},\omega)$ be an $n$-plectic structure, 
$f\in \mathcal{P}ois(A,\mathfrak{g},\omega)$ a homogeneous Poisson cotensor with 
associated Poisson constraint $y\in X(\mathfrak{g},A)$ and associated Hamilton
tensor $x\in \mathcal{H}am(A,\mathfrak{g},\omega)$. If both $x$ and $y$ are
homogeneous, then 
\begin{equation}
\begin{array}{lcr}
|x|=|f|+n& \mbox{ and } &|y|=|x|+1
\end{array}
\end{equation}
with respect to the tensor grading. Moreover the following boundary equation 
is satisfied:
\begin{equation}
di_y\omega = i_x\omega\;.
\end{equation}
\end{corollary}
\begin{proof}
Immediate, since we assume tensors to be concentrated in positive tensor
degrees and cotensors as concentrated in negative tensor degrees.
\end{proof}
One of the most important consequence of the Poisson constraint 
(\ref{Poisson_constraint_eq}) is what we call the \textbf{kernel property}:
The kernel of the $n$-plectic cocycle $\omega$ is always contained in the kernel
of a Poisson cotensor. As a consequence, the contractions of a Poisson cotensor 
along tensors, associated to the same cotensor, are equal.
The following proposition makes this precise:
\begin{proposition}\label{kernel_prop}Let 
$(A,\mathfrak{g},\omega)$ be an $n$-plectic structure and
$f\in \mathcal{P}ois(A,\mathfrak{g},\omega)$ a Poisson cotensor. Then
\begin{equation}
\ker(\omega) \subset \ker(f).
\end{equation}
If $y$ and $y'$ are tensors, with $i_{y}\omega=f$ and $i_{y'}\omega=f$, 
the difference $y-y'$ is an element of the kernel of $\omega$ 
and the contractions $i_yg$ and $i_{y'}g$ are equal for all Poisson cotensors 
$g\in \mathcal{P}ois(A,\mathfrak{g},\omega)$.

Similar,  if $x$ and $x'$ are tensors, with $i_{x}\omega=df$ and 
$i_{x'}\omega=df$, 
the difference $x-x'$ is an element of the kernel of $\omega$ 
and the contractions $i_xg$ and $i_{x'}g$ are equal for all Poisson cotensors 
$g\in \mathcal{P}ois(A,\mathfrak{g},\omega)$.
\end{proposition}
\begin{proof} The first part is an implication of the Poisson constraint equation
(\ref{Poisson_constraint_eq}). To see that assume $\xi \in \ker(\omega)$. Then 
there is an exterior tensor  $y$ with $i_\xi f = i_\xi i_y\omega =
\pm i_y i_\xi \omega=0$. 

For the second and third part compute
$0=f-f=i_y\omega - i_{y'}\omega= i_{\left(y-y'\right)}\omega$ and
$0=df-df=i_x\omega - i_{x'}\omega= i_{\left(x-x'\right)}\omega$, respectively. Similar $i_yg-i_{y'}g=i_{(y-y')}g=0$ and $i_xg-i_{x'}g=i_{(x-x')}g=0$ follows 
from the kernel property of $g$.
\end{proof}
\subsection{The differential graded commutative algebra}
We show that the exterior product and the de Rham differential project 
from arbitrary cotensors to Poisson cotensors and derive explicit formulas
for their associated Hamilton tensors and Poisson constraints. 
This refines the symplectic function algebra 
to the general $n$-plectic setting.

We start with an important technical detail, which is
a consequence of the Poisson constraint equation (\ref{Poisson_constraint_eq})
and equation (\ref{multi_rules-5}):
\begin{proposition}\label{prop_product}
Let $(A,\mathfrak{g},\omega)$ be an $n$-plectic
structure, $f\in \mathcal{P}ois(A,\mathfrak{g},\omega)$ a Poisson cotensor
and $x\in X(\mathfrak{g},A)$ any tensor. Then
\begin{equation}\label{fundament_product}
i_{j_f x}\omega = f\smwedge i_x\omega\;.
\end{equation} 
\end{proposition}
\begin{proof}
Since $f$ is Poisson, there is an associated Poisson constraint
$y\in X^\bullet(\mathfrak{g},A)$ with $i_y\omega=f$. Equation (\ref{multi_rules-5}) then implies the proposition.
\end{proof}
Equation (\ref{fundament_product}) does not work for arbitrary cotensors.
As the following counterexample shows, the cotensor
really needs to be Poisson:
\begin{example}Consider the presymplectic manifold 
$(\R^3,dx^1\smwedge dx^2)$, the closed 1-form $f\defeq dx^3$ and the 
Hamilton vector field $X\defeq \partial_1$. Clearly there can not
be a vector field $Y$ such that $i_Y\omega=dx^3$ and
therefore $f$ is Hamilton but not Poisson.
In this case equation (\ref{fundament_product}) is not satisfied since
we compute
$$
i_{j_f X}\omega=i_{j_{dx^3}\partial_1}\omega=0 \neq
dx^3\smwedge dx^2= f\smwedge i_X\omega\;. $$
\end{example}
Besides the kernel property (\ref{kernel_prop}), this simple counterexample is 
another justification for our definition of Poisson cotensors. 
The existence of a Poisson constraint 
(\ref{fake_fundamental_2}), hence a valid equation (\ref{fundament_product}), is
necessary for the exterior product to close on Poisson cotensors as we show in
(\ref{theorem_dga_algebra}). 

The following theorem is the first major step towards our homotopy Poisson-$n$ 
algebra. It solves the long standing problem of how to define a natural product 
in higher symplectic geometry. In a subsequent corollary we give explicit 
constructions for associated Hamilton tensors and Poisson constraints.
\begin{theorem}\label{theorem_dga_algebra}
Let $(A,\mathfrak{g},\omega)$ be an
$n$-plectic structure. Then
$\mathcal{P}ois(A,\mathfrak{g},\omega)$ is
a differential graded commutative and associative subalgebra of the de Rham 
algebra $\Omega^\bullet(\mathfrak{g},A)$.
\end{theorem}
\begin{proof} We have to to show, that $\mathcal{P}ois(A,\mathfrak{g},\omega)$
is a linear subspace and that both the exterior product 
and the de Rham differential close on Poisson
cotensors, i.e. that the fundamental 
equation (\ref{fundamental_eq}) and its Poisson constraint equation 
(\ref{Poisson_constraint_eq}) have solutions. 

To see the linear structure, observe that
both the fundamental equation and its Poisson constraint equation 
(\ref{Poisson_constraint_eq}) are linear in all arguments. 

To see that the exterior differential closes, 
let $f\in\mathcal{P}ois(A,\mathfrak{g},\omega)$
be a Poisson cotensor. Then there exists a Hamilton tensor $x$, associated to 
$f$ by the fundamental equation $i_x\omega=df$, which in turn is a Poisson 
constraint associated to $df$. In addition any element
from $\ker(\omega)$ is a Hamilton tensor associated to $df$. Therefore
$df\in\mathcal{P}ois(A,\mathfrak{g},\omega)$

To solve the Poisson constraint equation (\ref{Poisson_constraint_eq}), 
for the exterior product $f_1\smwedge f_2$ of 
Poisson cotensors $f_1$ and $f_2\in \mathcal{P}ois(A,\mathfrak{g},\omega)$,
let $y\in X^\bullet(\mathfrak{g},A)$ be a Poisson constraint associated to $f_2$. 
Using (\ref{fundament_product}) we get 
$i_{j_{f_1}y}\omega = f_1\smwedge f_2$, which shows that $j_{f_1}y$ is a Poisson
tensor associated to $f_1\smwedge f_2$.

The find a solution to the fundamental equation (\ref{fake_fundamental_1}),
let $x_1$ and $x_2$ be solutions of $i_{x_1}\omega=df_1$ and
$i_{x_2}\omega=df_2$, respectively, which exists, since both cotensors are
Poisson. Then, using (\ref{fundament_product}), we compute
\begin{align*}
(-1)^{|f_1|}i_{j_{f_1}x_2}\omega
&+(-1)^{(|f_1|-1)|f_2|}i_{ j_{f_2}x_1}\omega\\
&= (-1)^{|f_1|}f_1\smwedge df_2 + (-1)^{(|f_1|-1)|f_2|}f_2\smwedge df_1 \\
&=df_1\smwedge f_2 + (-1)^{|f_1|}f_1\smwedge df_2 \\
&=d(f_1\smwedge f_2)\;.
\end{align*}
\end{proof}
The previous proof provides an explicit way to compute Hamilton tensors and
Poisson constraints associated to the exterior product of two  Poisson cotensors. 
In particular we have:
\begin{corollary}[Associated tensors]
\label{product_tensors}
Let $f_1$, $f_2\in \mathcal{P}ois(A,\mathfrak{g},\omega)$ be two
Poisson cotensors with associated Hamilton tensors
$x_1$ and $x_2$ , respectively.
Then the exterior tensor
\begin{equation}
x_{f_1\wedge f_2}\defeq (-1)^{|f_1|}j_{f_1}x_2 +(-1)^{(|f_1|-1)|f_2|}j_{f_2}x_1
\end{equation}
is a solution to the fundamental equation $i_{x_{f_1\wedge f_2}}\omega=d(f_1\smwedge f_2)$ and therefore a Hamilton tensor associated to $f\smwedge g$.

If $y_2$ is associated to $f_2$ by the Poisson constraint
equation $i_{y_2}\omega=f_2$, the exterior tensor
\begin{equation}\label{product_Poisson_constraint}
y_{f_1\wedge f_2}\defeq j_{f_1}y_2
\end{equation}
is a solution to the equation $i_{y_{f_1\wedge f_2}}\omega=f_1\smwedge f_2$ and 
therefore a Poisson constraint associated to $f\smwedge g$.
\end{corollary}
\begin{proof}
See the proof of theorem (\ref{theorem_dga_algebra}).
\end{proof}
\begin{remark}Strictly speaking it is not correct to write 
$x_{f_1\wedge f_2}$ since associated tensors are in general not unique and
the expression depends on the particular chosen Hamilton tensors
 $x_i$. By abuse of notation we stick to the symbol
$x_{f_1\wedge f_2}$ to indicate that we have at least \textit{some} 
Hamilton tensor associated
to the exterior product $f_1\smwedge f_2$. 
According to proposition (\ref{kernel_prop}) 
all representatives compute equal contractions and Lie derivations of
Poisson cotensors and therefore the risk of confusion is low.  
\end{remark}
\begin{corollary}\label{corollary_main_1}
Let $(A,\mathfrak{g},\omega)$ be an $n$-plectic structure.
Then the differential graded algebra 
$(\mathcal{P}ois(A,\mathfrak{g},\omega),\wedge,d)$, satisfies the structure
equation (\ref{main_hom_pois_struc_eq}) of a homotopy Poisson-$n$ algebra
for the parameter $k=1$ and all $p_1\in\N$.
\end{corollary}
\begin{proof}
Define the needed structure maps of (\ref{main_hom_pois_struc_maps}) 
for $k=1$ and $f_1,f_2\in \mathcal{P}ois(A,\mathfrak{g},\omega)$ by:
$D_{1}(s^{n-1}f_1) \defeq sdf_1$,
$D_{2}(s^{n-2}(\overline{sf_1\otimes sf_2}))\defeq \sgn{|sf_1|}s(f_1\smwedge f_2)$
and $D_{q}=0$ for all $q\geq 3$.

Then all of these shifted structure maps are homogeneous of degree $(1-n)$ and 
since $D_{2}(s^{n-2}(\overline{sf_2\otimes sf_1}))+(-1)^{|sf_1||sf_2|}
D_{2}(s^{n-2}(\overline{sf_1\otimes sf_2}))=0$, they have the correct 
(shifted) Harrison symmetry.

Since $D_q=0$ for all $q\geq 3$,
equation (\ref{main_hom_pois_struc_eq}) is only non trivial 
for $p_1\leq 3$. For $p_1=1$  it is  
the 'square zero condition' on the differential,
for $p_1=2$ it is the requirement, that
$d$ is a derivation with respect to the exterior product and for
$p_1=3$ it is the associativity law. All this follows from the differential
graded structure on $\mathcal{P}ois(A,\mathfrak{g},\omega)$.
\end{proof}
\subsection{The Poisson Bracket}\label{subsec_binary}
We refine the symplectic Poisson bracket 
to the general n-plectic setting. This bracket has all the expected properties,
except that for $n\geq 2$,
neither the Jacobi nor the Leibniz identity is satisfied strictly, but
in terms of a \textit{homotopy} Poisson-n algebra. 

The following approach was inspired by the work of 
Forger and R\"omer \cite{FPR3}, but works without the additional coboundary
condition on the $n$-plectic cocycle. In contrast to their bracket,
this one interacts accurate with the exterior derivative.
In addition it can be seen as some kind of "Poisson-theoretic" refinement 
of the $n$-plectic homotopy \textit{Lie} bracket as given by Rogers in \cite{CR}.
\newpage
\begin{definition}\label{D_2}
Let $\left(A,\mathfrak{g}, \omega\right)$ be an n-plectic structure and $\mathcal{P}ois\left(A,\mathfrak{g},\omega\right)$ the set of 
Poisson cotensors. The map
\begin{equation}
\{\cdot,\cdot\}:  \mathcal{P}ois\left(A,\mathfrak{g},\omega\right)
 \times \mathcal{P}ois\left(A,\mathfrak{g},\omega\right) \to 
  \mathcal{P}ois\left(A,\mathfrak{g},\omega\right)\;, \\
\end{equation}
defined for any homogeneous Poisson cotensors 
$f_1, f_2\in \mathcal{P}ois\left(A,\mathfrak{g},\omega\right)$ and 
associated Hamilton tensors 
$x_1, x_2\in\mathcal{H}am(A,\mathfrak{g},\omega)$ by  
the equation
\begin{equation}
\{f_1,f_2\}=-L_{x_{1}}f_2 +(-1)^{(|x_1|-1)(|x_2|-1)}L_{x_{2}}f_1
\end{equation}  
and then extended to all of 
$\mathcal{P}ois\left(A,\mathfrak{g},\omega\right)$ by linearity, is called the
\textbf{homotopy Poisson 2-bracket}.
\end{definition}
We propose the '\textit{homotopy}' modifier since theorem
(\ref{jacobi_id}) and (\ref{first_leibniz_theorem}) show that neither the Jacobi nor
the Leibniz identity hold 'strictly' but only 'up to higher homotopies' 
as we will see in the next sections. The bracket is the 
antisymmetric incarnation of the one I developed in \cite{MR1}. It interacts 
with the differential and the exterior product in terms of a so called 
\textit{homotopy Poisson}-$n$ algebra (\ref{main_hom_pois_proof}), 
which justifies the name.
\begin{remark}\label{relation_to_Rogers}
To highlight the relation to the n-plectic
homotopy Lie bracket as defined by Rogers in \cite{CR}, observe the identity
\begin{equation}\label{old_2-bracket}
\textstyle\frac{1}{2}\{f_1,f_2\}= 
\sgn{|x_1|}\cdot i_{x_2\wedge x_1}\omega
- \frac{1}{2}d\left(i_{x_1}f_2 - (-1)^{(|x_1|-1)(|x_2|-1)}i_{x_2}f_1\right)\;.
\end{equation}
The additional cocycle can be seen as some sort of Poisson-theoretic
correction term, necessary for a correct interaction with the differential graded
commutative structure.
\end{remark}
\begin{remark}[Symplectic manifolds]Let $(M,\omega)$ be a 
symplectic manifold. The homotopy Poisson 
$2$-bracket is equal to the common symplectic Poisson bracket, 
usually defined by $-i_{x_1\wedge x_2}\omega$ (up to the factor $2$). 
This is a consequence of the previous remark. In particular the Jacobi
identity holds strictly for $n=1$.
\end{remark}
On the technical level, a first thing to show is, that the bracket is 
independent of the particular chosen associated Hamilton tensors. 
This is guaranteed by the kernel property (\ref{kernel_prop}). The
following proposition gives the details and 
computes associated tensors explicitly.
\begin{proposition}\label{well_def_2}Let $(A,\mathfrak{g},\omega)$ be an
$n$-plectic structure and  $f_1$, $f_2\in \mathcal{P}ois\left(A,\mathfrak{g},\omega\right)$ two Poisson cotensors. 
The image  $\{f_1,f_2\}$ is a well defined Poisson cotensor. 
If $y_1$ resp. $y_2$ and $x_1$ resp. $x_2$ 
are tensors associated to $f_1$ and $f_2$ by the equations 
$i_{x_i}\omega=df_i$ and $i_{y_i}\omega=f_i$, respectively, the exterior tensor
$$
y_{\{f_1,f_2\}}= \textstyle\left[x_2,y_1\right] -
(-1)^{(|x_1|-1)(|x_2|-1)}\left[x_1,y_2\right]
$$
is associated to the homotopy Poisson $2$-bracket by the equation 
$i_{y_{\{f_1,f_2\}}}\omega= \{f_1,f_2\}$ and the exterior tensor
\begin{equation}
x_{\{f_1,f_2\}}\defeq  [x_2,x_1]-(-1)^{(|x_1|-1)(|x_2|-1)}[x_1,x_2]
\end{equation}
is associated to the homotopy Poisson $2$-bracket by the equation 
$i_{x_{\{f_1,f_2\}}}\omega= d\{f_1,f_2\}$.
\end{proposition}
\begin{proof}
To see that $\{f_1,f_2\}$ is well defined, suppose $\xi$ is a tensor from the 
kernel of $\omega$. Then $L_{\xi}f=di_{\xi}f-(-1)^{|\,\xi\,|}i_{\xi}df=0$ since 
$f$ and $df$ have the kernel property. Therefore we get 
$L_{x+\xi}f=L_{x}f$ for any Poisson cotensor $f$ and Hamilton tensor $x$. 
By prop. (\ref{kernel_prop}) the difference of tensors, associated to 
the same Poisson cotensor is an element of the kernel of $\omega$ and 
consequently the image $\{f_1,f_2\}$ does not depend on any particular choice.

To show that $y_{\{f_1,f_2\}}$ satisfies the Poisson constraint equation
with respect to the homotopy Poisson $2$-bracket, we use $L_{x_i}\omega=0$
as well as $|x_i|=|y_i|-1$ and compute:
\begin{align*}
i_{y_{\{f_1,f_2\}}}\omega
&=\textstyle-(-1)^{(|x_1|-1)(|x_2|-1)}i_{\left[x_1,y_2\right]}\omega
	+i_{\left[x_2,y_1\right]}\omega\\
&=\textstyle-(-1)^{(|x_1|-1)(|x_2|-1)}\left(
	(-1)^{(|x_1|-1)(|x_2|-1)}L_{x_1}i_{y_2}\omega-i_{y_2}L_{x_1}\omega\right)\\
&\quad+\left((-1)^{(|x_1|-1)(|x_2|-1)}L_{x_2} i_{y_1}\omega
  -i_{y_1}L_{x_2}\omega\right)\\	
&=\textstyle -L_{x_1}i_{y_2}\omega 
  +(-1)^{(|x_1|-1)(|x_2|-1)}L_{x_2} i_{y_1}\omega\\
&=\textstyle -L_{x_1}f_2 +(-1)^{(|x_1|-1)(|x_2|-1)}L_{x_2} f_1\;.
\end{align*}
It only remains to show that the tensor $x_{\{f_1,f_2\}}$ satisfies the 
fundamental equation with respect to the homotopy Poisson $2$-bracket:
\begin{align*}
&\textstyle i_{[x_2,x_1]} \omega 
	-(-1)^{(|x_1|-1)(|x_2|-1)}i_{[x_1,x_2]} \omega\\
&=\textstyle((-1)^{(|x_2|-1)|x_1|}L_{x_2} i_{x_1} \omega
   - i_{x_1} L_{x_2} \omega ) \\
&\quad\textstyle -(-1)^{(|x_1|-1)(|x_2|-1)}((-1)^{(|x_1|-1)|x_2|}
 L_{x_1} i_{x_2} \omega
	- i_{x_2} L_{x_1} \omega ) \\
&= -\textstyle(-1)^{(|x_1|-1)(|x_2|-1)+|x_2|}
 L_{x_2} df_1 +(-1)^{|x_1|}L_{x_1}df_2\\
&=\textstyle (-1)^{(|x_1|-1)(|x_2|-1)}dL_{x_2}f_1 - dL_{x_1}f_2\\
&=\textstyle d((-L_{x_1}f_2+(-1)^{(|x_1|-1)(|x_2|-1)}L_{x_2}f_1))\\
&=d\{f_1,f_2\}\;.
\end{align*}
\end{proof}
\begin{remark}Strictly speaking it is not correct to write 
$x_{\{f_1,f_2\}}$ since associated tensors are in general not unique and
the expression depends on the particular chosen Hamiltons
 $x_i$. By abuse of notation we stick to the symbol
$x_{\{f_1,f_2\}}$ to indicate that we have at least \textit{some} 
Hamilton tensor associated
to the Poisson 2-bracket $\{f_1,f_2\}$. According to proposition (\ref{kernel_prop}) 
all representatives compute equal contractions and Lie derivations of
Poisson cotensors, therefore the risk of confusion is low.  
\end{remark}
Now lets see that the $n$-plectic Poisson $2$-bracket has the properties
expected from the bilinear bracket operator of a homotopy Poisson-$n$ algebra:
\begin{proposition}\label{jacobi-2}
The homotopy Poisson $2$-bracket is bilinear and homogeneous of degree 
$(n-1)$ with respect to the tensor grading. It satisfies the 
$(n-1)$-fold shifted antisymmetry equation
\begin{equation}
\{f_1,f_2\}=
-(-1)^{(|f_1|+n-1)(|f_2|+n-1)}
\{f_2,f_1\}
\end{equation}
for all homogeneous 
$f_1,f_2 \in\mathcal{P}ois(A,\mathfrak{g},\omega)$
and the $(n-1)$-fold shifted homotopy Jacobi equation in dimension two
\begin{equation}
d\{f_1,f_2\}=
 \{df_1,f_2\}-
  (-1)^{(|f_1|+n-1)(|f_2|+n-1)}
    \{df_2,f_1\}\;.
\end{equation}
\end{proposition}
\begin{proof}
Bilinearity is a straight forward implication of the definition, since all 
relevant operators are linear in all arguments. We assume that 
$f_1,f_2\in \mathcal{P}ois(A,\mathfrak{g},\omega)$ are homogeneous.
Since $|f|=|x|-n$, homogeneity can be computed by
\begin{align*}
|\{f_1,f_2\}|
&=|L_{x_1}f_2|\\
&=|x_1|+|f_2|-1\\
&=|f_1|+|f_2|+n-1\;.
\end{align*}
Similar, the $(n-1)$-fold shifted antisymmetry is a consequence of $|f|=|x|-n$ 
and can be computed according to
\begin{align*}
&\{f_1,f_2\}\\ 
&= -L_{x_1}f_2 +(-1)^{(|x_1|-1)(|x_2|-1)}L_{x_2}f_1\\
&= -(-1)^{(|x_1|-1)(|x_2|-1)}
 (-L_{x_2}f_1 +(-1)^{(|x_1|-1)(|x_2|-1)}L_{x_1}f_2)\\
&=-(-1)^{(|x_1|-1)(|x_2|-1)}\{f_2,f_1\}\\
&=-(-1)^{(|f_1|+n-1)(|f_2|+n-1)}\{f_2,f_1\}.
\end{align*}
Finally we compute the $(n-1)$-fold shifted homotopy Jacobi equation in 
dimension two. Since any Hamilton tensor associated to the cocycle $df$ is 
an element from the kernel of $\omega$, we get 
\begin{align*}
d\{f_1,f_2\}
&=-dL_{x_1}f_2 +(-1)^{(|x_1|-1)(|x_2|-1)} dL_{x_2}f_1\\
&= -(-1)^{|x_1|-1}L_{x_1}df_2 
 -(-1)^{(|x_1|-1)(|x_2|-1)+|x_2|} L_{x_2}df_1\\
&= +(-1)^{(|x_1|-2)(|x_2|-1)}L_{x_2}df_1
  -(-1)^{|x_1|-1}L_{x_1}df_2\\	
&= \{df_1,f_2\}
 - (-1)^{(|x_1|-1)(|x_2|-1)}\{df_2,f_1\}\\	
&= \{df_1,f_2\}
 - (-1)^{(|f_1|+n-1)(|f_2|+n-1)}\{df_2,f_1\}\;.
\end{align*}
\end{proof}
For $n\geq 2$, the $n$-plectic Poisson bracket does not satisfy the 
usual Jacobi identity. 
Still this 'failure' is not random, but controlled
by a certain cocycle, that emerge from an 
additional $3$-ary bracket, as we will see in the next section. 
The corrected Jacobi identity is equation (\ref{jacoby_dim_4}). 
It is first evidence for the appearance of 
a homotopy algebra in the $n$-plectic formalism.
\begin{theorem}[Jacobi Identity]\label{jacobi_id} 
The Poisson $2$-bracket does not satisfy the Jacobi identity of a graded Lie 
algebra. Instead
\begin{multline}
\textstyle \sum_{\sigma\in Sh\left(2,1\right)}
 \sgn{\sigma}e\left(\sigma;sx_1,sx_2,sx_3\right)
  \{\{f_{\sigma(1)},f_{\sigma(2)}\},f_{\sigma(3)}\}\\
=-\textstyle\sum_{\sigma\in Sh\left(2,1\right)}\sgn{\sigma}
 e\left(\sigma;sx_1,sx_2,sx_3\right)
  L_{[x_{\sigma(2)},x_{\sigma(1)}]}f_{\sigma(3)}
\end{multline}
for any homogeneous Poisson cotensors 
$f_1,f_2,f_3 \in \mathcal{P}ois(A,\mathfrak{g},\omega)$ and associated
Hamilton tensors $x_1,x_2$ and $x_3$, respectively.
\end{theorem} 
\begin{proof} To show that the equation is satisfied,
apply the definition of $\{\cdot,\cdot\}$ to
rewrite the left side into
\begin{multline*}
\textstyle
 -\sum_{\sigma\in Sh(2,1)}\sgn{\sigma}
   e\left(\sigma;sx_1,sx_2,sx_3\right)
   L_{2\cdot [x_{\sigma(2)},x_{\sigma(1)}]}f_{\sigma(3)}\\
+\textstyle\sum_{\sigma\in Sh(2,1)}\sgn{\sigma}
 e\left(\sigma;sx_1,sx_2,sx_3\right)
  (-1)^{(|[x_{\sigma(2)},x_{\sigma(1)}]|+1)(|x_{\sigma(3)}|+1)}L_{x_{\sigma(3)}}\\
\textstyle\left(-L_{x_{\sigma(1)}}f_{\sigma(2)} 
 +(-1)^{(|x_{\sigma(1)}|+1)(|x_{\sigma(2)}|+1)}
  L_{x_{\sigma(2)}}f_{\sigma(1)}\right)\;.
\end{multline*}
Since all terms in this expression are graded antisymmetric, we can simplify
and reorder to arrive at: 
\begin{multline*}
-\textstyle2\sum_{\sigma\in Sh(2,1)}\sgn{\sigma}
 e\left(\sigma;sx_1,sx_2,sx_3\right)
  L_{[x_{\sigma(2)},x_{\sigma(1)}]}f_{\sigma(3)}\\
+\textstyle\sum_{\sigma\in Sh(2,1)}\sgn{\sigma}
 e\left(\sigma;sx_1,sx_2,sx_3\right)\cdot{}\phantom{mmmmmmmm}\\
  \cdot((-1)^{(|x_{\sigma(1)}|+1)(|x_{\sigma(2)}|+1)}
    L_{x_{\sigma(2)}}L_{x_{\sigma(1)}}f_{\sigma(3)} 
     -L_{x_{\sigma(1)}}L_{x_{\sigma(2)}}f_{\sigma(3)})	\;.
\end{multline*}
Using (\ref{multi_rules-2}) the second shuffle sum can be rewritten in terms of 
the Schouten bracket to get the expression
\begin{multline*}
-\textstyle2\sum_{\sigma\in Sh(2,1)}\sgn{\sigma}
 e\left(\sigma;sx_1,sx_2,sx_3\right)
  L_{[x_{\sigma(2)},x_{\sigma(1)}]}f_{\sigma(3)}\\
+\textstyle\sum_{\sigma\in Sh(2,1)}\sgn{\sigma}
 e\left(\sigma;sx_1,sx_2,sx_3\right)L_{[x_{\sigma(2)},x_{\sigma(1)}]}f_{\sigma(3)}	
\end{multline*}
which proof the theorem on homogeneous and therefore on arbitrary Poisson cotensors.
\end{proof}
In contrast to the $n$-plectic Poisson bracket,
the Schouten-Nijenhuis bracket does satisfy the
usual Jacobi equation. From this follows, that the $n$-plectic 
Jacobi identity, although not trivial, is nevertheless a cocycle: 
\begin{theorem}\label{J2_mvf}Let $(A,\mathfrak{g},\omega)$ be an $n$-plectic
structure. The Jacobi expression
$$
\textstyle\sum_{\sigma\in Sh(2,1)}\sgn{\sigma}
 e\left(\sigma;sx_1,sx_2,sx_3\right)
\{\{f_{\sigma(1)},f_{\sigma(2)}\},f_{\sigma(3)}\}
$$
is a cocycle for any three Poisson cotensors
$f_1,f_2,f_3\in \mathcal{P}ois(A,\mathfrak{g},\omega)$.
If $y_1$, $y_2$, $y_3$ are Poisson constraints, associated to $f_1$, $f_2$ and
$f_3$, then the exterior tensor given by
\begin{equation}\label{J_2}
-\textstyle\sum_{\sigma\in Sh(1,2)}
 e\left(\sigma;sx_1,sx_2,sx_3\right)
  [[x_{\sigma(3)},x_{\sigma(2)}],y_{\sigma(1)}]
\end{equation}
is Poisson constraint associated to the Jacobi expression by 
equation (\ref{Poisson_constraint_eq}).
\end{theorem}
\begin{proof}
\begin{comment}
NEED TO SHOW THAT THE POR IS POISSON.
To see that (\ref{J_2}) is a Poisson tensor associated to the Jacobi expression,
recall that since 
$[x_i,x_j]\in \mathcal{S}ym(A,\mathfrak{g},\omega)$, the Lie derivative $L_{[X_i,X_j]}\omega$ vanishes and we get
\begin{align*}
&-\textstyle\sum_{\sigma\in Sh\left(1,2\right)}
 \sgn{\sigma}e\left(\sigma;sx_1,sx_2,sx_3\right)
  i_{[[x_{\sigma(3)},x_{\sigma(2)}],y_{\sigma(1)}]}\omega\\
&=-\textstyle\sum_{\sigma\in Sh\left(1,2\right)}
 \sgn{\sigma}e\left(\sigma;sx_1,sx_2,sx_3\right)
  (-1)^{(|[x_{\sigma(3)},x_{\sigma(2)}]|+1)|y_{\sigma(1)}|}
   L_{[x_{\sigma(3)},x_{\sigma(2)}]}i_{y_{\sigma(1)}}\omega\\
&=-\textstyle\sum_{\sigma\in Sh\left(1,2\right)}
 \sgn{\sigma}e\left(\sigma;sx_1,sx_2,sx_3\right)
  (-1)^{(|(x_{\sigma(3)}|+1)(|x_{\sigma(1)}|+1)+
   (|x_{\sigma(2)}|+1)(|x_{\sigma(1)}|+1)}\;\cdot\\
&\phantom{=.}\cdot L_{[x_{\sigma(3)},x_{\sigma(2)}]}i_{y_{\sigma(1)}}\omega\\  
&=-\textstyle\sum_{\sigma\in Sh\left(2,1\right)}
 \sgn{\sigma}e\left(\sigma;sx_1,sx_2,sx_3\right)
		L_{[x_{\sigma(2)},x_{\sigma(1)}]}i_{y_{\sigma(3)}}\omega\\		
&=-\textstyle\sum_{\sigma\in Sh\left(2,1\right)}
 \sgn{\sigma}e\left(\sigma;sx_1,sx_2,sx_3\right)
  L_{[x_{\sigma(2)},x_{\sigma(1)}]}f_{\sigma(3)}\\
&=\textstyle\sum_{\sigma\in Sh\left(2,1\right)}
 \sgn{\sigma}e\left(\sigma;sx_1,sx_2,sx_3\right)
	\{\{f_{\sigma(1)},f_{\sigma(2)}\},f_{\sigma(3)}\}\; .
\end{align*}
\end{comment}
To see that the Jacobi expression is a cocycle, we use the ordinary Jacobi identity 
of the Schouten bracket for exterior tensors. In particular we compute
\begin{align*}
0&=\textstyle\sum_{\sigma\in Sh(2,1)}
 e\left(\sigma;sx_1,sx_2,sx_3\right)
  i_{[[x_{\sigma(3)},x_{s_2}],x_{s_1}]}\omega\\
&=\textstyle\sum_{\sigma\in Sh(2,1)}
 e\left(\sigma;sx_1,sx_2,sx_3\right)
  i_{[[x_{\sigma(2)},x_{\sigma(1)}],x_{\sigma(3)}]}\omega\\
&=-\textstyle \sum_{\sigma\in Sh(2,1)}
   \sgn{\sigma}e\left(\sigma;sx_1,sx_2,sx_3\right)
    \sgn{|x_{\sigma(1)}|+|x_{\sigma(2)}|} 
     L_{[x_{\sigma(2)},x_{\sigma(1)}]}i_{x_{\sigma(3)}}\omega\\		
&=-\textstyle \sum_{\sigma\in Sh(2,1)}
   \sgn{\sigma}e\left(\sigma;sx_1,sx_2,sx_3\right)
    \sgn{|x_{\sigma(1)}|+|x_{\sigma(2)}|} 
     L_{[x_{\sigma(2)},x_{\sigma(1)}]}df_{\sigma(3)}\\
&=-\textstyle \sum_{\sigma\in Sh(2,1)}
   \sgn{\sigma} e\left(\sigma;sx_1,sx_2,sx_3\right)
    dL_{[x_{\sigma(2)},x_{\sigma(1)}]}f_{\sigma(3)}\\
&=\textstyle d\left(\sum_{\sigma\in Sh(2,1)}
   \sgn{\sigma}e\left(\sigma;sx_1,sx_2,sx_3\right)
    \{\{f_{\sigma(1)},f_{\sigma(2)}\},f_{\sigma(3)}\}\right)\;.
\end{align*}
\end{proof}
The existence of the associated tensor (\ref{J_2})
is tied to the assumption that any Poisson cotensor $f$ has
a solution $y$ to the Poisson constraint equation $i_y\omega=f$. 
If we drop that assumption, the equation
\begin{equation}\label{help_1}
i_y\omega = -\textstyle\sum_{\sigma\in Sh(2,1)}
 \sgn{\sigma}e\left(\sigma;sx_1,sx_2,sx_3\right)
  \{\{f_{\sigma(1)},f_{\sigma(2)}\},f_{\sigma(3)}\}
\end{equation}
must not have a solution anymore, as the following simple counterexample shows:
\begin{example}\label{counterexample_1}
Let $\left(\R^6,\omega\right)$ be the $3$-plectic manifold, with $3$-plectic
cocycle given by
$$ 
\omega\defeq  dx^1\smwedge dx^3\smwedge dx^5 \smwedge dx^6 +
dx^2\smwedge dx^4\smwedge dx^5 \smwedge dx^6\;.
$$
and consider
$f_1\defeq  \left(x_1^2\,x_3- x_4\right)\,dx^5\smwedge dx^6$ and
$f_2\defeq  -\left(x_3+x_2^2\,x_4\right)\,dx^5\smwedge dx^6$.
These differential form are Poisson, since Poisson constraint tensor fields 
are given (for example) by
$$
\begin{array}{lcr}
y_1\defeq  \left(x_1^2\,x_3-x_4\right)\,\partial_3\smwedge \partial_1 
&,&
y_2\defeq  -\left(x_3+x_2^2\,x_4\right)\,\partial_4\smwedge \partial_2
\end{array}
$$
and Hamilton tensor fields associated by the fundamental
equation (\ref{fake_fundamental_1}) are given (for example) by
$$
\begin{array}{lcr}
x_1\defeq  x_1^2\partial_1 -\partial_2 -2x_1x_3\partial_3 
&,&
x_2\defeq  -\partial_1 -x_2^2\partial_2 +2x_2x_4\partial_4 \;.
\end{array}
$$
\begin{comment}
since 
$$
df_1=-x_1^2 dx^3\smwedge dx^5 \smwedge dx^6
-2x_1x_3\,dx^1\smwedge dx^5\smwedge dx^6 + dx^4\smwedge dx^5 \smwedge dx^6
$$
$$
df_2=x_2^2 dx^4\smwedge dx^5 \smwedge dx^6 + 2x_2x_4\,dx^2\smwedge dx^5\smwedge dx^6  + dx^3\smwedge dx^5 \smwedge dx^6
$$
\end{comment}
In this case the Schouten bracket reduces to the usual Lie bracket and is given by
$$
[x_2,x_1]=2x_1\partial_1+2x_2\partial_2-2x_3\partial_3-2x_4\partial_4\;.
$$
Next define the differential form
$f_3\defeq  dx^1 \smwedge dx^2$.
Since any closed form is a Hamilton cotensor, so is $f_3$. In contrast $f_3$ is 
not Poisson, because there can't be any tensor field $y$ satisfying
$i_y\omega=f_3$. 

To find the Jacobian 
$\textstyle\sum_{\sigma\in Sh(2,1)}
 \sgn{\sigma}e\left(\sigma;sx_1,sx_2,sx_3\right)
  \{\{f_{\sigma(1)},f_{\sigma(2)}\},f_{\sigma(3)}\}$ we compute along the line of 
the proof of equation (\ref{jacobi_id}). Since $f_3$ is a cocycle, this gives 
$L_{[x_2,x_1]}f_3$, which is $4 dx^1 \smwedge dx^2$.
It follows that the Jacobi expression is not Poisson and that equation 
(\ref{help_1}) has no solution.
\end{example} 

Again, this justifies our proposed definition of Poisson cotensors. 
If we want a trilinear bracket operator to be related to our 
Poisson $2$-bracket by the homotopy Jacobi
equation in dimension four, a solution to the equation
\begin{multline*}
i_{x_{\{f_1,f_2,f_3\}}}\omega =
-\textstyle\sum_{\sigma\in Sh(2,1)}
 \sgn{\sigma}e\left(\sigma;sx_1,sx_2,sx_3\right)
  \{\{f_{\sigma(1)},f_{\sigma(2)}\},f_{\sigma(3)}\}\\
-\textstyle\sum_{\sigma\in Sh(1,2)}
 \sgn{\sigma}e\left(\sigma;sx_1,sx_2,sx_3\right)
  \{df_{\sigma(1)},f_{\sigma(2)},f_{\sigma(3)}\}
\end{multline*}
is required and since the contraction operator is linear, this needs 
a solution of equation (\ref{help_1}).
\begin{remark}
The Leibniz equation is in general not satisfied 'strictly'
by the n-plectic Poisson bracket, but only up to 
certain corrections terms. We will look at this in detail in section 
(\ref{Leibniz_operators_sec}). The correct homotopy Leibniz
equation is then derived in theorem (\ref{first_leibniz_theorem}).
\end{remark}
\subsection{The higher Poisson brackets}\label{subsec_trinary} 
Summarizing the last two sections we have seen that the general $n$-plectic
setting has something close to a Poisson algebra, but without the 
Jacobi and the Leibniz identity. 

In this section we look in detail on \textit{how} the
Jacobi equation fails. As it turns out this "failure" is controlled by another
bracket-like operator that has three arguments and is therefore called 
the Poisson $3$-bracket.

With additional operators, however, come additional Jacobi-like 
equations and to control them we need additional operators. The general 
pattern is then an infinite series of $k$-ary brackets for arbitrary integers 
$k\in\N$, that interact in terms of a so called $(n-1)$-fold shifted homotopy 
Lie algebra.

From a practical point of view it should be noted, that in an actual 
$n$-plectic setting all these $k$-ary brackets are trivial beyond a 
certain bound $\mathcal{O}(n)$, which means that for small $n$ only 'a few' brackets 
are different from zero. Loosely speaking we can say that the more $n$ deviates 
from $1$, the more the structure deviates from being a Poisson algebra.

\begin{definition}Let $(A,\mathfrak{g},\omega)$ be an n-plectic structure and $\mathcal{P}ois(A,\mathfrak{g},\omega)$ the set of Poisson cotensors. The 
\textbf{homotopy Poisson $3$-bracket}
\begin{equation}
\{\cdot,\cdot,\cdot\}: \textstyle\bigtimes^{\,3} \mathcal{P}ois(A,\mathfrak{g},\omega) \to 
\mathcal{P}ois(A,\mathfrak{g},\omega) \\
\end{equation}
is defined for any homogeneous 
$f_1,f_2,f_3\in \mathcal{P}ois(A,\mathfrak{g},\omega)$ and associated
Hamilton tensors $x_1$, $x_2$ resp. $x_3$, by the equation 
\begin{equation*}
\{f_1,f_2,f_3\}=
 \textstyle\sum_{\sigma\in Sh(2,1)}\sgn{\sigma}
  e\left(\sigma;sx_1,sx_2,sx_3\right)
   i_{[x_{\sigma(2)},x_{\sigma(1)}]}f_{\sigma(3)}
\end{equation*} 
and is then extended to all of $\mathcal{P}ois(A,\mathfrak{g},\omega)$ by linearity. 
\end{definition}
Again the kernel property (\ref{kernel_prop}) guarantees that this definition does not depend on the actual choice of any associated tensors. Hamilton tensors and
Poisson constraints can be computed explicitly.
\begin{theorem}\label{homotopy_3_bracket}
The homotopy Poisson $3$-bracket is trilinear, well defined and homogeneous of 
degree $(2n-1)$. It is $(n-1)$-fold shifted graded antisymmetric and 
for any three homogeneous Poisson cotensors 
$f_1$, $f_2$, $f_3 \in \mathcal{P}ois(A,\mathfrak{g},\omega)$ 
the $(n-1)$-fold shifted homotopy Jacobi equation in dimension four
\begin{multline}\label{jacoby_dim_4}
d\{f_1,f_2,f_3\}+
\textstyle\sum_{\sigma\in Sh(1,2)}\sgn{\sigma}
 e(\sigma;sx_1,sx_2,sx_3)
  \{df_{\sigma(1)},f_{\sigma(2)},f_{\sigma(3)}\}\\
+\textstyle\sum_{\sigma\in Sh(2,1)}\sgn{\sigma}
 e(\sigma;sx_1,sx_2,sx_3)
  \{\{f_{\sigma(1)},f_{\sigma(2)}\},f_{\sigma(3)}\}=0
\end{multline}
is satisfied. If $y_1$, $y_2$, $y_3$  are Poisson constraints and
$x_1$, $x_2$ and $x_3$ are Hamilton tensors, associated to 
$f_1$, $f_2$ and $f_3$ respectively,
the tensor
\begin{equation}\label{ass_Ham_3}
y_{\{f_1,f_2,f_3\}}\defeq 
 \textstyle\sum_{\sigma\in Sh(2,1)}\sgn{\sigma}
  e(\sigma;sx_1,sx_2,sx_3) 
   y_{\sigma(3)}\smwedge [x_{\sigma(2)},x_{\sigma(1)}]
\end{equation}
is a Poisson constraints associated to the homotopy Poisson $3$-bracket. In 
addition an associated Hamilton tensor is given by
\begin{multline*}\label{ass_sHam_3}
x_{\{f_1,f_2,f_3\}}\defeq 
\textstyle\sum_{\sigma\in Sh(2,1)}\sgn{\sigma}
 e(\sigma;sx_1,sx_2,sx_3)
  [[x_{\sigma(3)},x_{\sigma(2)}],y_{\sigma(1)}]\\
-\textstyle\sum_{\sigma\in Sh(2,1)}\sgn{\sigma}
 e(\sigma;sx_1,sx_2,sx_3)\sgn{|x_{\sigma(1)}|+|x_{\sigma(2)}|} 
   x_{\sigma(3)}\smwedge [x_{\sigma(2)},x_{\sigma(1)}]\;.
\end{multline*}
\end{theorem}
\begin{proof}
To see that the definition does not depend on the particular 
chosen associated Hamilton tensors, we use (\ref{kernel_prop}) and proceed 
as in the proof of (\ref{well_def_2}). 

If we assume that $f_1$ , $f_2$ and $f_3$ are homogeneous, we can see the 
homogeneity of the Poisson $3$-bracket, for example by
\begin{align*}
|\{f_1,f_2,f_3\}|
&= |i_{[x_2,x_1]}f_3|\\
&=|x_1|+|x_2|+|f_3|-1\\
&= |f_1|+|f_2|+|f_3|-1+2n\;.
\end{align*}

Since $[x_i,x_j]=-(-1)^{(|f_i|+n-1)(|f_j|+n-1)}[x_j,x_i]$,
the $(n-1)$-fold shifted graded antisymmetry follows directly from the
definition of the Poisson $3$-bracket.

To compute the homotopy Jacobi equation in dimension four, 
we apply the definition of the differential and the trinary
bracket to rewrite the mostleft term of the identity 
\begin{equation*}
d\{f_1,f_2,f_3\}=
 \textstyle \sum_{\sigma\in Sh(2,1)}
 \sgn{\sigma}e(\sigma;sx_1,sx_2,sx_3)
  di_{[x_{\sigma(2)},x_{\sigma(1)}]}f_{\sigma(3)}
\end{equation*}
and insert appropriate correction terms, using $|[x_i,x_j]|=|x_i|+|x_j|+1$. 
This leads to
\begin{multline*}
\textstyle \sum_{\sigma\in Sh(2,1)}
 \sgn{\sigma}e(\sigma;sx_1,sx_2,sx_3)\;\cdot\\
 \cdot (di_{[x_{\sigma(2)},x_{\sigma(1)}]}f_{\sigma(3)}
  +(-1)^{|x_{\sigma(1)}|+|x_{\sigma(2)}|}
   i_{[x_{\sigma(2)},x_{\sigma(1)}]}df_{\sigma(3)})\\
-\textstyle \sum_{\sigma\in Sh(2,1)}
 \sgn{\sigma}e(\sigma;sx_1,sx_2,sx_3)
  (-1)^{|x_{\sigma(1)}|+|x_{\sigma(2)}|}
   i_{[x_{\sigma(2)},x_{\sigma(1)}]}df_{\sigma(3)}\;.
\end{multline*}
Now we can rewrite the first shuffle sum, using Cartans graded infinitesimal
omotopy formular, into a sum over Lie derivations: 
\begin{multline*}
\textstyle \sum_{\sigma\in Sh(2,1)}
 \sgn{\sigma}e(\sigma;sx_1,sx_2,sx_3)L_{[x_{\sigma(2)},x_{\sigma(1)}]}f_{\sigma(3)}\\
-\textstyle \sum_{\sigma\in Sh(2,1)}
\sgn{\sigma} e(\sigma;sx_1,sx_2,sx_3)
  (-1)^{|x_{\sigma(1)}|+|x_{\sigma(2)}|}
   i_{[x_{\sigma(2)},x_{\sigma(1)}]}df_{\sigma(3)}\;.
\end{multline*}
According to (\ref{jacobi_id}) the first shuffle sum is just the negative Jacobi
expression. Since any Hamilton tensor associated to $df_{\sigma(3)}$ must be an
element of the kernel of $\omega$, we can rewrite the last expression to 
arrive at the right side of the equation:
\begin{multline*}
-\textstyle \sum_{\sigma\in Sh(2,1)}\sgn{\sigma}
  e(\sigma;sx_1,sx_2,sx_3)
  \{\{f_{\sigma(1)},f_{\sigma(2)}\},f_{\sigma(3)}\}\\
   -\textstyle \sum_{\sigma\in Sh(2,1)}\sgn{\sigma} e(\sigma;sx_1,sx_2,sx_3)
  (-1)^{|x_{\sigma(1)}|+|x_{\sigma(2)}|}
   \{f_{\sigma(1)},f_{\sigma(2)},df_{\sigma(3)}\}=
\end{multline*}
\begin{multline*}
-\textstyle \sum_{\sigma\in Sh(2,1)}
  \sgn{\sigma}e(\sigma;sx_1,sx_2,sx_3)
  \{\{f_{\sigma(1)},f_{\sigma(2)}\},f_{\sigma(3)}\}\\
   -\textstyle \sum_{\sigma\in Sh(1,2)} \sgn{\sigma}e(\sigma;sx_1,sx_2,sx_3)
   \{df_{\sigma(1)},f_{\sigma(2)},f_{\sigma(3)}\}\;.
\end{multline*}

To see that $y_{\{f_1,f_2,f_3\}}$ is a Poisson constraint associated 
to the trinary bracket, 
apply the contraction of $\omega$ along $y_{\{f_1,f_2,f_3\}}$ using 
$i_{y_i}\omega=f_i$. 

To see that $x_{\{f_1,f_2,f_3\}}$ is a Hamilton tensor
associated to $\{f_1,f_2,f_3\}$ use (\ref{J_2}) to compute:
\begin{align*}
i_{x_{\{f_1,f_2,f_3\}}}\omega
&=-\textstyle \sum_{s\in Sh\left(2,1\right)}\sgn{\sigma}e(s;sx_1,sx_2, sx_3)
  \{\{f_{\sigma(1)},f_{\sigma(2)}\},f_{\sigma(3)}\}\\
&\quad-\textstyle\sum_{s\in Sh\left(2,1\right)}\sgn{\sigma}
 e(s;sx_1,sx_2, sx_3)\sgn{|x_{\sigma(1)}|+|x_{\sigma(2)}|} 
  i_{x_{\sigma(3)}\wedge [x_{\sigma(2)},x_{\sigma(1)}]}\omega\\
&=-\textstyle \sum_{s\in Sh\left(2,1\right)}\sgn{\sigma}e(s;sx_1,sx_2, sx_3)
  \{\{f_{\sigma(1)},f_{\sigma(2)}\},f_{\sigma(3)}\}\\
&\quad-\textstyle\sum_{s\in Sh\left(1,2\right)}\sgn{\sigma}
e(s;sx_1,sx_2, sx_3)\sgn{|x_{\sigma(1)}|+|x_{\sigma(2)}|} 
  i_{[x_{\sigma(2)},x_{\sigma(1)}]}df_{\sigma(3)}\\
&=-\textstyle\sum_{s\in Sh\left(2,1\right)}\sgn{\sigma}e(s;sx_1,sx_2, sx_3)
					\{\{f_{\sigma(1)},f_{\sigma(2)}\},f_{\sigma(3)}\}\\
&\quad-\textstyle\sum_{s\in Sh\left(1,2\right)}\sgn{\sigma}e(s;sx_1,sx_2, sx_3)
			\{df_{\sigma(1)},f_{\sigma(2)},f_{\sigma(3)}\}\\
&=d\{f_1,f_2,f_3\} \; .			
\end{align*}
\end{proof}
\begin{remark}At this point we should stress again, that if there is 
\textit{no} tensor $y$ which satisfies 
$$i_y\omega = -\textstyle\sum_{s\in Sh\left(2,1\right)}
 \sgn{\sigma}e\left(s;sx_1,sx_2,sx_3\right)
	\{\{f_{\sigma(1)},f_{\sigma(2)}\},f_{\sigma(3)}\}\; ,
$$
then the previous proof shows, that the image 
$\{f_1,f_2,f_3\}$ must not be Hamilton. 
Regarding example (\ref{counterexample_1}) this justifies our definition of Poisson 
cotensors as exterior cotensors satisfying both the fundamental equation
and its Poisson constraint.
\end{remark}

With the homotopy Poisson $3$-bracket at hand,
all higher Poisson brackets are now defined inductively in terms of 
contractions along Hamilton tensors of the previously defined bracket:
\begin{definition}\label{niceBracket}
Let $(A,\mathfrak{g},\omega)$ be an n-plectic structure and $\mathcal{P}ois(A,\mathfrak{g},\omega)$ the set of Poisson cotensors. The 
\textbf{homotopy Poisson $k$-bracket}
\begin{equation}
\{\cdot,\cdots,\cdot\}: \textstyle\bigtimes^{\,k} \mathcal{P}ois(A,\mathfrak{g},\omega) \to 
\mathcal{P}ois(A,\mathfrak{g},\omega) \\
\end{equation}
is defined inductively for any $k>3$, homogeneous $f_1,\ldots,f_k\in \mathcal{P}ois(A,\mathfrak{g},\omega)$ 
and Hamilton tensor $x_{\{\cdot,\cdots,\cdot\}}$ 
associated to the homotopy Poisson $(k-1)$-bracket by  
\begin{equation*}
\textstyle\{f_1,\ldots ,f_{k}\}\defeq 
 \sgn{k-1}\sum_{\sigma\in Sh(k-1,1)}\sgn{\sigma}e(\sigma;sx_1,\ldots,sx_{k})
  i_{x_{\{f_{\sigma(1)},\cdots,f_{\sigma(k-1)}\}}}f_{\sigma(k)}
\end{equation*}
and is then extended to all of $\mathcal{P}ois(A,\mathfrak{g},\omega)$ by linearity. 
\end{definition}
The induction base is the homotopy Poisson $3$-bracket. If we refer to the 
Poisson $k$-bracket for arbitrary $k\in\N$, the differential $d$ is usually 
meant to be  the homotopy Poisson "$1$-bracket" and in this context 
sometimes written as $\{\cdot\}$.

The following theorem basically says, that the infinite
sequence of homotopy Poisson $k$-brackets defines an $(n-1)$-fold shifted
homotopy Lie algebra on Poisson cotensors.
\begin{theorem}\label{main_theorem_1}For any $k\in\N$,
the homotopy Poisson $k$-bracket is well defined, $(n-1)$-fold shifted
graded antisymmetric, homogeneous of degree $(k-1)n-1$
and the $(n-1)$-fold shifted homotopy Jacobi equation 
\begin{multline*}
\textstyle\sum_{j=1}^k\sum_{\sigma\in Sh(j,k-j)}\sgn{\sigma+j(k-j)}
e(\sigma;sx_1,\ldots,sx_k)\cdot{}\\
\cdot\{\{f_{\sigma(1)},\ldots, f_{\sigma(j)}\},
 f_{\sigma(j+1)},\ldots, f_{\sigma(k)}\}=0
\end{multline*}
is satisfied for any homogeneous 
$f_1,\ldots,f_k \in \mathcal{P}ois(A,\mathfrak{g},\omega)$. 
If $y_1,\ldots,y_k$ are Poisson constraints, associated to $f_1,\ldots,f_k$,
respectively, a Poisson constraint 
of their $k$-ary bracket is given by
\begin{multline}\label{ass_Ham_k}
y_{\{f_1,\ldots,f_k\}}\defeq \\
-\textstyle\sum_{\sigma\in Sh(k-1,1)}\sgn{\sigma+k}e(\sigma;y_1,\ldots,y_{k})\;
  y_{\sigma(k)}\smwedge x_{\{f_{\sigma(1)},\ldots,f_{\sigma(k-1)}\}}\;.
\end{multline}
If $x_1,\ldots,x_k$ are Hamilton tensors, associated to $f_1,\ldots,f_k$,
respectively, a Hamilton tensor of the $k$-ary bracket is given by
\begin{multline}\label{ass_sHam_k}
x_{\{f_1,\ldots,f_k\}}\defeq y_{J_k}(f_{\sigma(1)},\ldots,f_{\sigma(k)})+{}\\
\textstyle\sum_{\sigma\in Sh(k-1,1)}
 \sgn{\sigma}e\left(\sigma;sx_1,\ldots,sx_{k}\right)\\
  \sgn{\sum_{i=1}^{k-1}(|x_{\sigma(i)}|-1)}
   x_{\sigma(k)}\smwedge x_{\{f_{\sigma(1)},\ldots,f_{\sigma(k-1)}\}}\;,
\end{multline}
where the 'higher Jacobi' tensor 
$y_{J_k}(f_{\sigma(1)},\ldots,f_{\sigma(k)})$ is a Poisson constraint, 
defined by the equation
\begin{multline}\label{higher_Jacobi}
i_{y_{J_k}(f_{\sigma(1)},\ldots,f_{\sigma(k)})}\omega = \\
-\textstyle\sum_{j=2}^{k-1}\sum_{\sigma\in Sh(j,k-j)}\sgn{\sigma+j(k-j)}
e(\sigma;sx_1,\ldots,sx_k)\cdot{}\\
\phantom{=.}\cdot\{\{f_{\sigma(1)},\ldots, f_{\sigma(j)}\},
 f_{\sigma(j+1)},\ldots, f_{\sigma(k)}\}\;.
\end{multline}
\end{theorem}
\begin{proof}(By induction)
For $k\leq 3$ this was shown previously. For the induction step assume
that all statements of the theorem are true for some $k\in\N$. We proof that they 
are true for $(k+1)$:

First of all lets see that the definition does not depend on the particular chosen
associated Hamilton tensor. This follows from proposition
(\ref{kernel_prop}). Since the difference of tensors associated to the 
same Poisson cotensor differ only in elements of the kernel of $\omega$ we can 
find a $\xi\in\ker(\omega)$ with
$
i_{x'_{\{f_{1},\ldots,f_{k}\}}}f_{k+1}=
i_{x_{\{f_{1},\ldots,f_{k}\}}+\xi}f_{k+1}=
i_{x_{\{f_{1},\ldots,f_{k}\}}}f_{k+1}
$ because each $f_i$ has the kernel property. 

To see the $(n-1)$-fold shifted graded antisymmetry, we use the assumed 
$(n-1)$-fold shifted graded antisymmetry of any associated 
Hamilton tensor $x_{\{f_1,\ldots,f_k\}}$ (up to elements of the kernel of 
$\omega$) and rewrite the definition in terms of the symmetric group
$$
-\textstyle\frac{1}{(k-1)!}\sum_{\sigma\in S_{k}}\sgn{\sigma+k}
e(\sigma;sx_1,\ldots,sx_k) 
 i_{x_{\{f_{\sigma(1)},\ldots,f_{\sigma(k-1)}\}}}f_{\sigma(k)}\;.
$$
This expression is $(n-1)$-fold shifted graded antisymmetric, since 
$|sx_i|=|s^{n-1}f_i|$.

To see the homogeneity, assume that every 
argument is homogeneous. Then 
\begin{align*}
|\{f_1,\ldots,f_{k}\}|= 
&|i_{x_{\{f_1,\ldots,f_{k-1}\}}}f_{k}|=\\
&|x_{\{f_1,\ldots,f_{k-1}\}}|+|f_{k}|=\\
&|x_1|+\cdots+ |x_{k-1}| + |f_{k}|-1=\\
&|f_1|+\cdots+|f_{k}|-1+(k-1)n\;.
\end{align*}
The proof of the $(n-1)$-fold shifted homotopy Jacobi equation is a very long 
calculations. According to a better readable text, we compute it in section 
(\ref{homotopy_Lie_algebra}).

To see that (\ref{ass_Ham_k}) is a Poisson constraint associated 
to $\{f_1,\ldots,f_k\}$ just compute the contraction
$i_{y_{\{f_1,\ldots,f_k\}}}\omega$.

To see that (\ref{ass_sHam_k}) is a Hamilton tensor associated 
to $\{f_1,\ldots,f_k\}$, first observe that equation (\ref{higher_Jacobi})
always has a solution $y_{J_k}(f_{\sigma(1)},\ldots,f_{\sigma(k)})$, since
the right side of the equation is a Poisson tensor by the induction hypothesis.
Using this and the assumption that the homotopy Jacobi equation holds,
we compute the contraction
 
\begin{align*}
&i_{x_{\{f_1,\ldots,f_k\}}}\omega\\
=&\begin{multlined}[t][\mylength]
 i_{y_{J_k}(f_{\sigma(1)},\ldots,f_{\sigma(k)})}\omega
 +\textstyle\sum_{\sigma\in Sh(k-1,1)}
 \sgn{\sigma}e\left(s;sx_1,\ldots,sx_{k}\right)\\
  \sgn{\sum_{i=1}^{k-1}(|x_{\sigma(i)}|-1)}
  i_{x_{\sigma(k)}\smwedge x_{\{f_{\sigma(1)},\ldots,f_{\sigma(k-1)}\}}}\omega
\end{multlined}\\
=&\!\begin{multlined}[t][\mylength]
-\textstyle\sum_{j=2}^{k-1}\sum_{\sigma\in Sh(j,k-j)}\sgn{\sigma+j(k-j)}
e(\sigma;sx_1,\ldots,sx_k)\\
\{\{f_{\sigma(1)},\ldots, f_{\sigma(j)}\},
 f_{\sigma(j+1)},\ldots, f_{\sigma(k)}\}
\end{multlined}\\
&+\textstyle\sum_{\sigma\in Sh(k-1,1)}
 \sgn{\sigma}e\left(s;sx_1,\ldots,sx_{k}\right)
  \sgn{\sum_{i=1}^{k-1}(|x_{\sigma(i)}|-1)}
  i_{x_{\{f_{\sigma(1)},\ldots,f_{\sigma(k-1)}\}}}df_{\sigma(k)}\\
=&\!\begin{multlined}[t][\mylength]
-\textstyle\sum_{j=2}^{k-1}\sum_{\sigma\in Sh(j,k-j)}\sgn{\sigma+j(k-j)}
e(\sigma;sx_1,\ldots,sx_k)\\
\{\{f_{\sigma(1)},\ldots, f_{\sigma(j)}\},
 f_{\sigma(j+1)},\ldots, f_{\sigma(k)}\}
\end{multlined}\\
&+\textstyle\sum_{\sigma\in Sh(k-1,1)}
 \sgn{\sigma}e\left(s;sx_1,\ldots,sx_{k}\right)
  \sgn{\sum_{i=1}^{k-1}(|x_{\sigma(i)}|-1)}
  \{f_{\sigma(1)},\ldots,f_{\sigma(k-1)},df_{\sigma(k)}\}\\     
=&\!\begin{multlined}[t][\mylength]
-\textstyle\sum_{j=2}^{k-1}\sum_{\sigma\in Sh(j,k-j)}\sgn{\sigma+j(k-j)}
e(\sigma;sx_1,\ldots,sx_k)\\
\{\{f_{\sigma(1)},\ldots, f_{\sigma(j)}\},
 f_{\sigma(j+1)},\ldots, f_{\sigma(k)}\}
\end{multlined}\\
&-\textstyle\sum_{\sigma\in Sh(1,k-1)}\sgn{\sigma+(k-1)}
e(\sigma;sx_1,\ldots,sx_k)
\{df_{\sigma(1)},f_{\sigma(2)},\ldots, f_{\sigma(k)}\}\\
=&\!\begin{multlined}[t][\mylength]
-\textstyle\sum_{j=1}^{k-1}\sum_{\sigma\in Sh(j,k-j)}\sgn{\sigma+j(k-j)}
e(\sigma;sx_1,\ldots,sx_k)\\
\{\{f_{\sigma(1)},\ldots, f_{\sigma(j)}\},
 f_{\sigma(j+1)},\ldots, f_{\sigma(k)}\}
\end{multlined}\\
=&\; d\{f_1,\ldots,f_k\}\;.
\end{align*}
\end{proof}
\begin{corollary}[The shifted homotopy Lie algebra]
Let $(A,\mathfrak{g},\omega)$ be an $n$-plectic structure and 
$\mathcal{P}ois(A,\mathfrak{g},\omega)$ the
graded linear space of Poisson cotensors. The sequence 
$(\{\cdot,\ldots,\cdot\})_{k\in\N}$
of all $k$-ary brackets then defines an $(n-1)$-fold shifted homotopy 
Lie algebra on $\mathcal{P}ois(A,\mathfrak{g},\omega)$.
\end{corollary}
\subsection{The Leibniz operators}\label{Leibniz_operators_sec}
In this section we look in detail at the Poisson-Leibniz equation. 
For $n>1$, it does not hold strictly,
but only up to certain correction terms, which are controlled by what we call
the first $n$-plectic Leibniz operator. It behaves somewhat in between a product
and a bracket. 

This operator is enough to build the homotopy Poisson-$n$ structure in
dimension two, but beyond that, the question remains how the higher Poisson 
brackets interact with the commutative structure. As it turns out, this 
is controlled in a similar manner, by higher Leibniz-like equations.
 
Again these equations don't hold strictly, but up to certain correction 
terms, which are controlled by additional operators. We call them the 
higher $n$-plectic Leibniz operators. 
They complete the homotopy Poisson-$n$ structure. 

We start with the first Leibniz operator and look at all the Leibniz-equations
in dimension two. 
After that, we define the general patter for the higher Leibniz operators and
derive the additional Leibniz-like equations for arbitrary dimension.
\begin{definition}Let $(A,\mathfrak{g},\omega)$ be an n-plectic structure and 
$\mathcal{P}ois(A,\mathfrak{g},\omega)$ the set of Poisson cotensors. The 
\textbf{first n-plectic Leibniz operator} is the map
\begin{equation}
\{\,\cdot\, \|\, \cdot, \cdot \}: \textstyle\bigtimes^{\,3} \mathcal{P}ois(A,\mathfrak{g},\omega) \to 
\mathcal{P}ois(A,\mathfrak{g},\omega)\;,
\end{equation}
which is defined for any homogeneous 
$f_1,f_2$ and $f_3\in \mathcal{P}ois(A,\mathfrak{g},\omega)$ 
and associated Hamilton tensors $x_1$, $x_2$ and $x_3$
by the equation 
\begin{multline}\label{first_leibniz_op}
\{f_1\| f_2,f_3\}=
 -i_{x_{1}}(f_{2}\smwedge f_{3})
  +(-1)^{(|x_{1}|-1)(|x_{f_{2}\wedge f_{3}}|-1)}
   i_{x_{f_{2}\wedge f_{3}}}f_{1}\\
+\left(i_{x_{1}}f_{2}
 -(-1)^{(|x_{1}|-1)(|x_{2}|-1)} i_{x_{2}}f_{1}\right)\smwedge f_{3}\\
+(-1)^{|x_{1}||f_{2}|} f_{2}\smwedge\left(i_{x_{1}}f_{3}
  -(-1)^{(|x_{1}|-1)(|x_{3}|-1)} i_{x_{3}}f_{1}\right)
\end{multline}
and is then extended to all of $\mathcal{P}ois(A,\mathfrak{g},\omega)$ by linearity. 
\end{definition}
As the following theorem shows, this map is well defined, does not depend on the
actual choice of any Hamilton tensor and
interacts with the Poisson $2$-bracket in terms of a homotopy
Leibniz equation. 
\begin{theorem}\label{first_leibniz_theorem}
The first homotopy Leibniz operator 
$\{\,\cdot\, \|\, \cdot, \cdot \}$ is well defined, trilinear,
homogeneous of degree $n$ with respect to the  tensor grading and
it satisfies the symmetry equation
\begin{equation}
\{f_1\| f_2,f_3\}=
(-1)^{|f_2||f_3|}\{f_1\| f_3,f_2\}\;.
\end{equation}
for all (homogeneous) $f_1, f_2$ and $f_3$.
If $x_1$, $x_2$ and $x_3$ are associated Hamilton tensors,
the following \textbf{first n-plectic Leibniz equation} in dimension two holds:
\begin{multline}
\{f_{1},f_{2}\smwedge f_{3}\}=
\{f_{1},f_{2}\}\smwedge f_{3}
+(-1)^{(|x_{1}|-1)|f_{2}|}f_{2}\smwedge\{f_{1},f_{3}\}\\
+d\{f_{1}\|f_{2},f_{3}\}
+\{df_{1}\|f_{2},f_{3}\}
-(-1)^{|x_{1}|}\{f_{1}\|df_{2},f_{3}\}\\
-(-1)^{|f_{2}|+|x_{1}|}\{f_{1}\|f_{2},df_{3}\}\;.
\end{multline}
If $y_1$, $y_2$ resp. $y_3$ are tensors associated to $f_1$, $f_2$ resp. 
$f_3$ by the equations $i_{y_i}\omega=f_i$, the exterior tensor
\begin{multline}\label{Leibniz_1_Poisson_constraint}
y_{\{f_1\|f_2,f_3\}}=-j_{f_{2}}y_{3}\smwedge x_{1}
+(-1)^{(|x_{1}|-1)(|x_{f_{2}\wedge f_{3}}|-1)}y_{1}\smwedge x_{f_{2}\wedge f_{3}}\\
+(-1)^{|x_{1}||f_{2}|+|f_{2}|(|x_{1}|+|f_{3}|)}
 j_{(i_{x_{1}}f_{3}-(-1)^{(|x_{1}|-1)(|x_{3}|-1)}i_{x_{3}}f_{1})}y_{2}\\
+j_{(i_{x_{1}}f_{2}-(-1)^{(|x_{1}|-1)(|x_{2}|-1)}i_{x_{2}}f_{1})}y_{3}
\end{multline}
is Poisson constraint, associated to the first $n$-plectic Leibniz 
operator by the equation $i_{y_{\{f_1\|f_2,f_3\}}}\omega= \{f_1\|f_2,f_3\}$.
Moreover the exterior tensor
\begin{multline}\label{Leibniz_1_Hamilton}
x_{\{f_1\|f_2,f_3\}}\defeq y_{\{f_{1},f_{2}\wedge f_{3}\}}
 -y_{\{f_{1},f_{2}\}\wedge f_{3}}
  -(-1)^{(|x_{1}|-1)|f_{2}|}y_{f_{2}\wedge\{f_{1},f_{3}\}}\\
   -y_{\{df_{1}\|f_{2},f_{3}\}}
    +(-1)^{|x_{1}|}y_{\{f_{1}\|df_{2},f_{3}\}}
     +(-1)^{|f_{2}|+|x_{1}|}y_{\{f_{1}\|f_{2},df_{3}\}}
\end{multline}
is a Hamilton tensor associated to the first $n$-plectic Leibniz operator by the 
equation $i_{x_{\{f_1\|f_2,f_3\}}}\omega= d\{f_1\|f_2,f_3\}$.

\end{theorem}
\begin{proof}Trilinearity is immediate. To see that the definition does not depend on the particular 
chosen associated Hamilton tensors, we use (\ref{kernel_prop}) and proceed 
as in the proof of (\ref{well_def_2}). 
To compute the homogeneity, recall $|x|=|f|+n$. Then
\begin{align*}
|\{f_1\| f_2,f_3\}|
&= |i_{x_1}(f_2\smwedge f_3)|\\
&=|x_1|+|f_2|+|f_3|\\
&= |f_1|+|f_2|+|f_3|+n\;.
\end{align*}
The symmetry is a consequence of the graded symmetry of the exterior cotensor
product.

To see the $n$-plectic Leibniz equation we compute the 'strict part'
and the part that involves all the first $n$-plectic Leibniz operator separately.
For the strict part we get
\begin{align*}
-\{f_{1},f_{2}\smwedge f_{3}\}
&+\{f_{1},f_{2}\}\smwedge f_{3}
+(-1)^{(|x_{1}|-1)|f_{2}|}f_{2}\smwedge\{f_{1},f_{3}\}=\\
&+L_{x_{1}}(f_{2}\smwedge f_{3})
 -(-1)^{(|x_{1}|-1)(|x_{f_{2}\wedge f_{3}}|-1)}
  L_{x_{f_{2}\smwedge f_{3}}}f_{1}\\
&-(L_{x_{1}}f_{2}-
 (-1)^{(|x_{1}|-1)(|x_{2}|-1)}
  L_{x_{2}}f_{1})\smwedge f_{3}\\
&-(-1)^{(|x_{1}|-1)|f_{2}|}f_{2}\smwedge
 (L_{x_{1}}f_{3}
  -(-1)^{(|x_{1}|-1)(|x_{3}|-1)}L_{x_{3}}f_{1})\;.
\end{align*}
To compute the other part recall $|x_{df}|=|x_f|-1$. After a long and
tedious, but completely basic computation we get
\begin{align*}
d\{f_{1}\|f_{2},f_{3}\}
&+\{df_{1}\|f_{2},f_{3}\}
-(-1)^{|x_{1}|}\{f_{1}\|df_{2},f_{3}\}
-(-1)^{|f_{2}|+|x_{1}|}\{f_{1}\|f_{2},df_{3}\}=\\
&-L_{x_{1}}(f_{2}\smwedge f_{3})
 +(-1)^{(|x_{1}|+1)(|x_{f_{2}\wedge f_{3}}|+1)}
  L_{x_{f_{2}\wedge f_{3}}}f_{1}\\
&+(L_{x_{1}}f_{2}
 -(-1)^{(|x_{1}|+1)(|x_{2}|+1)}
  L_{x_{2}}f_{1})\smwedge f_{3}\\
&+(-1)^{(|x_{1}|-1)|f_{2}|}f_{2}\smwedge(
 L_{x_{1}}f_{3}
  -(-1)^{(|x_{1}|+1)(|x_{3}|+1)}
   L_{x_{3}}f_{1})\\
&-(-1)^{|x_{1}|+(|x_{1}|+1)(|x_{df_{2}\wedge f_{3}}|+1)}
 i_{x_{df_{2}\wedge f_{3}}}f_{1}\\
&-(-1)^{|f_{2}|+|x_{1}|+(|x_{1}|+1)(|x_{f_{2}\wedge df_{3}}|+1)}
 i_{x_{f_{2}\wedge df_{3}}}f_{1}\;.
\end{align*}
Since $|x_{df_{2}\wedge f_{3}}|= |x_{f_{2}\wedge df_{3}}|$,
it follows that the homotopy Leibniz equation holds if the term 
$$-(-1)^{|x_{1}|+(|x_{1}|+1)(|x_{df_{2}\wedge f_{3}}|+1)}
 i_{x_{df_{2}\wedge f_{3}}+(-1)^{|v_2|}x_{f_{2}\wedge df_{3}}}f_{1}\\$$
vanishes. This, however, is true since 
$x_{df_{2}\wedge f_{3}}+(-1)^{|v_2|}x_{f_{2}\wedge df_{3}}=
x_{d(f_{2}\wedge f_{3})}\in\ker(\omega)$ and $f_{1}$ has the kernel property
(\ref{kernel_prop}).

To show that the image $\{f_1,\|f_2,f_2\}$ is a Poisson cotensors, it in
enough to proof (\ref{Leibniz_1_Poisson_constraint}) and 
(\ref{Leibniz_1_Hamilton}). To see 
(\ref{Leibniz_1_Poisson_constraint}) we use 
(\ref{product_Poisson_constraint}) and compute
\begin{multline*}
i_{y_{\{f_1\|f_2,f_3\}}}\omega=-i_{j_{f_{2}}y_{3}\wedge x_{1}}\omega
 +(-1)^{(|x_{1}|-1)(|x_{f_{2}\wedge f_{3}}|-1)}
  i_{y_{1}\wedge x_{f_{2}\wedge f_{3}}}\omega\\
+(-1)^{|x_{1}||f_{2}|+|f_{2}|(|x_{1}|+|f_{3}|)}
 i_{j_{(i_{x_{1}}f_{3}-(-1)^{(|x_{1}|-1)(|x_{3}|-1)}i_{x_{3}}f_{1})}y_{2}}\omega\\
+i_{j_{(i_{x_{1}}f_{2}-(-1)^{(|x_{1}|-1)(|x_{2}|-1)}i_{x_{2}}f_{1})}y_{3}}\omega=
\end{multline*}
\begin{multline*}
-i_{x_{1}}i_{j_{f_{2}}y_{3}}\omega
 +(-1)^{(|x_{1}|-1)(|x_{f_{2}\wedge f_{3}}|-1)}
  i_{x_{f_{2}\wedge f_{3}}}i_{y_{1}}\omega\\
+(-1)^{|x_{1}||f_{2}|+|f_{2}|(|x_{1}|+|f_{3}|)}
 (i_{x_{1}}f_{3}-(-1)^{(|x_{1}|-1)(|x_{3}|-1)}i_{x_{3}}f_{1})\wedge f_{2}\\  
 +i_{j_{(i_{x_{1}}f_{2}-(-1)^{(|x_{1}|-1)(|x_{2}|-1)}i_{x_{2}}f_{1})}y_{3}}\omega=
\{f_1\|f_2,f_3\}\;.
\end{multline*}
To see (\ref{Leibniz_1_Hamilton}), observe that it is an immediate consequence 
of the $n$-pletic Leibniz equation, since all involved
terms are Poisson cotensors. To be more precise, we compute
\begin{multline*}
d\{f_{1}\|f_{2},f_{3}\}=\\
+\{f_{1},f_{2}\wedge f_{3}\}
 -\{f_{1},f_{2}\}\wedge f_{3}
  -(-1)^{(|x_{1}|-1)|f_{2}|}f_{2}\wedge\{f_{1},f_{3}\}\\
   -\{df_{1}\|f_{2},f_{3}\}
    +(-1)^{|x_{1}|}\{f_{1}\|df_{2},f_{3}\}
     +(-1)^{|f_{2}|+|x_{1}|}\{f_{1}\|f_{2},df_{3}\}=
\end{multline*}
\begin{multline*}
i_{y_{\{f_{1},f_{2}\wedge f_{3}\}}}\omega
 -i_{y_{\{f_{1},f_{2}\}\wedge f_{3}}}\omega
  -(-1)^{(|x_{1}|-1)|f_{2}|}i_{y_{f_{2}\wedge\{f_{1},f_{3}\}}}\omega\\
   -i_{y_{\{df_{1}\|f_{2},f_{3}\}}}\omega
    +(-1)^{|x_{1}|}i_{y_{\{f_{1}\|df_{2},f_{3}\}}}\omega
     +(-1)^{|f_{2}|+|x_{1}|}i_{y_{\{f_{1}\|f_{2},df_{3}\}}}\omega
\end{multline*}
\end{proof}
\begin{remark}
We say the first $n$-plectic Leibniz equation is satisfied \textit{strictly}
or 'on the nose', if all terms which contain a Leibniz operator vanish. In 
that case the equation is precisely equal to the common Leibniz equation of
a Poisson algebra. For $n=1$, this is always the case.
\end{remark}
Loosely speaking, the Leibniz equation controls the interaction 
between the Poisson bracket and the differential graded commutative structure.
As we will see in corollary (\ref{corollary_main_2}), this is exactly the 
general structure equation
(\ref{main_hom_pois_struc_eq}) for the parameters $k=2$, $p_1=1$ and $p_2=2$.

However, two more Leibniz-like equations have to hold in case $k=2$. 
They derive from (\ref{main_hom_pois_struc_eq}), for the parameters 
$p_1=1$ and $p_2=3$ as well as $p_1=2$ and $p_2=2$. The following theorem
makes this precise: 
\begin{theorem}
Let $(A,\mathfrak{g},\omega)$ be an $n$-plectic structure. Then for any
four Poisson cotensors
$f_1,\ldots,f_4\in \mathcal{P}ois(A,\mathfrak{g},\omega)$ and associated
Hamilton tensors $x_1$, $\ldots$, $x_4$, the following 
\textbf{second n-plectic Leibniz equation} in dimension two holds:
\begin{multline}
\{f_{1}\smwedge f_{2}\|f_{3},f_{4}\}
-\sgn{|x_{f_1\wedge f_2}||x_{f_3\wedge f_4}|}\{f_{3}\smwedge f_{4}\|f_{1},f_{2}\}=\\
f_{1}\smwedge\{f_{2}\|f_{3},f_{4}\}
+\sgn{|f_{2}||x_{f_3\wedge f_4}|}\{f_{1}\|f_{3},f_{4}\}\smwedge f_{2}\\
 -\sgn{|x_{f_1\wedge f_2}||x_{f_3\wedge f_4}|}f_{3}\smwedge\{f_{4}\|f_{1},f_{2}\}\\
  -\sgn{|x_{f_1\wedge f_2}|(|x_{f_3\wedge f_4}|+|f_4|)}
   \{f_{3}\|f_{1},f_{2}\}\smwedge f_{4}\;.
\end{multline}
In addition the \textbf{third n-plectic Leibniz equation} in dimension
two holds for the same arguments:
\begin{multline}
\sgn{|f_{2}||x_{1}|}f_{2}\smwedge\{f_{1}\|f_{3},f_{4}\}
-\{f_{1}\|f_{2},f_{3}\}\smwedge f_{4}=\\
-\{f_{1}\|f_{2},f_{3}\smwedge f_{4}\}
+\{f_{1}\|f_{2}\smwedge f_{3},f_{4}\}
\end{multline}
\end{theorem}
\begin{proof}
To see that both equations are satisfied, it is enough to just apply the 
definition (\ref{first_leibniz_op}) of the first Leibniz operator to all 
expressions and collect the terms. This is a
rather long, but simple computation and it is left to the reader.
\end{proof}
This is everything that needs to happen
in 'dimension two', i.e. for the parameter $k=2$ in the defining structure
equation (\ref{main_hom_pois_struc_eq}). As the following corollary
make precise, we can choose all structure maps 
(\ref{main_hom_pois_struc_maps}) beside the exterior product, the Poisson bracket
and the first Leibniz operator to be zero, in this case.
\begin{corollary}\label{corollary_main_2}
Let $(A,\mathfrak{g},\omega)$ be an $n$-plectic structure.
The main structure equation (\ref{main_hom_pois_struc_eq}) of a homotopy
Poisson-$n$ algebra is satisfied for $k=2$ and all $p_1,p_2\in\N$.
\end{corollary}
\begin{proof}
Recall the structure in dimension one, as given in the proof of corollary
(\ref{corollary_main_1}) and 
define the following additional maps of (\ref{main_hom_pois_struc_maps}) 
for $k=2$: The Poisson-$2$ bracket is shifted into
$D_{1,1}(s^{n-1}f_1\smwedge s^{n-1}f_2) \defeq s\{f_1,f_2\}$ and the
first Leibniz operator is shifted into
$D_{1,2}(s^{n-1}f_1\smwedge s^{n-2}(\overline{sf_2\otimes sf_3}))\defeq 
\sgn{|sf_2|}s\{f_1\|f_2,f_3\}$. In addition consider the permutation
of the shifted Leibniz operator
\begin{multline*}
D_{2,1}(s^{n-2}(\overline{sf_2\otimes sf_3})\smwedge s^{n-1}f_1)\defeq \\
-\sgn{|s^{n-1}f_1||s^{n-2}(\overline{sf_2\otimes sf_3})|}
D_{1,2}(s^{n-1}f_1\smwedge s^{n-2}(\overline{sf_2\otimes sf_3})\;.
\end{multline*}
All other structure maps $D_{q_1,q_2}$ are assumed to be zero.

These maps are homogeneous of degree $(2-n)$ and 
have the expected symmetry. Since most maps are just zero,
equation (\ref{main_hom_pois_struc_eq}) is only non trivial 
in the following four cases: $(p_1,p_2)\in 
\{(1,1),(1,2),(1,3),(2,2)\}$. 

The case $(1,1)$ is the requirement, that the de Rham differential is a
derivation with respect to bracket.  The cases 
$(1,2)$, $(2,2)$ and $(1,3)$ are the first, second and third $n$-plectic
Leibniz equation, respectively.
\end{proof}
Finally, lets look at the higher Leibniz equations. Basically, they control the 
interaction between the higher Poisson brackets and the commutative structure. 
However, for them to hold we need additional structure: 
The higher Leibniz operators. 

With the first Leibniz operator at hand,
these operators are defined inductively.  Unlike the higher brackets,
contractions along Hamilton tensors of the previously defined bracket, is not
enough. The following definition makes this precise: 
\begin{definition}
Let $(A,\mathfrak{g},\omega)$ be an n-plectic structure and 
$\mathcal{P}ois(A,\mathfrak{g},\omega)$ the set of Poisson cotensors. The 
$k$-th \textbf{n-plectic Leibniz operator} is the map
\begin{equation}
\{\,\cdot\, \|\, \cdot, \cdot \}: \textstyle\bigtimes^{\,k+2} \mathcal{P}ois(A,\mathfrak{g},\omega) \to 
\mathcal{P}ois(A,\mathfrak{g},\omega)\;,
\end{equation}
defined inductively for any $k>1$, 
homogeneous $f_1,\ldots,f_{k+2} \in \mathcal{P}ois(A, g, \omega)$ and
Hamilton tensor $x_{\{\cdot,\cdots,\cdot\|\cdot,\cdot\}}$ 
associated to the $(k-1)$-th Leibniz operator
by

\begin{align}
\label{higher_leibniz_op}
\{f_1,&\ldots,f_k\|f_{k+1},f_{k+2}\}\defeq\\
-&\nonumber\!\begin{multlined}[t][\mylength]
\textstyle \sum_{\sigma\in Sh(k-1,1)}\sgn{\sigma}e(\sigma;sx_{1},\ldots,sx_{k})
 \sgn{k+(|x_{v_{k+1}\wedge v_{k+2}}|-1)(|x_{\sigma(k)}|-1)}\\
 i_{x_{\{v_{\sigma(1)},\ldots,v_{\sigma(k-1)}\|v_{k+1},v_{k+2}\}}}v_{\sigma(k)}
\end{multlined}\\
+&\nonumber\!\begin{multlined}[t][\mylength]
\textstyle \sum_{j_{1}=0}^{k}\sum_{j_{2}=j_{1}}^{k}
 \sum_{\sigma\in Sh(j_{1},j_{2}-j_{1},k-j_{2})}
\sgn{\sigma}e(\sigma;sx_{1},...,sx_{k})\\
 \sgn{j_{1}+\sum_{i=1}^{k}(k-i)(|x_{\sigma(i)}|-1)+(\sum_{i=j_{2}+1}^{k}|x_{\sigma(i)}|)|v_{k+1}|}\phantom{mmmm}\\
  i_{x_{\sigma(1)}}...i_{x_{\sigma(j_{1})}}((i_{x_{\sigma(j_{1}+1)}}...i_{x_{\sigma(j_{2})}}v_{k+1})\wedge(
   i_{x_{\sigma(j_{2}+1)}}...i_{x_{\sigma(k)}}v_{k+2}))
\end{multlined}\\
-&\nonumber\!\begin{multlined}[t][\mylength]
\textstyle \sum_{j_{1}=0}^{k-1}\sum_{j_{2}=j_{1}}^{k-1}\sum_{\sigma\in Sh(j_{1},j_{2}-j_{1},1,k-1-j_{2})}(-1)^{\sigma}e(\sigma;sx_{1},...,sx_{k})\\
 (-1)^{j_{1}+\sum_{i=1}^{k}(k-i)(|x_{\sigma(i)}|-1)+(\sum_{i=j_{2}+2}^{k}|x_{\sigma(i)}|)|v_{k+1}|+(|x_{\sigma(j_{2}+1)}|+1)(|x_{k+1}|+1)}\\
 i_{x_{\sigma(1)}}...i_{x_{\sigma(j_{1})}}((i_{x_{\sigma(j_{1}+1)}}...i_{x_{\sigma(j_{2})}}
  i_{x_{k+1}}v_{\sigma(j_{2}+1)})\wedge(
   i_{x_{\sigma(j_{2}+2)}}...i_{x_{\sigma(k)}}v_{k+2}))
\end{multlined}\\
-&\nonumber\!\begin{multlined}[t][\mylength]
\textstyle \sum_{j_{1}=0}^{k-1}\sum_{j_{2}=j_{1}}^{k-1}\sum_{\sigma\in Sh(j_{1},j_{2}-j_{1},k-1-j_{2},1)}(-1)^{\sigma}e(\sigma;sx_{1},...,sx_{k})\\
(-1)^{j_{1}+\sum_{i=1}^{k}(k-i)(|x_{\sigma(i)}|-1)+(\sum_{i=j_{2}+1}^{k}|x_{\sigma(i)}|)|v_{k+1}|+(|x_{\sigma(k)}|+1)(|x_{k+2}|+1)}\\
 i_{x_{\sigma(1)}}...i_{x_{\sigma(j_{1})}}((i_{x_{\sigma(j_{1}+1)}}...
  i_{x_{\sigma(j_{2})}}v_{k+1})\wedge(i_{x_{\sigma(j_{2}+1)}}...
   i_{x_{\sigma(k-1)}}i_{x_{k+2}}v_{\sigma(k)}))
\end{multlined}\\
+&\nonumber\!\begin{multlined}[t][\mylength]
\textstyle \sum_{j_{1}=0}^{k-2}\sum_{j_{2}=j_{1}}^{k-2}\sum_{\sigma\in Sh(j_{1},j_{2}-j_{1},1,k-2-j_{2},1)}(-1)^{\sigma}e(\sigma;sx_{1},...,sx_{k})\\
\sgn{j_{1}+\sum_{i=1}^{k}(k-i)(|x_{\sigma(i)}|-1)+(\sum_{i=j_{2}+2}^{k}|x_{\sigma(i)}|)|v_{k+1}|}\\
\sgn{(|x_{\sigma(j_{2}+1)}|+1)(|x_{k+1}|+1)+(|x_{\sigma(k)}|+1)(|x_{k+2}|+1)}\\
 i_{x_{\sigma(1)}}...i_{x_{\sigma(j_{1})}}((i_{x_{\sigma(j_{1}+1)}}...i_{x_{\sigma(j_{2})}}
  i_{x_{k+1}}v_{\sigma(j_{2}+1)})\smwedge(\\
   i_{x_{\sigma(j_{2}+2)}}...i_{x_{\sigma(k-1)}}i_{x_{k+2}}v_{\sigma(k)}))
\end{multlined}
\end{align}
and is then extended to all of $\mathcal{P}ois(A,\mathfrak{g},\omega)$ by 
linearity.
\end{definition}
The induction base is the first Leibniz operator. The idea of the notation
$\{\cdot,\ldots,\cdot\|\cdot,\cdot\}$ is, that everything on the left of
$\|$ 'is like a bracket', while everything on the right 'is like a product'. 
This is meant in particular with respect to homogeneity and symmetry.
\begin{remark}
In actual computations it
is sometimes advantageous to consider the commutative product as the 
'zeroth-Leibniz operator'. In that case we follow the notational logic of the
higher Leibniz operators and write $\{\|f_1,f_2\}$ for the product 
$f_1\smwedge f_2$.
\end{remark}
The following theorem basically says, that the infinite
sequence of higher Leibniz operators is well defined and interact with the 
commutative structure and the Poisson brackets in terms of a homotopy Poisson-$n$
algebra:
\begin{theorem}\label{main_theorem_2}For any $k\in\N$,
the $k$-th Leibniz operator is well defined, $(n-1)$-fold shifted
graded antisymmetric with respect to all arguments on the left and graded
symmetric with respect to all arguments on the right. Moreover it is
homogeneous of degree $n\cdot k$ with respect to the tensor grading.

For any homogeneous Poisson cotensors
$f_1,\ldots,f_{k+2}\in \mathcal{P}ois(A,\mathfrak{g},\omega)$ 
and associated (homogeneous)
Hamilton tensor $x_1,\ldots,x_{k+2}$, the following
 \textbf{first n-plectic Leibniz equation} in dimension $(k+1)$ is satsfied:

\begin{align}
-&\nonumber\!\begin{multlined}[t][\mylength]
\{f_{1},\ldots,f_{k},f_{k+1}\smwedge f_{k+2}\}
  +\{f_{1},\ldots,f_{k},f_{k+1}\}\smwedge f_{k+2}\\
+\sgn{(\sum_{i=1}^{k}|x_{i}|-1)|f_{k+1}|}
 f_{k+1}\smwedge\{f_{1},\ldots,f_{k},f_{k+2}\}
\end{multlined}\\
+&\nonumber\!\begin{multlined}[t][\mylength]
\textstyle \sum_{\sigma\in Sh(1,k-1)}
 \sgn{\sigma}e(\sigma;sx_{1},\ldots,sx_{k})\cdot{}\\
\cdot\{df_{\sigma(1)},f_{\sigma(2)},\ldots,f_{\sigma(k)}\|f_{k+1},f_{k+2}\}\\
-\sgn{k}d\{f_{1},\ldots,f_{k}\|f_{k+1},f_{k+2}\}
\end{multlined}\\
+&\nonumber\!\begin{multlined}[t][\mylength]
\sgn{\sum_{i=1}^{k}(|x_{i}|-1)}\{f_{1},\ldots,f_{k}\|df_{k+1},f_{k+2}\}\\
+\sgn{\sum_{i=1}^{k}(|x_{i}|-1)+|f_{k+1}|}\{f_{1},\ldots,f_{k}\|f_{k+1},df_{k+2}\}
\end{multlined}\\
+&\nonumber\!\begin{multlined}[t][\mylength]
\textstyle \sum_{j=2}^{k}\sum_{\sigma\in Sh(j,k-j)}
  \sgn{\sigma}e(\sigma;sx_{1},\ldots,sx_{k})\sgn{(j+1)(k+1-j)}\cdot{}\\
\cdot \{\{f_{\sigma(1)},\ldots,f_{\sigma(j)}\},
  f_{\sigma(j+1)},\ldots,f_{\sigma(k)}\|f_{k+1},f_{k+2}\}
\end{multlined}\\
-&\nonumber\!\begin{multlined}[t][\mylength]
\textstyle \sum_{j=1}^{k-1}\sum_{\sigma\in Sh(j,k-j)}
 \sgn{\sigma}e(\sigma;sx_{1},\ldots,sx_{k})\cdot{}\\
  \cdot{}\sgn{j(k-j)+k+
   (|x_{f_{k+1}\wedge f_{k+2}}|-1)(\sum_{i=j+1}^{k}(|x_{\sigma(i)}|-1))}\cdot{}\\
   \cdot\{\{f_{\sigma(1)},\ldots,f_{\sigma(j)}\|f_{k+1},f_{k+2}\},
    f_{\sigma(j+1)},\ldots,f_{\sigma(k)}\}
\end{multlined}\\
+&\nonumber\!\begin{multlined}[t][\mylength]
\textstyle \sum_{j=1}^{k-1}\sum_{\sigma\in Sh(j,k-j)}
 \sgn{\sigma}e(\sigma;sx_{1},\ldots,sx_{k})\cdot{}\\
  \cdot\sgn{(\sum_{i=k-j+1}^{k}|x_{\sigma(i)}|-1)|f_{k+1}|
    +(j-1)\sum_{i=1}^{k-j}(|x_{\sigma(i)}|-1)}\cdot{}\\
     \cdot\{f_{\sigma(1)},\ldots,f_{\sigma(k-j)}\|f_{k+1},
      \{f_{\sigma(k-j+1)},\ldots,f_{\sigma(k)},f_{k+2}\}\}
\end{multlined}\\
+&\nonumber\!\begin{multlined}[t][\mylength]
\textstyle \sum_{j=1}^{k-1}\sum_{\sigma\in Sh(j,k-j)}
 \sgn{\sigma}e(\sigma;sx_{1},\ldots,sx_{k})
  \sgn{(j-1)(\sum_{i=1}^{k-j}(|x_{\sigma(i)}|-1))}\cdot{}\\
   \cdot \{f_{\sigma(1)},\ldots,f_{\sigma(k-j)}\|\{f_{\sigma(k-j+1)},
    \ldots,f_{\sigma(k)},f_{k+1}\},f_{k+2}\}
\end{multlined}\\
=&\label{higher_first_Leibniz} 0\;.
\end{align} 
\end{theorem}
\begin{proof}
Multi-linearity is immediate. To see that the definition does not depend on the 
particular chosen associated Hamilton tensors, we use (\ref{kernel_prop}) 
and proceed as in the proof of (\ref{well_def_2}). 
To compute the homogeneity, recall $|x|=|f|+n$. Then
\begin{align*}
|\{f_1,\ldots,f_k\| f_{k+1},f_{k+2}\}|
&= |i_{x_1\wedge\cdots\wedge x_k}(f_{k+1}\smwedge f_{k+2})|\\
&=\textstyle\sum_{i=1}^k|x_i|+|f_{k+1}|+|f_{k+2}|\\
&= \textstyle\sum_{i=1}^k|f_i|+|f_2|+|f_3|+kn\;.
\end{align*}
The graded symmetry for arguments on the right is a consequence of the 
graded symmetry of the exterior cotensor product. The 
$(n-1)$-fold shifted graded antisymmetry for the arguments on the left, is
a consequence of the degree identity $|s^{n-1}f_i|=(|x_i|-1)$ and the 
decalage sign $\sgn{\sum_{i=1}^k(|x_i|-1)}$ which appears in the definition.

To see that $\{f_1,\ldots,f_k\|f_{k+1},f_{k+2}\}$ is a Poisson cotensor, for all
Poisson cotensors $f_1,\ldots,f_{k+2}$, we have to show that both equations
$$
\begin{array}{rcl}
i_{x_{\{f_1,\ldots,f_k\|f_{k+1},f_{k+2}\}}}&=&d\{f_1,\ldots,f_k\|f_{k+1},f_{k+2}\}\\
i_{y_{\{f_1,\ldots,f_k\|f_{k+1},f_{k+2}\}}}&=&\{f_1,\ldots,f_k\|f_{k+1},f_{k+2}\}
\end{array}
$$
have tensor solutions. We proof this by induction on $k$. From the proof of
theorem (\ref{first_leibniz_theorem}), we know that both equations are satisfied 
for $k=1$ and that the first $n$-plectic Leibniz equation holds for $k=1$, too. 
This serves as the induction base. 

For the induction step, suppose that there is a $k\in\N$, such that we 
already know that $\{f_1,\ldots,f_k\|f_{k+1},f_{k+2}\}$ is a Poisson cotensor
and that the first $n$-plectic Leibniz equation holds 
for all $j\leq k$.

Then we need to show,  that the first $n$-plectic Leibniz equation 
is moreover satisfied for $k+1$. 
Unfortunately the present proof of this equation is extremely long
and therefore we only sketch the proof for now. 
What we do is basically just to apply the definition of 
the various Leibniz operators and the Poisson brackets 
to each equation and then collect terms. 
What remains is a contraction along
a tensor, which is the Hamilton tensor of the same equation but for $(k-1)$.
By the induction hypothesis, this tensor is an element of the kernel of 
$\omega$ and therefore of any Poisson tensor by the kernel property 
(\ref{kernel_prop}). This completes the computation. 

Now since the $n$-plectic Leibniz equation is satisfied for all $j\leq k+1$, each
term is in particular a Poisson cotensor and therefore a solution to the equation 
$i_{y_{\{f_1,\ldots,f_k\|f_{k+1},f_{k+2}\}}}=\{f_1,\ldots,f_k\|f_{k+1},f_{k+2}\}$ 
exists. 

A solution to
$i_{x_{\{f_1,\ldots,f_k\|f_{k+1},f_{k+2}\}}}=d\{f_1,\ldots,f_k\|f_{k+1},f_{k+2}\}$
then follows from the $n$-plectic Leibniz equation in dimension $k+1$, too.
\end{proof}
\begin{remark}
We say that the first $n$-plectic Leibniz equation holds \textit{strictly} 
or 'on the nose', if all terms which involve Leibniz operators vaish. 
In that case we could say that the n-plectic Poisson $k$-bracket acts as 
some kind of derivation with respect to the exterior product.
\end{remark}
Loosely speaking we can say, that the first higher Leibniz equations, 
control the interaction 
between the various Poisson brackets and the differential graded commutative 
structure.
As we will see in corollary (\ref{corollary_main_3}), this is precisely the 
general structure equation
(\ref{main_hom_pois_struc_eq}) for the parameters 
$k\in \N$ and $p_1=1$, $\ldots$, $p_k=1$ and $p_{k+1}=2$.

However, two more Leibniz-like equations have to hold in any dimension $k$. 
They derive from (\ref{main_hom_pois_struc_eq}) for the parameters 
$p_1=1,\ldots,p_{k}=1,p_{k+1}=3$ as well as 
$p_1=1,\ldots,p_k=1,p_{k+1}=2,p_{k+2}=2$. The following theorem
makes this precise: 
\begin{theorem}
Let $(A,\mathfrak{g},\omega)$ be an $n$-plectic structure. Then for any
Poisson cotensors
$f_1,\ldots,f_{k+4}\in \mathcal{P}ois(A,\mathfrak{g},\omega)$ and associated
Hamilton tensors $x_1$, $\ldots$, $x_{k+4}$, the following 
\textbf{second n-plectic Leibniz equation} in dimension $(k+2)$ holds:
\begin{align}
-&\nonumber\!\begin{multlined}[t][\mylength]
\textstyle \sum_{j=0}^{k}\sum_{\sigma\in Sh(j,k-j)}\sgn{\sigma+j(k+1-j)}e(\sigma;sx_{1},\ldots,sx_{k})\\
\sgn{|x_{f_{k+1}\smwedge f_{k+2}}|\sum_{i=j+1}^{k}(|x_{\sigma(i)}|-1)}\\
\{\{f_{\sigma(1)},\ldots ,f_{\sigma(j)}\|f_{k+1},f_{k+2}\},f_{\sigma(j+1)},\ldots ,f_{\sigma(k)}\|f_{k+3},f_{k+4}\}
\end{multlined}\\
+&\nonumber\!\begin{multlined}[t][\mylength]
\textstyle \sum_{j=0}^{k}\sum_{\sigma\in Sh(j,k-j)}\sgn{\sigma+j(k+1-j)}e(\sigma;sx_{1},\ldots,sx_{k})\\
\sgn{\sum_{i=j+1}^{k}(|x_{\sigma(i)}|-1)+|x_{f_{k+1}\smwedge f_{k+2}}|)|x_{f_{k+3}\smwedge f_{k+4}}|}\\
\{\{f_{\sigma(1)},\ldots ,f_{\sigma(j)}\|f_{k+3},f_{k+4}\},f_{\sigma(j+1)},\ldots ,f_{\sigma(k)}\|f_{k+1},f_{k+2}\}
\end{multlined}\\
+&\nonumber\!\begin{multlined}[t][\mylength]
\textstyle \sum_{j=0}^{k}\sum_{\sigma\in Sh(j,k-j)}\sgn{\sigma+k}e(\sigma;sx_{1},\ldots,sx_{k})\\
\sgn{j(\sum_{i=1}^{k-j}(|x_{\sigma(i)}|-1)+|f_{k+1}|)+(\sum_{i=k-j+1}^{k}(|x_{\sigma(i)}|-1))(|f_{k+1}|+1)}\\
\{f_{\sigma(1)},\ldots ,f_{\sigma(k-j)}\|f_{k+1},\{f_{\sigma(k-j+1)},\ldots ,f_{\sigma(k)},f_{k+2}\|f_{k+3},f_{k+4}\}\}
\end{multlined}\\
+&\nonumber\!\begin{multlined}[t][\mylength]
\textstyle \sum_{j=0}^{k}\sum_{\sigma\in Sh(j,k-j)}\sgn{\sigma+(k-j)}e(\sigma;sx_{1},\ldots,sx_{k})\\
\sgn{j\sum_{i=1}^{k-j}(|x_{\sigma(i)}|-1)+|f_{k+2}||x_{f_{k+3}\smwedge f_{k+4}}|+\sum_{i=k-j+1}^{k}|x_{\sigma(i)}|}\\
\{f_{\sigma(1)},\ldots ,f_{\sigma(k-j)}\|\{f_{\sigma(k-j+1)},\ldots ,f_{\sigma(k)},f_{k+1}\|f_{k+3},f_{k+4}\},f_{k+2}\}
\end{multlined}\\
-&\nonumber\!\begin{multlined}[t][\mylength]
\textstyle \sum_{j=0}^{k}\sum_{\sigma\in Sh(j,k-j)}\sgn{\sigma+(k-j)}e(\sigma;sx_{1},\ldots,sx_{k})\\
\sgn{j(\sum_{i=1}^{k-j}(|x_{\sigma(i)}|-1))+|x_{f_{k+1}\smwedge f_{k+2}}||x_{k+3}|+\sum_{i=k-j+1}^{k}|x_{\sigma(i)}|}\\
\{f_{\sigma(1)},\ldots ,f_{\sigma(k-j)}\|\{f_{\sigma(k-j+1)},\ldots ,f_{\sigma(k)},f_{k+3}\|f_{k+1},f_{k+2}\},f_{k+4}\}
\end{multlined}\\
-&\nonumber\!\begin{multlined}[t][\mylength]
\textstyle \sum_{j=0}^{k}\sum_{\sigma\in Sh(j,k-j)}\sgn{\sigma+(k-j)}e(\sigma;sx_{1},\ldots,sx_{k})\\
\sgn{j|f_{k+3}|+j(\sum_{i=1}^{k-j}(|x_{\sigma(i)}|-1))+\sum_{i=k-j+1}^{k}|x_{\sigma(i)}|+|x_{f_{k+1}\smwedge f_{k+2}}|)|x_{f_{k+3}\smwedge f_{k+4}}|}\\
\{f_{\sigma(1)},\ldots ,f_{\sigma(k-j)}\|f_{k+3},\{f_{\sigma(k-j+1)},\ldots ,f_{\sigma(k)},f_{k+4}\|f_{k+1},f_{k+2}\}\}
\end{multlined}\\
=&0\label{second_higher_leibniz}\;.
\end{align}
In addition, for any
Poisson cotensors
$f_1,\ldots,f_{k+3}\in \mathcal{P}ois(A,\mathfrak{g},\omega)$ and associated
Hamilton tensors $x_1$, $\ldots$, $x_{k+3}$, the 
\textbf{thirs n-plectic Leibniz equation} in dimension $(k+2)$ holds: 
\begin{align}
0=&\{f_{1},\ldots,f_{k}\|f_{k+1}\smwedge f_{k+2},f_{k+3}\}\\
-&\{f_{1},\ldots,f_{k}\|f_{k+1},f_{k+2}\smwedge f_{k+3}\}\\
+&\{f_{1},\ldots,f_{k}\|f_{k+1},f_{k+2}\}\smwedge f_{k+3}\\
-&\sgn{\sum_{i=1}^{k}|x_{i}||f_{k+1}|}
 f_{k+1}\smwedge{\textstyle \{f_{1},\ldots,f_{k}\|f_{k+2},f_{k+3}\}}\\
+&\!\begin{multlined}[t][\mylength]
\textstyle \sum_{j=1}^{k-1}\sum_{\sigma\in Sh(j,k-j)}
 \sgn{\sigma+j\sum_{i=1}^{k-j}(|x_{\sigma(i)}|-1)}
  e(\sigma;sx_{1},\ldots,sx_{k})\\
\{f_{\sigma(1)},\ldots,f_{\sigma(k-j)}\|
 \{f_{\sigma(k-j+1)},\ldots,f_{\sigma(k)}\|f_{k+1},f_{k+2}\},f_{k+3}\}
\end{multlined}\\
-&\!\begin{multlined}[t][\mylength]
\textstyle \sum_{j=1}^{k-1}\sum_{\sigma\in Sh(j,k-j)}
 \sgn{\sigma}e(\sigma;sx_{1},\ldots,sx_{k})\\
  \sgn{j\sum_{i=1}^{k-j}(|x_{\sigma(i)}|-1)
 +\sum_{i=k-j+1}^{k}|x_{\sigma(i)}||f_{k+1}|}\phantom{mmmmmmm}\\
\{f_{\sigma(1)},\ldots,f_{\sigma(k-j)}\|f_{k+1},
 \{f_{\sigma(k-j+1)},\ldots,f_{\sigma(k)}\|f_{k+2},f_{k+3}\}\}
 \end{multlined}
 \label{third_higher_leibniz}
\end{align}
\end{theorem}
\begin{proof}
The computation is extremely long and therefore we only sketch
the proof for now. What we do is basically just to apply the definition of 
the various Leibniz operators to each equation and then just collect terms, 
using the associativity of the exterior product. What remains is a contraction along
a tensor, which is the Hamilton tensor of the same equation but for $(k-1)$.
By the induction hypothesis, this tensor is an element of the kernel of 
$\omega$ and therefore of any Poisson tensor by the kernel property 
(\ref{kernel_prop}). This completes the computation. 
\end{proof}
The following corollary summaries the partial steps for the various 
computations we did to derive the $n$-plectic homotopy Poisson-$n$ algebra of
Poisson cotensors in higher symplectic geometry:
\begin{corollary}\label{corollary_main_3}
Let $(A,\mathfrak{g},\omega)$ be an $n$-plectic structure.
The main structure equation (\ref{main_hom_pois_struc_eq}) of a homotopy
Poisson-$n$ algebra is satisfied for any $k\in\N$ and all 
$p_1,\ldots,p_k\in\N$.
\end{corollary}
\begin{proof}
Recall the structure in dimension $k\leq 2$, as given in the proof of corollary
(\ref{corollary_main_1}) and (\ref{corollary_main_2}). 
Define the following additional maps of (\ref{main_hom_pois_struc_maps}) 
for $k\geq 3$: The shifted brackets 
$D_{1,\ldots,1}(s^{n-1}f_1 \smwedge\cdots\smwedge s^{n-1}f_k)\defeq 
s\{f_1,\ldots,f_k\}$, as well as the shifted Leibniz operators
$D_{1,\ldots,1,2}(s^{n-1}f_1\smwedge\cdots\smwedge s^{n-1}f_{k}
\smwedge s^{n-2}(\overline{sf_{k+1}\otimes sf_{k+2}}))\defeq$
$\sgn{|sf_{k+1}|}s\{f_1,\ldots,f_k\|f_{k+1},f_{k+2}\}$. 
the operators $D_{1,\ldots,1,2,1,\ldots,1}$, are defined, similar to the 
one in the proof of (\ref{corollary_main_2}), by permuting the arguments
in $D_{1,\ldots,1,2}$ accordingly. All other structure maps 
$D_{q_1,\ldots,q_k}$ are assumed to be zero.

These maps are homogeneous of degree $(k-n)$ and 
have the expected symmetry. Since most maps are just zero,
equation (\ref{main_hom_pois_struc_eq}) is only non trivial 
in the following four cases: $(p_1,\ldots,p_k)\in 
\{(1,\ldots,1),(1,\ldots,1,2),(1,\ldots,1,3),(1,\ldots,1,2,2)\}$. 

The cases $(1,\ldots,1)$ are precisely the $(n-1)$-fold shifted homotopy
Jacoby equations.The cases 
$(1,\ldots,1,2)$, $(1,\ldots,1,2,2)$ and $(1,\ldots,1,3)$ are the first, second 
and third $n$-plectic Leibniz equation in dimension $k$, respectively.
\end{proof}

\begin{appendix}

\section{A Primer on Homotopy Poisson-n algebras} This chapter aims to give
a short 'generators and relations'-style introduction to 
homotopy Poisson-$n$ algebras. 
By definition these are algebras over the Koszul resolution 
$\Omega\, p_n^\text{!`}$ of the 
Poisson-$n$ operads $p_n$. For a more elaborate introduction see \cite{GeJo}.

We start with an introduction to the very basic notations of shuffles and graded
vector spaces. Then we look at the linear space that underlays the cofree 
Poisson-$n$ coalgebra. After that we make heavy use of the ideas from \cite{GTV} 
to derive the explicit structure of a homotopy Poisson-$n$ algebra in terms of
generating structure maps.
\subsection{Shuffle Permutation}\label{shuffle_permutation}
Let $S_k$ be the symmetric group, i.e the group of all bijective maps 
of the ordinal $\ordinal{k}$. Then for any $p,q\in \N$
a $(p,q)$-shuffle is a permutation 
$(\mu(1),\ldots,\mu(p),\nu(1),\ldots,\nu(q))\in S_{p+q}$
subject to the condition $\mu(1)<\ldots<\mu(p)$ and $\nu(1)<\ldots<\nu(q)$.
We write $Sh(p,q)$ for the set of all $(p,q)$-shuffles.
For more on shuffles, see for example at \cite{RS}.
\begin{remark}
Another common definition of a shuffle is a permutation subject to the condition
$\mu^{-1}(1)<\ldots<\mu^{-1}(p)$ and $\nu^{-1}(1)<\ldots<\nu^{-1}(q)$. However
we stick to our definition above and call a permutation with the latter property
an $(p,q)$-\textit{unshuffle}.
\end{remark}
\subsection{Graded Vector Spaces}We recall the most basics facts
about $\Z$-graded $\R$-vector spaces (just graded vector spaces for short). 

A $\Z$-graded vector space $V$ is the direct sum 
$\oplus_{n\in \Z} V_n$ of vector spaces $V_n$. An element 
$v\in V$ is said to be \textbf{homogeneous} of degree $n$, written as
$|x|=n$, if it is in the image of the natural inclusion $i_n: V_n \to V$, 
which comes from the direct sum. Every vector has a decomposition 
into homogeneous elements. 

A morphism $f : V \to W$ of graded vector spaces, homogeneous of degree
$r$, is a sequence of 
linear maps $f_n : V_n \to W_{n+r}$ for all $n\in \Z$. 
The integer $r$ is called the degree of $f$ and is denoted by $|f|$. 

A $k$-linear morphism $f: V_1 \times ... \times V_k \to W$
of graded vector spaces, , homogeneous of degree $r$, is a sequence of $k$-linear maps 
$f_{n_1,\ldots,n_k} : (V_1)_{n_1} \times \ldots \times (V_k)_{n_k} \to W_{\sum n_i+r}$ for all $n_i\in \Z$.

The (graded )tensor product $V \otimes W$ of two graded vector spaces
$V$ and $W$ is given by
$$
(V \otimes W )_n \defeq \oplus_{i+j=n}\left( V_i \otimes W_j\right)
$$
and the twisting morphism is given by
$\tau: V \otimes W \to  W \otimes V$ on homogeneous elements $v\otimes w \in V \otimes W$
by 
$$\tau(v \otimes w)\defeq (-1)^{|v||w|} w \otimes v$$ and then extended to
$V\otimes W$ by linearity. 
The category of $\Z$-graded vectors spaces is symmetric monoidal, with respect
to this tensor product and twisting morphism.
\subsubsection{The Koszul sign}\label{Koszul_sign_rules} Let $V$ be a $\Z$-graded
vector space, $\sigma\in S_k$ a permuation and let
$v_1,\ldots,v_k\in V$ be homogeneous vectors. Then the 
\textbf{Koszul sign} $e(\sigma;v_1,\ldots,v_k) \in \{-1,+1\}$ of the
permutation and the vectors is defined by the equation
\begin{equation}
v_1\otimes \ldots \otimes v_k= 
e(\sigma;v_1,\ldots,v_k) v_{\sigma(1)}\otimes \ldots \otimes v_{\sigma(k)}.
\end{equation}
In an actual computation the Koszul sign can be computed by the following rules: 
If a permutation $\sigma\in S_k$ is a transposition of consecutive 
neighbors only, i.e. it is a transposition $j\leftrightarrow j+1$, 
then $e(\sigma;v_1,\ldots,v_k)= (-1)^{|v_j||v_{j+1}|}$. 

If $\tau\in S_k$ is another permutation, then the Koszul sign of the 
composition is the composition of the Koszul signs, i.e.
$$e(\tau\sigma;v_1,\ldots,v_k)=
e(\tau;v_{\sigma(1)},\ldots,v_{\sigma(k)})e(\sigma;v_1,\ldots,v_k)\;.$$
Since any permutation can be realized as a composition of transpositions of
consecutive neighbors, these ruse are enough to compute any Koszul sign.
\begin{example} Let $V$ be a $\Z$-graded vector space,
 $v_1,\ldots,v_3\in V$ be homogeneous vectors and $\sigma\in S_3$ be given by
$(3,1,2)$. Then the associated Koszul sign is given by
$$
e(\sigma;v_1,v_2,v_3)=(-1)^{|v_1||v_3|+|v_2||v_3|}\;,
$$
since the permutation can be realized by two transpositions of consecutive
neighbors and for each such transposition, we can apply the twisting morphism.
\end{example}
\subsection{Homotopy Poisson-n algebras}\label{homotopy_poisson_n}
Most of this section is taken, more or less, from the work of Galvez-Cerrillo, 
Tonks and Vallette \cite{GTV} on homotopy Gerstenhaber algebras. 
Only slight adoptions are necessary to deal with
the situation of Poisson-$n$ algebras for arbitrary $n$.

If $V$ is a differential graded module, then
$s^{j}V$ means the $j$-fold shifting of $V$. Moreover,
given any vectors $v_1,\ldots,v_k\in sV$, permutation $\sigma\in S_k$ and interval 
$[i,j]=\{i,i+1,\ldots ,j\}$, $1\leq i\leq j\leq k$, we use the notation:
\begin{align*}
v_{[i,j]}&\defeq 
v_{i}\otimes v_{i+1}\otimes\dots\otimes
v_{j}
\\
v^\sigma_{[i,j]}&\defeq 
v_{\sigma^{-1}(i)}\otimes v_{\sigma^{-1}(i+1)}\otimes\dots\otimes
v_{\sigma^{-1}(j)}
\end{align*}
\subsubsection{The cofree Poisson-n coalgebra}
For any natural number $n\in\N$, let  
$p_n$ be the \textbf{Poisson-n} operad as for example introduced in
Getzler ans Jonas \cite{GeJo}. This operad is Koszul and can be seen as a 
composition of the commutative operad $\mathcal{C}om$ and the $(n-1)$-fold 
shifted Lie operad $\mathcal{S}^{n-1}\mathcal{L}ie$ by certain distributive 
laws.
$$
p_n \simeq \mathcal{C}om \circ \mathcal{S}^{n-1}\mathcal{L}ie\;.
$$
By the general algorithm \cite{LV}, the structure equations of a 
\textit{homotopy} $p_n$-algebra are encoded as a degree $-1$ 
coderivation on the locally nilpotent, cofree $p_n$-coalgebra, which, 
for any differential graded (dg) module $(V,d)$, is given by
$$
p^\text{!`}_n(V)= s^{-1} p_n^c(sV)= 
s^{-n}\mathcal{C}om^c\left(s^{n-1}\mathcal{L}ie^c(sV)\right)\;.
$$
For any (dg) module $W$, $\mathcal{C}om^c(W)$ is the locally nilpotent, 
cofree, commutative and coassociative coalgebra, 
which has the symmetric exterior tensor power 
$SW\defeq \bigoplus_{j\in\N}\odot^j W$ 
as underlying linear space. Moreover 
$\mathcal{L}ie^c(sV)$ is the cofree Lie coalgebra. The underlying 
graded vector space is given by
$$
\mathcal{L}ie^c(sV)=
\textstyle\bigoplus_{j\in\N} \overline{(sV)^{\otimes j}}\;,
$$
where $\overline{sV^{\otimes k}}$ is the quotient of the tensor power 
$(sV)^{\otimes k}$ by the images of the \textit{shuffle maps}
\begin{equation}
Sh_{(i,k-i)}: (sV)^{\otimes i}\otimes
(sV)^{\otimes (k-i)}\to (sV)^{\otimes k}
\end{equation}
for all $(1\leq i\leq k-1)$. These maps are defined for simple 
tensors $(v_1\otimes\cdots\otimes v_i)\otimes
(v_{i+1}\otimes\cdots\otimes v_k)$ by
\begin{multline*}
Sh_{(i,k-1)}((v_1\otimes\cdots\otimes v_i)\otimes
(v_{i+1}\otimes\cdots\otimes v_k))=\\
\textstyle\sum_{\sigma\in Sh(i,k-i)}e(\sigma;v_1,\ldots,v_k)
v_{\sigma^{-1}(1)}\otimes\cdots\otimes v_{\sigma^{-1}(k)}
\end{multline*}
and are then extended to all of $(sV)^{\otimes i}\otimes
(sV)^{\otimes (k-i)}$ by linearity.

From this follows that the underlying linear space of the 
cofree $p_n$-coalgebra is given by
\begin{multline*}
p_n^\text{!`}(V)\simeq 
\textstyle \bigoplus_{k\in \N}
 \bigoplus_{q_1\leq\cdots\leq q_k}s^{-n}\left(
  \left(s^{n-1}\overline{sV^{\otimes q_1}}\right)\odot \cdots 
  \odot\left(s^{n-1}\overline{sV^{\otimes q_k}}\right)\right)\simeq\\
\textstyle \bigoplus_{k\in \N}
 \bigoplus_{q_1\leq\cdots\leq q_k}s^{k-n}\left(
  \left(s^{n-2}\overline{sV^{\otimes q_1}}\right)\wedge \cdots 
  \wedge\left(s^{n-2}\overline{sV^{\otimes q_k}}\right)\right)
\end{multline*}
and that any linear map $D:p_n^\text{!`}(V)\to V$, 
homogeneous of degree $-1$ can equivalently be defined by a family of 
linear maps
$$
D_{q_1,\ldots,q_k}: 
  \left(s^{n-2}\overline{sV^{\otimes q_1}}\right)\wedge \cdots 
  \wedge\left(s^{n-2}\overline{sV^{\otimes q_k}}\right)\to sV\;,
$$
which are homogeneous of degree $k-n$ and are indexed by 
 $k,q_1,\ldots, q_k \in \N$.
\subsubsection{The generating structure maps}
The following definition of a straight unshuffle
is taken more or less verbatim from \cite{GTV}, where it is called a 
straight shuffle instead. The difference is due to a different approach 
to shuffles. 
\begin{definition}[Straight unshuffle]
Consider integers $k\geq 1$ and $l_j,r_j\geq 0$, $q_j\geq 1$,
$p_k=l_k+q_k+r_k$ for each $j=1,\dots,k+1$, and let
$$P_j=\textstyle\sum_{i=1}^{j-1}p_i,\qquad\qquad
L_j=\sum_{i=1}^k l_i+\sum_{i=1}^{j-1}{q}_i.
$$
Then  a \textit{straight
$(l_j,{q}_j,r_j,p_j)_{j=1}^k$-unshuffle}
is a $(p_1,\dots,p_k)$-unshuffle $\sigma$ satisfying the following extra property
for each $j=1,\dots,k$:
$$\sigma[P_j+l_j+1,P_j+l_j+{q}_j]=[L_j+1,L_{j+1}].
$$
By a {\em straight
$({q}_j,p_j)_{j=1}^k$-unshuffle}
we mean a straight $(l_j,{q}_j,r_j,p_j)_{j=1}^k$-unshuffle for some
values of $l_j,r_j\geq0$ with $l_j+r_j=p_j-q_j$.
\end{definition}
Less formally, a straight unshuffle may be thought of as a permutation
$\sigma$ of
a concatenation $X$ of $k$ strings of lengths $p_j$, each of which contains a
non-empty distinguished interval of length $q_j$, with $l_j$ elements on the left 
and $r_j$ elements on the right. For example,
$$
X\;\;=\;\;1\;\underline{2}\;3\;4\;|\;\underline{5\;
6}\;7\;8\;|\;9\;\underline{10\;11}\;
|\;\underline{12\;13}.
$$
For the permutation to be a straight unshuffle it must satisfy 
following conditions:
\begin{itemize}
\item
the orders of the elements within the $k$ strings are preserved by $\sigma$ (\emph{unshuffle}),
\item
the distinguished intervals appear unchanged and contiguously in the image (\emph{straight}).
\end{itemize}
For example, one possible straight unshuffles of $X$ is
$$
\quad\qquad 9\;1\;\;\underline{2\;5\;6\;10\;11\;12\;13}\;\;7\;3\;8\;4.
$$
More precisely, it is a straight $\big((1,0,1,0), (1,2,2,2), (2,2,0,0),(4,4,3,2)\big)$-unshuffles.

Using this definition, we can extend the straight shuffle extensions from 
\cite{GTV} to the situation of Poisson-$n$ algebras for arbitrary $n\in\N$.
\begin{definition}[Shifted straight shuffle extension]
Suppose we have natural numbers $n,k,q_1,\ldots, q_k\in\N$ and a graded
linear map
$$
D_{q_1,\dots,q_k}:
\left(s^{n-2}\overline{sV^{\otimes q_1}}\right)\wedge\cdots\wedge 
\left(s^{n-2}\overline{sV^{\otimes q_k}}\right)\rightarrow sV\;,
$$
which is homogeneous of degree $k-n$. Then, for any integers
$p_k\geq q_k$, the appropriate $(n-2)$-fold \textbf{shifted straight shuffle 
extension} of $D_{q_1,\dots,q_k}$ is the graded linear map
\begin{equation}
D_{q_1,\dots,q_k}^{p_1,\dots,p_k}:
\left(s^{n-2}\overline{sV^{\otimes p_1}}\right)\wedge\cdots\wedge 
\left(s^{n-2}\overline{sV^{\otimes p_k}}\right)\rightarrow 
\left(s^{n-2}\overline{sV^{\otimes p}}\right)
\end{equation}
with $p=1+\sum_{i=1}^k (p_i-{q}_i)$, defined for any simple tensor
$\overline{v_{[1,\,p_1]}}\smwedge \ldots\smwedge
\overline{v_{[P_k+1,\,P_{k+1}]}}\in$
$\left(s^{n-2}\overline{sV^{\otimes p_1}}\right)\wedge\cdots\wedge 
\left(s^{n-2}\overline{sV^{\otimes p_k}}\right)$ by
\begin{multline*}
\textstyle\sum_\sigma
e(\sigma;v_1,\ldots,v_{p_1+\ldots+p_k})\sgn{k\cdot|v^\sigma_{[1,L_1]}|}\cdot{}\\
{}\cdot s^{n-2}\left(\overline{
v^\sigma_{[1,L_1]}\otimes D_{{q}_1,\dots,{q}_k}(
s^{n-2}\overline{v^\sigma_{[L_1+1,L_2]}}
\smwedge\dots\smwedge s^{n-2}\overline{v^\sigma_{[L_k+1,L_{k+1}]}}
)\otimes v^\sigma_{[L_{k+1}+1,P_{k+1}]}}\right)
\end{multline*}
Here $\sigma$ runs over all straight 
$({q}_j,p_j)_{j=1}^k$-unshuffle and the integers $P_k$, $L_k$ are as above. 
\end{definition}
Using these shifted straight shuffle extensions, we can define the structure
of a homotopy Poisson-$n$ algebra in terms of generating functions and their
relations. The following definition makes this precise:
\begin{definition}\label{explicit Ginfty}
Let $n\in \N$ be a natural number an $(V,d)$ a differential graded module.
Then a \textbf{homotopy Poisson-n algebra} over $(V,d)$ is a family of 
graded linear maps
\begin{equation}\label{main_hom_pois_struc_maps}
D_{q_1,\dots,q_k}:
\left(s^{n-2}\overline{sV^{\otimes q_1}}\right)\wedge\cdots\wedge 
\left(s^{n-2}\overline{sV^{\otimes q_k}}\right)\rightarrow sV,
\end{equation}
which are homogeneous of degree $k-n$ and indexed by the parameters
$k,q_1,\ldots,q_k\in\N$, with
$q_1\leq\dots\leq q_k$ and $D_1=d$. In addition this family of structure 
maps has to satisfy the structure equations 
\begin{multline}\label{main_hom_pois_struc_eq}
0=\textstyle \sum_{j=1}^{k}\sum_{\sigma\in Sh(j,k-j)}
 \sum_{q_{1}=1}^{p_{\sigma(1)}}\ldots\sum_{q_{j}=1}^{p_{\sigma(j)}}
  \sgn{\sigma+j(k-j)}e(\sigma;v_{1},\ldots,v_{k})\\
D_{p,p_{\sigma(j+1)},\ldots,p_{\sigma(k)}}(
 D_{q_{1},\ldots,q_{j}}^{p_{\sigma(1)},\ldots,p_{\sigma(j)}}(
  v_{\sigma(1)}\smwedge\cdots\smwedge v_{\sigma(j)})\smwedge 
   v_{\sigma(j+1)}\smwedge\cdots\smwedge v_{\sigma(k)})
\end{multline}
with $p=1+\sum_i (p_{\sigma(i)}-{q}_{i})$,
for all $k,p_1,\ldots,p_k\in\N$ and 
homogeneous $v_i\in \overline{sV^{\otimes p_i}}$, with 
$1\leq i \leq k$.
\end{definition}
\section{A primer on Lie Rinehart pairs}\label{Lie_Rinehart_section}
Lie Rinehart pairs generalize the algebraic structure of
vector fields and smooth functions to commutative- and Lie-algebras, 
which are some kind of module with respect to each other. 

According to Huebschmann \cite{JH2}, this idea was first used implicitly
in the work of Jacobsen \cite{NJ} to study certain field extensions. 
Further research then appeared in Herz \cite{JCH}, Palais
\cite{RSP}, Rinehart \cite{GSR}, Huebschmann \cite{JH1}, Moerdijk \& Mrcun
\cite{MoMr} and now a comprehensive survey is given in \cite{JH2}.

They provide a purely algebraic model for the half geometric,
half algebraic \textit{Lie algebroids}, but appear more general, since 
there is no restriction to any notion of 'smoothness' whatsoever.

As a nontrivial 'horizontal' generalization of Lie \textit{algebras} 
they have their natural interpretation in the Poisson operad 
(not the Lie operad), which might be  
the basic difference to ordinary Lie theory, when it comes to 
differentiation/integration.

\subsection{Lie Rinehart pairs} After a short introduction to
general (commutative) Lie Rinehart pairs we look at the special case,
where the Lie algebra is a torsionless module with respect to its
commutative partner. Those pairs allow for general Cartan calculus
as known from multivector fields and differential forms.
 
In what follows, the symbol $\mathfrak{g}$ will always mean a real Lie algebra,
i.e. a $\R$-vector space together with an antisymmetric, bilinear map
\begin{equation}
[\cdot,\cdot]: \mathfrak{g}\times\mathfrak{g}\to\mathfrak{g}
\end{equation}
called the \textbf{Lie bracket}, such that for any three vector $x_1$, $x_2$ and $x_3\in\mathfrak{g}$ the \textbf{Jacobi identity}
$[x_1,[x_2,x_3]]+[x_2,[x_3,x_1]]+[x_3,[x_1,x_2]]=0$ is satisfied. 
If not stated otherwise, $\mathfrak{g}$ will be considered as a 
(trivially) $\Z$\textit{-graded} Lie algebra, concentrated in degree zero.

In addition $A$ will always mean a real
associative and commutative algebra with unit, 
that is a $\R$-vector space together with an associative and commutative, bilinear map
\begin{equation}
\cdot : A \times A \to A
\end{equation}
called the product of $A$ and a unit $1_A\in A$. If not stated otherwise, $A$ 
will be considered as a (trivially) $\Z$\textit{-graded} algebra, concentrated 
in degree zero.

According to a better readable text, 
we usually suppress the symbol of the multiplication in $A$ and just
write $ab$ instead of $a\cdot b$.

Moreover $Der(A)$ will be the Lie algebra of derivations of $A$, 
i.e. the vector space of linear endomorphisms of $A$, 
with $D(ab)=D(a)b+aD(b)$ and Lie bracket
$[D,D'](a)\defeq D(D'(a))-D'(D(a))$ for any $a,b\in A$ and $D,D'\in Der(A)$.

Before we get to Lie Rinehart pairs, it is handy to define Lie algebra modules first:
\begin{definition}[Lie algebra module]The algebra $A$ is called 
a \textbf{Lie algebra module} for the Lie algebra $\mathfrak{g}$, if
there is a morphism of Lie algebras $D:\mathfrak{g}\to Der(A)$. 
In that case, $D$ is called the $\mathfrak{g}$\textbf{-scalar multiplication}
on $A$.
\end{definition}
Now a \textit{Lie Rinehart pair} is nothing but a Lie algebra 
together with a unital and commutative algebra, 
each of them being a module with respect to the other, such that 
a particular compatibility equation is satisfied:  
\begin{definition}[Lie Rinehart pair]\label{Lie_Rinehart_pair}
Let $A$ be an associative and 
commutative algebra with unit, $\mathfrak{g}$ a Lie algebra and 
$\cdot_A: A\times \mathfrak{g} \to \mathfrak{g}$ as well as
$D: \mathfrak{g}\to Der(A)\;;\; x\mapsto D_x$ maps, such that 
$A$ is a $\mathfrak{g}$-module with $\mathfrak{g}$-scalar multiplication
$D$, the vector space $\mathfrak{g}$ is
an $A$-module with $A$-scalar multiplication $\cdot_A$ and 
the \textbf{Leibniz rule}
\begin{equation}
[x,a\cdot_A y] = D_x(a)\cdot_A y + a \cdot_A[x,y]
\end{equation}
is satisfied for any $x,y\in\mathfrak{g}$ and $a\in A$. Then
$\left(A,\mathfrak{g}\right)$ is called a \textbf{Lie Rinehart pair}
and the map $D$ is called its \textbf{anchor map}.
\end{definition}
This was first referred to as an $(\R, A)$-Lie algebra 
(\cite{JH1} and \cite{GSR})
and later J. Huebschmann called it a Lie Rinehart algebra. 
We use the term Lie Rinehart \textit{pair}, to stress
that both partners should be seen on an equal footing.
\begin{remark}Judged by the fact that any Lie Rinehart pair is predominantly an
algebra for the \textit{Poisson} operad, not the Lie operad \cite{LV},
this structure should eventually be called a \textbf{Poisson}-Rinehart pair, 
therefore pointing to the fundamental role of the Poisson 
structure instead of (only) the Lie structure.
\end{remark}
The two most extreme examples derive from commutative algebras on one side and
Lie algebras on the other. This reflects the fact that a Poisson algebras is a
certain combination of a Lie and a commutative algebra:
\begin{example}\label{example_com_al}
For any commutative and associative algebra with unit $A$, a 
Lie Rinehart pair is given by $(A, Der(A))$, together with the standard 
$A$-module structure of $Der(A)$ and the identity as anchor map.
\end{example}
\begin{example}
Any real Lie algebra $\mathfrak{g}$ is a $\R$-module
with respect to its ordinary scalar multiplication 
and therefore $\left(\R,\mathfrak{g}\right)$ is a Lie Rinehart pair,
with trivial anchor
$$
D:\mathfrak{g}\times\R \to\R\;;\; (x,\lambda)\mapsto D_x(\lambda)\defeq 0\;.
$$
\end{example}
One of the most prominent examples is a particular instance of example 
(\ref{example_com_al}). It is at the heart of calculus in differential geometry
and provides the mathematical background for symplectic and
multisymplectic geometry:
\begin{example}
Let $M$ be a differentiable manifold, $C^\I(M)$ the algebra of smooth, 
real valued functions and $\mathfrak{X}(M)$ the Lie algebra of vector
fields on $M$. $\mathfrak{X}(M)$ is a $C^\I(M)$-module and 
vector fields act as derivations on smooth functions, that is the
map
$$D:\mathfrak{X}(M)\times C^\I(M) \to C^\I(M)\;;\;
 (X,f) \mapsto D_X(f)\defeq  X(f)
$$
satisfies the equation
$D_X(fg)=D_X(f)g + fD_X(g)$. Moreover the Leibniz rule
$[X,fY]=D_X(f)Y+f[X,Y]$ holds. 
\end{example}
The following important example unmasks Lie algebroids as special 
Lie Rinehart pairs, but expressed in a more \textit{geometric} flavor. 
This is analog to the situation of projective modules and smooth vector bundles, as 
exhibited by the Serre-Swan theorem:
\begin{example} 
A \textbf{Lie algebroid} $(E,M,[\cdot,\cdot],D)$ is a 
smooth vector bundle $E \to M$ with a Lie bracket 
$[\cdot,\cdot] : \Gamma(E)\times \Gamma(E) \to \Gamma(E)$ on its sections
$\Gamma(E)$
and a morphisms of vector bundles $D : E \to TM$, called the anchor, such that 
the tangent map $TD$ induce a morphism of Lie algebras,
with $[X, f \cdot Y] = f\cdot [X,Y] + TD_X(f) \cdot Y$
for all $X, Y \in \Gamma(E)$ and $f \in C^\infty(X)$.
\end{example}
With the Lie Rinehart structure at hand, their morphisms are defined
as appropriate algebra maps, that interact properly with
respect to the additional module structures \cite{GSR}:
\begin{definition}[Lie Rinehart Morphism]\label{Lie_Rinehart_morphism}
Let $(A,\mathfrak{g})$ and $(B,\mathfrak{h})$ be two Lie Rinehart pairs.
A \textbf{morphism of Lie Rinehart pairs} is a pair of maps $(f,g)$, such
that $f:A\to B$ is a morphism of associative and commutative, real algebras
with unit, $g:\mathfrak{g}\to\mathfrak{h}$ is a morphism of Lie
algebras and the equations
\begin{equation}\label{LR-morph}
\begin{array}{ccc}
g(a\cdot_A x)=f(a)\cdot_B g(x)
& and &
f(D_x(a))=D_{g(x)}(f(a))
\end{array}
\end{equation} 
are satisfied for any $a\in A$ and $x\in \mathfrak{g}$.  
\end{definition}
This is the covariant definition. If we see a Lie Rinehart pair as a Lie
algebroid, the structure maps are usually defined 
contravariant as functions between the de Rham complexes
of the algebroids.

The usability of differential geometry is build to a large extent on its 
computational simplicity. Often basic analysis can solve 
complex geometric problems. The underlying rules are \textit{Cartan calculus}, 
but to make something like this available in the general Lie Rinehart setting, 
a certain non-degeneracy condition on the double dual of the Lie partner is 
necessary:

We write $\mathfrak{g}^\vee_A\defeq Hom_{A\_ mod}(\mathfrak{g},A)$ 
for the $A$-dual of the $A$-module $\mathfrak{g}$ and 
$\mathfrak{g}^{\vee\vee}$ for the appropriate double dual. 
The following definition specializes torsionless (semi-reflexive)
modules to the Lie Rinehart setting:
\begin{definition}[Torsionless Lie Rinehart pair]\label{torsionless}
Let $(A,\mathfrak{g})$ be a Lie Rinehart pair, such that the natural map
$\mathfrak{g}\to \mathfrak{g}^{\vee\vee}\;;\;
x\mapsto \left(\mathfrak{g}^\vee \to A\;;\; f\mapsto f(x)\right)$ is 
injective. Then $(A,\mathfrak{g})$ is called \textbf{torsionless} (or 
\textbf{semi-reflexive}).
\end{definition}
Torsionless Lie Rinehart pairs hold a non degenerate and therefore
unique pairing between tensors and cotensors,
providing a Cartan calculus, that is completely analog to the 
differential geometric setting (\ref{Cartan_Calculus}).
\subsection{Exterior tensor algebra}We look at the 
exterior tensor power of the Lie algebra, seen
as a module with respect to its commutative partner. The Lie bracket
extends to the \textit{Schouten-Nijenhuis bracket}, which interacts
with the exterior product in terms of a Gerstenhaber structure. This 
is the 'free commutative prolongation' 
$\mathcal{C}om(A\oplus s\mathfrak{g})$ of the Gerstenaber structure on
$A\oplus s\mathfrak{g}$.

We define $\otimes^0_A\mathfrak{g}\defeq A$ 
and write $\otimes^n_A\mathfrak{g}$ for the 
$n$-fold $A$-tensor products of the $A$-module 
$\mathfrak{g}$. Since $A$ is commutative, 
$\otimes^n_A\mathfrak{g}$ is an $A$-module.

\begin{definition}[Exterior tensor algebra]\label{tensor_algebra} 
Let $(A,\mathfrak{g})$ be a Lie Rinehart pair and 
$n\in\Z$. For $n<0$ define $X_n(\mathfrak{g},A)=\{0\}$ and for $n\geq 0$
let $X_n(\mathfrak{g},A)\defeq \otimes^n_A\mathfrak{g}/J^n$ be
the quotient $A$-module of the $n$-th tensor power by the submodule $J^n$, 
which is spanned from all $x_1\otimes\cdots\otimes x_n$ with 
$x_i = x_j$ for some $i \neq j$. Then the direct sum 
\begin{equation}
X_\bullet(\mathfrak{g},A)\defeq \textstyle\bigoplus_{n\in\Z} X_n(\mathfrak{g},A)
\end{equation}
together with the quotient 
$\smwedge: X_\bullet(\mathfrak{g},A) \times 
 X_\bullet(\mathfrak{g},A) \mapsto
  X_\bullet(\mathfrak{g},A)\;;\; (x,y)\mapsto x\smwedge y$ of the
$A$-tensor multiplication, is called the \textbf{exterior tensor algebra} of 
$(A,\mathfrak{g})$ and the product is called the \textbf{exterior tensor product}. 
The induced grading on $X_\bullet(\mathfrak{g},A)$ is called the 
\textbf{tensor grading}.

In addition we define $X^{-n}(\mathfrak{g},A)\defeq  X_n(\mathfrak{g},A)$ for any
integer $n\in\Z$. The induced grading on the direct sum
\begin{equation}
X^\bullet(\mathfrak{g},A)\defeq \textstyle\bigoplus_{n\in\Z} X^n(\mathfrak{g},A)
\end{equation}
is called the \textbf{cotensor grading}. If the grading is irrelevant we just 
write $X(\mathfrak{g},A)$.
\end{definition}
\begin{example}
If $(C^\I(M),\mathfrak{X}(M))$ is the Lie Rinehart pair of smooth 
functions and vector fields, the exterior tensor algebra is the algebra 
of \textbf{multivector fields} $\bwedge\mathfrak{X}(M)$.
\end{example}
The exterior tensor power of a Lie Rinehart pair has the structure of a 
\textit{Gerstenhaber algebra} (also called Poisson-$2$ algebra) with respect to the 
exterior product and the Schouten-Nijenhuis bracket. The latter is nothing but
the free commutative lift of the Lie bracket on the Gerstenhaber algebra
 $A\oplus s\mathfrak{g}$.
\begin{definition}[Schouten-Nijenhuis bracket]
Let $(A,\mathfrak{g})$ be a Lie Rinehart pair with exterior
tensor power $X(\mathfrak{g},A)$. Then the
\textbf{Schouten Nijenhuis bracket} is the map
\begin{equation}
\left[\cdot\;,\cdot\right]: X(\mathfrak{g},A) \times 
 X(\mathfrak{g},A) \to X(\mathfrak{g},A)\;,
\end{equation}
defined by $[a,b]=0$ as well as $[x,a]=[a,x]=D_x(a)$ on scalars $a,b\in A$ and 
vectors $x\in\mathfrak{g}$ and by
\begin{multline}\label{SN_1}
[x_1\smwedge\cdots \smwedge x_n,y_1\smwedge\cdots\smwedge y_m]=\\
	\textstyle\sum_{i,j}(-1)^{i+j}[x_i,y_j]\smwedge x_1\smwedge\cdots\smwedge 
		\widehat{x_i}\smwedge\cdots\smwedge x_n\smwedge
			y_1\smwedge\cdots\smwedge \widehat{y_j}\smwedge\cdots\smwedge y_m
\end{multline}
on simple tensors 
$x_{1}\smwedge\cdots\smwedge x_{n}$, 
$y_{1}\smwedge\cdots\smwedge y_{m}\in X(\mathfrak{g},A)$ 
and then extend to $X(\mathfrak{g},A)$ by $A$-additivity.

We write $(X(\mathfrak{g},A),\smwedge,[\cdot,\cdot])$ for
the appropriate Gerstenhaber algebra and call it the 
\textbf{Schouten-Nijenhuis} algebra of $(A,\mathfrak{g})$. 
\end{definition}

In case of multivector fields, a proof can be found for example 
in \cite{MM} or \cite{PM}. Those proofs are an implication of the Lie 
Rinehart structure only and therefore carry over
verbatim to the general situation.
\subsection{The de Rham complex}
We generalize the de Rham complex of differential forms to the
setting of a torsionless Lie Rinehart pair.

Recall that we consider $\mathfrak{g}$ as an $A$-module
but concentrated in cotensor degree $-1$. Therefore the dual module 
$\mathfrak{g}^\vee$ is an $A$-module concentrated in cotensor degree $1$.
We put $\otimes^0_A\mathfrak{g}^\vee_A\defeq A$ (concentrated in degree zero)
and write $\otimes^n_A\mathfrak{g}^\vee_A$ for the 
$n$-fold graded $A$-tensor products of the $A$-module 
$\mathfrak{g}^\vee$. Since $A$ is commutative, 
each $\otimes^n_A\mathfrak{g}_A^\vee$  is an $A$-module.
\begin{definition}[Exterior cotensor algebra]\label{ctensor_algebra} 
Let $(A,\mathfrak{g})$ be a Lie Rinehart pair and 
$n\in\Z$. For $n<0$ define $\Omega^{n}(\mathfrak{g},A)=\{0\}$ and if 
$n\geq 0$ let $\Omega^{n}(\mathfrak{g},A)\defeq \otimes^n_A\mathfrak{g}^\vee_A/J^n$ be
the quotient $A$-module of the $n$-th tensor product of the dual module
and the submodule $J^n$, spanned by all cotensors $x^1\otimes\cdots\otimes x^n$ with 
$x^i = x^j$ for some $i \neq j$. Then the direct sum 
\begin{equation}
\Omega^\bullet(\mathfrak{g},A)=\textstyle\bigoplus_{n\in\Z} 
 \Omega^{n}(\mathfrak{g},A)
\end{equation}
together with the quotient 
$\smwedge: \Omega^\bullet(\mathfrak{g},A) \times 
 \Omega^\bullet(\mathfrak{g},A) \mapsto
  \Omega^\bullet(\mathfrak{g},A)\;;\; (x,y)\mapsto x\smwedge y$ of the
$A$-cotensor multiplication, is called the \textbf{exterior cotensor algebra} of 
$(A,\mathfrak{g})$ and the product is called the \textbf{exterior cotensor product}. 
The grading on $\Omega^\bullet(\mathfrak{g},A)$ is called the 
\textbf{cotensor grading}.

In addition we define $\Omega_{n}(\mathfrak{g},A)\defeq  
\Omega^{-n}(\mathfrak{g},A)$ for any integer $n\in\Z$. The induced grading on the 
direct sum
\begin{equation}
\Omega_\bullet(\mathfrak{g},A)\defeq \textstyle\bigoplus_{n\in\Z} 
\Omega_n(\mathfrak{g},A)
\end{equation}
is called the \textbf{tensor grading}. If the grading is irrelevant we just 
write $\Omega(\mathfrak{g},A)$.
\end{definition}
The exterior cotensor power has
the structure of a differential graded algebra, at least if the
Lie Rinehart pair is torsionless. This refines the well
known de Rham Complex of differential forms to the general torsionless Lie 
Rinehart setting:
\begin{definition}[De Rham differential]
Let $(A,\mathfrak{g})$ be a torsionless Lie Rinehart pair with exterior 
cotensor algebra $\Omega(\mathfrak{g},A)$. Then the \textbf{de Rham differential} (or exterior differential)
$$
d:\Omega(\mathfrak{g},A)\to \Omega(\mathfrak{g},A)
$$
is defined for a homogeneous cotensor $f\in \Omega^k(\mathfrak{g},A)$ 
and any simple exterior tensors
$x_0\smwedge\cdots\smwedge x_k\in X(\mathfrak{g},A)$ 
by the equation
\begin{multline}\label{deRam-diff}
df(x_0\smwedge\cdots\smwedge x_k) 
= \textstyle\sum_j (-1)^{j} 
 D_{x_j}(f(x_0\smwedge \cdots\smwedge \widehat{x}_j\smwedge
   \cdots\smwedge x_k))\\
+\textstyle\sum_{i<j}
 (-1)^{i+j}f([x_i, x_j]\smwedge x_0\smwedge \cdots\smwedge 
  \widehat{x}_i\smwedge \cdots\smwedge 
  \widehat{x}_j\smwedge \cdots\smwedge x_k)
\end{multline}
and is then extended to all of $\Omega(\mathfrak{g},A)$ by $A$-linearity. 

We call both the chain complex
$\left(\Omega^\bullet(\mathfrak{g},A),d\right)$ as well as the cochain complex
$\left(\Omega_\bullet(\mathfrak{g},A),d\right)$, the 
\textbf{de Rham complex} of a Lie Rinehart pair.
\end{definition}
Since the natural pairing is non degenerate, the exterior derivative is
uniquely defined by equation (\ref{deRam-diff}). It is moreover a graded map, 
homogeneous of degree $1$ with respect to the cotensor grading
(hence of degree $-1$ with respect to the tensor grading).
Moreover $d$ is a differential, that is $d^2=0$.
\begin{example}If the Lie Rinehart pair happens to be a Lie algebra (), the
de Rham complex is equal to the Chevalley-Eilenberg complex of that Lie algebra.
It should however be noted, that in general  the de Rham complex of a Lie 
Rinehart pair is \textit{not} equivalent to the Chevalley-Eilenberg complex of 
just the Lie partner.
\end{example}
\subsection{Cartan Calculus}\label{Cartan_Calculus}
We close this section with an introduction to the basic computational rules of 
Cartan calculus. These rules are simple yet powerful tools, which we 
use extensively in the present work.

The sceptical reader should pay special attention to the last equation in 
(\ref{multi_rules-5}). To the best of the authors knowledge,
this equation does not appear anywhere else, but is of central importance in
this work. 

To start, observe that the dual nature of tensors and cotensors 
gives rise to a bilinear map, that basically \textit{evaluates}
cotensors along tensors of equal degree: 
\begin{definition}[Natural Pairing]Let $(A,\mathfrak{g})$
be a Lie Rinehart pair. The graded $A$-linear map
\begin{equation}
\langle\cdot,\cdot\rangle:\Omega^\bullet(\mathfrak{g},A)\otimes_A 
 X^\bullet(\mathfrak{g},A)\to A,
\end{equation}
defined by $\langle a,b\rangle=ab$ on scalars $a,b\in A$, by
$\langle f,x\rangle=0$ on cotensors $f\in \Omega^p(\mathfrak{g},A)$ and
tensors $x\in X^q(\mathfrak{g},A)$, such that $p+q\neq 0$ and by
\begin{equation}\label{determinant_eq_1}
\langle f_1\smwedge\cdots\smwedge f_k, x_1\smwedge\cdots\smwedge x_k\rangle=
\textstyle\sum_{s\in S_k}\sgn{s}
 \langle f_{s(1)},x_1\rangle\cdots\langle f_{s(k)},x_k\rangle
\end{equation}
on simple cotensors 
$f_1\smwedge\cdots\smwedge f_k\in \Omega^\bullet(\mathfrak{g},A)$
and tensors $x_1\smwedge\cdots\smwedge x_k\in X^\bullet(\mathfrak{g},A)$
and then extended to all of 
$\Omega^\bullet(\mathfrak{g},A)\otimes_A X^\bullet(\mathfrak{g},A)$ 
by $A$-linearity, is called the \textbf{natural pairing}.
\end{definition}
Since the determinant is invariant with respect to transposition, 
the defining equation can equivalently be written in its transposed version
\begin{equation}\label{determinant_eq_2}
\langle f_1\smwedge\cdots\smwedge f_m,x_1\smwedge\cdots\smwedge x_m\rangle =
\textstyle\sum_{s\in S_m}\sgn{s}
 \langle f_1,x_{s(1)}\rangle\cdots\langle f_m,x_{s(m)}\rangle\;.
\end{equation}
The following proposition shows that the natural pairing is non-degenerate
and therefore unique, in case the Lie Rinehart pair is torsionless. Since
all operators in Cartan calculus are defined in terms of this pairing, it
justifies our restriction to the torsionless setting:
\begin{proposition}\label{unique_pairing}
Let $(A,\mathfrak{g})$ be a \textit{torsionless} Lie Rinehart pair with natural
pairing 
$\langle\cdot,\cdot\rangle:
 \Omega^\bullet(\mathfrak{g},A)\otimes_A 
  X^\bullet(\mathfrak{g},A)\to A$
and $x,y\in \mathfrak{g}$ two vectors. Then 
$\langle f,x \rangle = \langle f, y\rangle$ for all 
$f \in \mathfrak{g}^\vee$ implies $x=y$.
\end{proposition}
\begin{proof}Since $(A,\mathfrak{g})$ is torsionless, the inclusion 
$\mathfrak{g}\to \mathfrak{g}^{\vee\vee}\;;\;
x\mapsto \left(\mathfrak{g}^\vee \to A\;;\; f\mapsto f(x)\right)$ is injective.
Hence $f(x)=f(y)$ for all $f\in \mathfrak{g}^\vee$ implies $x=y$.
\end{proof}
With a non degenerate pairing, contraction can be defined consistently
by a certain adjointness property.
In fact we determine the usual right contraction of a cotensor along
a tensor as right adjoint to the exterior product and dually 
the left contraction of a tensor by a cotensor as left adjoint to the exterior 
product.

The latter operator seems to be not widely known in the literature. 
However it appears in \cite{MM} for multivector fields and differential 
forms and we generalize it to arbitrary torsionless Lie Rinehart pairs:
\begin{definition}[Right/Left Contraction]
Let $(A,\mathfrak{g})$ be a torsionless Lie Rinehart pair and
$x\in X_{-k}(\mathfrak{g},A)\simeq X^{k}(\mathfrak{g},A)$ a tensor homogeneous
of cotensor degree $k$. Then the 
\textbf{right contraction along }$x$ is the map
$$
i_x:\Omega^\bullet(\mathfrak{g},A)\to \Omega^\bullet(\mathfrak{g},A)\;,
$$
that is defined for any integer $n\in\Z$ and homogeneous cotensor 
$f\in \Omega^{n}(\mathfrak{g},A)$ by the equation
\begin{equation}\label{right contration}
\langle i_xf, y \rangle = \langle f, x\smwedge y\rangle
\end{equation}
for all $y\in X^\bullet(\mathfrak{g},A)$ and is then extended to 
$\Omega^\bullet(\mathfrak{g},A)$ by additivity. 

Dually for any
cotensor $f\in \Omega^{k}(\mathfrak{g},A)\simeq \Omega_{-k}(\mathfrak{g},A)$, 
homogeneous of cotensor degree $k$, the 
\textbf{left contraction along }$f$ is the map
$$
j_f:X^\bullet(\mathfrak{g},A)\to X^\bullet(\mathfrak{g},A)
$$
that is defined for any integer $n\in\Z$ and homogeneous tensor 
$x\in X^n(\mathfrak{g},A)$ by the equation
\begin{equation}\label{left contration}
\langle g, j_f x \rangle = \langle g\smwedge f, x\rangle
\end{equation}
for all $g\in \Omega^\bullet(\mathfrak{g},A)$ and then extended to  
$\Omega^\bullet(\mathfrak{g},A)$ by additivity.
\end{definition}
The following corollary summarizes the basic rules, regarding the left and
right contraction operators:
\begin{corollary}Let $(A,\mathfrak{g})$ be a torsionless Lie Rinehart
pair, $x,y\in X^\bullet(\mathfrak{g},A)$ and 
$f,g\in \Omega^\bullet(\mathfrak{g},A)$. Then
\begin{equation}
\begin{array}{lcr}
i_{x\wedge y}f=i_y \circ i_x f & and  & j_{f\wedge g}x = j_f \circ j_g x\;.
\end{array}
\end{equation}
If $x$ and $f$ are homogeneous, the contractions $i_xf$ and $j_fx$ are 
homogeneous, with $|i_xf|=|x|+|f|$ and $|j_fx|=|x|+|f|$.
\end{corollary}
\begin{proof}The first two equations follow from the associativity of the 
exterior products. In particular, let $z\in X^\bullet(A,\mathfrak{g})$ be
any tensor. Then 
\begin{align*}
\langle i_y \circ i_x f, z\rangle
&=\langle i_x f, y \smwedge z\rangle\\
&= \langle f, x \smwedge y \smwedge z\rangle\\
&= \langle i_{x \wedge y} f, z\rangle\;.
\end{align*}
Hence the first equation follows from (\ref{unique_pairing}).
The computation of the second equation is analog.

To see the homogeneity part, observe that we consider tensors as concentrated
in negative degrees and cotensors as concentrated in positive degrees, with 
respect to the cotensor grading.
\begin{comment}
To see the first equation, let $h\in E(A,\mathfrak{g}^\vee)$ be
any cotensor. Then 
\begin{align*}
\langle h, j_f \circ j_g x\rangle
&=\langle h \smwedge f, j_g x\rangle\\
&= \langle h \smwedge f \smwedge g, x\rangle\\
&= \langle h, j_{f\wedge g} x\rangle\;.
\end{align*}
\end{comment}
\end{proof}
Since the contraction operators are linear in both arguments, we can introduce
their kernels the usual way:
\begin{definition}[Kernel]Let $(A,\mathfrak{g})$ be a torsionless
Lie Rinehart pair. The \textbf{kernel} of a cotensor 
$f\in \Omega^\bullet(\mathfrak{g},A)$ is the
set 
$$\ker(f)\defeq \{x\in X^\bullet(\mathfrak{g},A)\;|\;
 i_xf=0\}$$
of all tensors, such that the right contraction of $f$ along that tensor
vanishes and the \textbf{kernel} of a tensor 
$x\in X^\bullet(\mathfrak{g},A)$ is the
set 
$$\ker(x)\defeq \{f\in \Omega^\bullet(\mathfrak{g},A)\;|\;
 j_fx=0\}$$
of all cotensors, such that the left contraction of $x$ along that cotensor
vanishes. 
\end{definition}
The Lie derivative along tensors in now defined as usual in terms of 
Cartans infinitesimal homotopy formula: 
\begin{definition}Let $(A,\mathfrak{g})$ be a torsionless Lie Rinehart pair and
$x\in X^\bullet(\mathfrak{g},A)$ a homogeneous tensor. Then the map
\begin{equation}\label{Cartans_formula}
L_x: \Omega^\bullet(\mathfrak{g},A)\to \Omega^\bullet(\mathfrak{g},A)\;;\;
f\mapsto di_xf - (-1)^{|x|}i_xdf
\end{equation}
is called the \textbf{Lie derivative} along $x$ and equation
(\ref{Cartans_formula}) is called \textbf{Cartans infinitesimal graded
homotopy formula}.
\end{definition}
The following proposition shows the basic computation rules that we need in
what follows. Most of these equations are long known. However
to the best of the authors knowledge it seems that the last equation does not
appear anywhere else in the literature.
\begin{proposition}Let $(A,\mathfrak{g})$ be a torsionless Lie Rinehart 
pair, $x,y\in X_\bullet(\mathfrak{g},A)$ homogeneous tensors and 
$f\in \Omega_\bullet(\mathfrak{g},A)$ a cotensor. Then
\begin{align}
\label{multi_rules-1}
dL_x f &=(-1)^{|x|-1}L_x df \\
\label{multi_rules-2}
i_{[x,y]} f &=(-1)^{(|x|-1)|y|}L_x i_y\, f - i_y L_x f\\
\label{multi_rules-3}
L_{[x,y]} f &= (-1)^{(|x|-1)(|y|-1)}L_x L_x f - L_y L_x f\\
\label{multi_rules-4}
L_{x\wedge y} f &= (-1)^{|y|} i_y L_x f + L_y i_x f \\
\label{multi_rules-5}
i_{j_{(i_y f)} x}f &= i_{y}f \smwedge i_xf
\end{align}
\end{proposition}
\begin{proof} Except for (\ref{multi_rules-5}), 
these equations are long known.
A proof is given in \cite{FPR} in the context of multivector fields
and differential forms, but can be carried over verbatim into the general 
Lie Rinehart setting.
My proof of  (\ref{multi_rules-5}) is verly long and therefore 
done separately in the next section.
\end{proof}
\subsubsection{Proof of Equation (\ref{multi_rules-5})}
This section is devoted to the proof of
equation (\ref{multi_rules-5}). Unfortunately I was not able to derive
it easily, but only by a direct and very long computation.

To start, we derive explicit expressions for
the left and right contraction operations, in case both arguments are simple:
\begin{proposition}Given a torsionless Lie Rinehart pair $(A,\mathfrak{g})$, a
simple tensor
$x\defeq x_1\smwedge\cdots\smwedge x_r\in X(\mathfrak{g},A)$  and a simple
cotensor $f\defeq f_1\smwedge\cdots\smwedge f_m\in \Omega(\mathfrak{g},A)$, the
right contraction $i_xf$ vanishes for $r>m$ and can be computed by 
\begin{equation}\label{right_contrac_1}
i_xf=
\textstyle\frac{1}{(m-r)!}\sum_{\sigma\in S_m}\sgn{\sigma}
\langle f_{\sigma(1)},x_1\rangle\cdots\langle f_{\sigma(r)},x_r\rangle
\cdot f_{\sigma(r+1)}\smwedge\cdots\smwedge f_{\sigma(m)}
\end{equation}
if $r\leq m$.
Similar the left contraction $j_fx$ vanishes for $m>r$ and 
can be computed by
\begin{equation}\label{left_contrac_1}
j_fx=
\textstyle\frac{1}{(r-m)!}\sum_{\sigma\in S_r}\sgn{\sigma}
\langle f_{1},x_{\sigma(r-m+1)}\rangle\cdots\langle f_m,x_{\sigma(r)}\rangle
\cdot x_{\sigma(1)}\smwedge\cdots\smwedge x_{\sigma(r-m)}\;.
\end{equation}
if $m\leq r$.
\end{proposition}
\begin{proof}Both equations are direct consequences of the defining identities
for the contraction operators (\ref{left contration}), (\ref{right contration}) 
and the natural pairing (\ref{determinant_eq_1}). To be more precise we use 
the following arguments:

To see that the right contraction vanishes for $r>m$ observe that the defining
equation $\langle i_xf,y\rangle=\langle f,x\smwedge y \rangle$ for all 
$y\in X(\mathfrak{g},A)$ implies $i_x f=0$, since either $|x\smwedge y|>|f|$,
or $x\smwedge y=0$ and the natural pairing is zero in both cases.

To show equation (\ref{right_contrac_1}) for any
$r\leq m$, let $y_1\smwedge\cdots\smwedge y_{m-r} \in X(\mathfrak{g},A)$ be an 
arbitrary simple tensor of degree $(m-r)$. Then substitute 
(\ref{right_contrac_1}) into the left side of (\ref{right contration}) to get
\begin{multline*}
\langle i_{x_1\smwedge\cdots\smwedge x_r} f_1\smwedge\cdots\smwedge f_m,
y_1\smwedge\cdots\smwedge y_{m-r} \rangle =\\
\textstyle\frac{1}{(m-r)!}\sum_{\sigma\in S_m}\sgn{\sigma}
\langle f_{\sigma(1)},x_1\rangle\cdots\langle f_{\sigma(r)},x_r\rangle\;\cdot\\
\cdot\langle f_{\sigma(r+1)}\smwedge\cdots\smwedge f_{\sigma(m)},
y_1\smwedge\cdots\smwedge y_{m-r} \rangle\;.
\end{multline*}
Then apply (\ref{determinant_eq_1}) to the most right pairing on the right 
side. This can be written as follows:
(According to a better readable text we write $S_{M}$ for the set of all
permutations defined explicit on the finite set $M$)
\begin{multline}\label{rigt_contract_eq_2}
\textstyle\frac{1}{(m-r)!}\sum_{\sigma\in S_m}
 \sum_{\tau \in S_{\{\sigma(r+1),\ldots,\sigma(m)\}}} 
  \sgn{\sigma}\sgn{\tau}\;\cdot\\
\cdot\langle f_{\sigma(1)},x_1\rangle\cdots\langle f_{\sigma(r)},x_r\rangle
 \cdot\langle f_{\tau(\sigma(r+1))},y_1\rangle\cdots\langle f_{\tau(\sigma(m))},
  y_{m-r}\rangle\;.
\end{multline}
Now observe, that for any $\sigma\in \mathbb{S}_m$ and 
$\tau\in \mathbb{S}_{\{\sigma(r+1),\ldots,\sigma(m)\}}$, the
permutation $(\sigma(1),\ldots, \sigma(r), 
\tau\sigma(r+1),\ldots,\tau\sigma(m))$ is again an element of 
$\mathbb{S}_m$ and since there are precisely $(m-r)!$ many permutations in 
$\mathbb{S}_{\{\sigma(r+1),\ldots,\sigma(m)\}}$, we can just 'absorb' the sum over 
$\mathbb{S}_{\{\sigma(r+1),\ldots,\sigma(m)\}}$ in (\ref{rigt_contract_eq_2}) into a 
sum over $\mathbb{S}_m$ and therefore (\ref{rigt_contract_eq_2}) is the same as
\begin{multline}\label{rigt_contract_eq_3}
\textstyle\frac{1}{(m-r)!}\sum_{\sigma\in S_m} 
  \sgn{\sigma}\langle f_{\sigma(1)},x_1\rangle\cdots\langle f_{\sigma(r)},x_r\rangle
 \cdot\langle f_{\sigma(r+1)},y_1\rangle\cdots\langle f_{\sigma(m)},
  y_{m-r}\rangle\;.
\end{multline}
Using the definition of the natural pairing (\ref{determinant_eq_1}), 
this transforms expression (\ref{rigt_contract_eq_3}) into
$$
\langle f_1\smwedge\cdots\smwedge f_m,x_1\smwedge\cdots\smwedge x_k 
 \smwedge y_1\smwedge\cdots\smwedge y_{m-r} \rangle
$$
and therefore proofs (\ref{right contration}) for arbitrary simple tensors 
$y_1\smwedge\cdots\smwedge y_{m-r}$ and hence for any tensor of degree
$(m-r)$, since (\ref{right contration}) is $A$-linear in both arguments and 
any such tensor is a sum of simple ones. For tensor $y\in X(\mathfrak{g},A)$
of degree other that $(m-r)$ bot sides of equation (\ref{right contration})
vanish and this proof equation (\ref{right_contrac_1}).

The proof for the second equation is analog, taking the 'left nature' of
the defining equation (\ref{left contration}) into account.

To see that the left contraction vanishes for $m>r$, observe that
$\langle g,j_f x\rangle=\langle g\smwedge f,x\rangle$ for all 
$g\in \Omega(\mathfrak{g},A)$ implies $j_f x=0$, since either 
$|g\smwedge f|>|x|$, or $g\smwedge f=0$ and the natural pairing 
is zero in both cases.

To see that the left contraction operator satisfies equation 
(\ref{left_contrac_1}) for any $r\leq m$, let
$g_1\smwedge\cdots\smwedge g_{r-m} \in \Omega(\mathfrak{g},A)$ be any arbitrary 
simple cotensor of degree $(r-m)$. Substitute (\ref{left_contrac_1}) into the left side of (\ref{left contration}) to get
\begin{multline*}
\langle g_1\smwedge\cdots\smwedge g_{r-m},
 j_{f_1\wedge\cdots\wedge f_m}x_1\smwedge\cdots\smwedge x_{r} \rangle =\\
\textstyle\frac{1}{(r-m)!}\sum_{\sigma\in \mathbb{S}_r}\sgn{\sigma}
\langle f_{1},x_{\sigma(r-m+1)}\rangle\cdots\langle f_m,x_{\sigma(r)}\rangle\;\cdot\\
\cdot\langle g_1\smwedge\cdots\smwedge g_{r-m}, 
 x_{\sigma(1)}\smwedge\cdots\smwedge x_{\sigma(r-m)}\rangle\;.
\end{multline*}
and apply the transposed identity (\ref{determinant_eq_2}) to the 
most right pairing on the
right side. This can be written as follows:
(According to a better readable text we write $\mathbb{S}_{M}$ for the set of 
all permutations defined explicit on the finite set $M$.)
\begin{multline*}
\textstyle\frac{1}{(r-m)!}\sum_{\sigma\in \mathbb{S}_r}
 \sum_{t \in \mathbb{S}_{\{\sigma(1),\ldots,\sigma(r-m)\}}} 
  \sgn{\sigma}\sgn{\tau}\;\cdot\\
\cdot\langle f_1,x_{\sigma(r-m+1)}\rangle\cdots\langle f_m,x_{\sigma(r)}\rangle
 \cdot\langle g_1,x_{\tau(\sigma(1))}\rangle\cdots
  \langle g_{r-m}, x_{\tau(\sigma(r-m))}\rangle\;.
\end{multline*}
Now observe, that for any $\sigma\in \mathbb{S}_r$ and 
$\tau\in \mathbb{S}_{\{\sigma(1),\ldots,\sigma(r-m)\}}$, the
permutation $(\tau(\sigma(1)),\ldots, \tau(\sigma(r-m)), 
\sigma(r-m+1),\ldots,\sigma(r))$ is again 
an element of $\mathbb{S}_r$ and since there are precisely $(r-m)!$ many
permutations in $\mathbb{S}_{\{\sigma(1),\ldots,\sigma(r-m)\}}$, we can just 
'absorb' the sum over 
$\mathbb{S}_{\{\sigma(1),\ldots,\sigma(r-m)\}}$ 
in the previous expression into the sum over $\mathbb{S}_r$. 
This gives
\begin{multline*}
\textstyle\sum_{\sigma\in \mathbb{S}_{r}} \sgn{\sigma}
\langle f_1,x_{\sigma(r-m+1)}\rangle\cdots\langle f_m,x_{\sigma(r)}\rangle
\langle g_1,x_{\sigma(1)}\rangle\cdots\langle g_{r-m},x_{\sigma(r-m)}\rangle=\\
\langle f_1\smwedge\cdots\smwedge f_m\smwedge 
 g_1\smwedge\cdots\smwedge g_{r-m}, 
 x_{r-m+1}\smwedge\cdots\smwedge x_{r} 
  \smwedge x_1\smwedge\cdots\smwedge x_{r-m} \rangle=\\
 (-1)^{m(r-m)+m(r-m)}\langle g_1\smwedge\cdots\smwedge g_{r-m}\smwedge 
 f_1\smwedge\cdots\smwedge f_m, x_1\smwedge\cdots\smwedge x_r \rangle\;.
\end{multline*}
This proofs (\ref{left contration}) for arbitrary simple cotensors 
$g_1\smwedge\cdots\smwedge g_{r-m}$ and hence for any cotensor of degree
$(r-m)$, since (\ref{left contration}) is $A$-linear in both arguments and 
any such cotensor is a sum of simple ones. For cotensor s
$g\in \Omega(\mathfrak{g},A)$ of degree other that $(r-m)$ both sides of 
equation (\ref{left contration}) vanish.
\end{proof}
With these precomputations we can now go on and proof equation 
(\ref{multi_rules-5}) by a long but
very basic combinatorial exercise.
\begin{proposition}Let $(A,\mathfrak{g})$ be a torsionless Lie Rinehart pair,
$x,y\in X(\mathfrak{g},A)$ homogeneous tensors and
$f\in \Omega(\mathfrak{g},A)$ a homogeneous cotensor. Then
\begin{equation}\label{eq_contraction}
i_{j_{(i_y f)} x}f = i_{y}f \smwedge i_xf\;.
\end{equation}
\end{proposition}
\begin{proof}To see that the equation is well defined, it is enough to show
that both sides are homogeneous 
of the same tensor degree. Therefore recall that we consider tensors as 
concentrated in positive degrees and cotensors as concentrated in negative degrees. Then 
\begin{align*}
|i_{j_{(i_y f)}x}f| 
&= |f|+ |j_{i_y f} x|\\
&= |f|+ \left(|x|+ |i_y f|\right)\\ 
&= \left( |f| + |y|\right) + \left(|f| + |x|\right)\\ 
&= |i_y f|+|i_x f|\\
&= |i_y f \smwedge i_x f|\;.
\end{align*} 
To proof the equation it is enough to consider simple (co)tensors only. 
The general situation then follows since all involved operations
are $A$-linear in all arguments and since any tensor is a (not necessarily unique) 
sum of simple tensors. 

We write $\langle x,f\rangle$ simultaneously for the natural
pairing, the left contraction $i_xf$, or the right contraction $j_fx$, if 
$|x|\leq 1$ and if $-|f|\leq 1$. In 
particular this means $\langle x,f\rangle= x\cdot f$, if at least one argument
is an element of $A$ (has degree zero). This reduces the overall complexity of
the proof, since we don't have to distinguish cases where at least one argument is 
an element of $A$.

To proceed let $r,k,m\in\N$ be natural numbers and 
for all $1\leq i \leq r$, $1\leq j \leq k$ and $1\leq h\leq m$ let
$x_i,y_j\in A \cup \mathfrak{g}$ be any scalar or vector and
let $f_h\in A\cup \mathfrak{g}^\vee $ be any scalar or covector. 
Therefore $x_1\smwedge\cdots\smwedge x_r$
and $y_1\smwedge\cdots\smwedge y_k$ are simple tensors and
$f_1\smwedge\cdots\smwedge f_m$ is a simple cotensor.

We first look at the right side of equation (\ref{eq_contraction}).
Due to (\ref{right_contrac_1}) it vanishes whenever $k>m$ or $r>m$ and
for $k\leq m$ and $r\leq m$ we can rewrite it into the more explicit form
\begin{multline}\label{right_side}
\textstyle( i_{y_1\wedge\cdots\wedge y_k}f_1\smwedge\cdots\smwedge f_m)\smwedge
( i_{x_1\wedge\cdots\wedge x_r}f_1\smwedge\cdots\smwedge f_m)=\\
\textstyle(\frac{1}{(m-k)!}\sum_{\sigma\in S_m}\sgn{\sigma}
 \langle f_{\sigma(1)},y_1\rangle\cdots\langle f_{\sigma(k)},y_k\rangle
  f_{\sigma(k+1)}\smwedge\cdots\smwedge f_{\sigma(m)})\smwedge\\
\textstyle(\frac{1}{(m-r)!}\sum_{t\in S_m}\sgn{\tau}
 \langle f_{\tau(1)},x_1\rangle\cdots\langle f_{\tau(r)},x_r\rangle
  f_{\tau(r+1)}\smwedge\cdots\smwedge f_{\tau(m)})\;.
\end{multline}
Due to the antisymmetry of the exterior product, this vanishes if $r+k<m$. 
To see that observe
$\{\sigma(k+1),\ldots,\sigma(m)\}\cap \{\tau(r+1),\ldots,r(m)\} \neq \emptyset$,
for $r+k<m$ and any permutations $\sigma,\tau\in S_m$. Hence
$
(f_{\sigma(k+1)}\smwedge\cdots\smwedge f_{\sigma(m)})\smwedge 
(f_{\tau(r+1)}\smwedge\cdots\smwedge f_{\tau(m)})=0\;.
$

In conclusion we can say that the right side of the equation is only nontrivial
in a narrow range of tensor degrees, namely for $r+k\geq m$ and $r\leq m$ as
well as $k\leq m$. 

Now suppose $r+k\geq m$, $r\leq m$ and $k\leq m$. Since the product in 
$A$ is commutative we rewrite (\ref{right_side}) further:
\begin{multline*}
\textstyle\frac{1}{(m-k)!(m-r)!r!}
 \sum_{\sigma,\upsilon\in S_m}\sum_{\tau\in S_r}
 \sgn{\sigma + \upsilon}\;\cdot\\
  \cdot\langle f_{\sigma(1)},y_1\rangle\cdots\langle f_{\sigma(k)},y_k\rangle
   \langle f_{\upsilon(\tau(1))},x_{\tau(1)}\rangle\cdots\langle 
    f_{\upsilon(\tau(r))},x_{\tau(r)}\rangle\;\cdot\\
\cdot f_{\sigma(k+1)}\smwedge\cdots\smwedge f_{\sigma(m)}\smwedge
 f_{\upsilon(r+1)}\smwedge\cdots\smwedge f_{\upsilon(m)}=
\end{multline*}
\begin{multline}\label{right_side-2}
\textstyle\frac{1}{(m-k)!(m-r)!r!}
 \sum_{\sigma,\upsilon\in S_m}\sum_{\tau\in S_r}
 \sgn{\sigma + \upsilon + \tau}\;\cdot\\
  \cdot\langle f_{\sigma(1)},y_1\rangle\cdots\langle f_{\sigma(k)},y_k\rangle
   \langle f_{\upsilon(1)},x_{\tau(1)}\rangle\cdots\langle 
    f_{\upsilon(r)},x_{\tau(r)}\rangle\;\cdot\\
\cdot f_{\sigma(k+1)}\smwedge\cdots\smwedge f_{\sigma(m)}\smwedge
 f_{\upsilon(r+1)}\smwedge\cdots\smwedge f_{\upsilon(m)}\;.
\end{multline}
Due to the antisymmetry of the exterior product only terms
for those permutations $\sigma,\upsilon\in S_m$ 
with $\{\sigma(k+1),\ldots,\sigma(m)\}\cap
\{\upsilon(r+1),\ldots,\upsilon(m)\}=\emptyset$ are non zero here
and since we assume $(m-k)\leq r$ we can restrict the summation in 
(\ref{right_side-2}) to those $\sigma,\upsilon \in S_m$ satisfying
$\{\sigma(k+1),\ldots,\sigma(m)\}\subset 
\{\upsilon(1),\ldots,\upsilon(r)\}$. Writing
$$I_1(m,k,r)=\{(\sigma,\upsilon)\in S_m \times S_m\;|\; 
\{\sigma(k+1),\ldots,\sigma(m)\}\subset\{\upsilon(1),\ldots,\upsilon(r)\}\}$$
for the appropriate index set of restricted permutations,
(\ref{right_side-2}) rewrites into
\begin{multline}\label{right_side-3}
\textstyle\frac{1}{(m-k)!(m-r)!r!}
 \sum_{\sigma,\upsilon\in I_1(m,k,r)}
  \sum_{\tau\in S_r}\sgn{\sigma+\upsilon+\tau}\cdot{}\\
  \cdot\langle f_{\sigma(1)},y_1\rangle\cdots\langle f_{\sigma(k)},y_k\rangle
   \langle f_{\upsilon(1)},x_{\tau(1)}\rangle\cdots\langle 
     f_{\upsilon(r)},x_{\tau(r)}\rangle\;\cdot\\
\cdot f_{\sigma(k+1)}\smwedge\cdots\smwedge f_{\sigma(m)}\smwedge
 f_{\upsilon(r+1)}\smwedge\cdots\smwedge f_{\upsilon(m)}\;.
\end{multline}
To simplify this further, observe that for any
$(\sigma,\upsilon)\in I_1(m,k,r)$ there is a unique permutation,  
$\phi\in S_r$, which 'unshuffles $\sigma(k+1),\ldots,\sigma(m)$ in
$\{\upsilon(1),\ldots,\upsilon(r)\}$ to the left', i.e. which acts by
\begin{multline*}
\upsilon\circ (\phi\times id_{m-r})(\{1,\ldots,m\})=\\
\{\sigma(k+1),\ldots, \sigma(m),
\upsilon\phi(m-k+1),\ldots,\upsilon\phi(r),\upsilon(r+1),\ldots,\upsilon(m)\}\;,
\end{multline*}
where $id_{m-r}\in S_{m-r}$ is the identity permutation, and where the order 
in the (ordered) set $\{\upsilon(1),\ldots,\upsilon(r)\}
\backslash\{\sigma(k+1),\ldots,\sigma(m)\}$ 
is not chanced by $\phi$. Using this we can 
use the commutativity of the multiplication in $A$, to reorder the
scalar factors. In particular we get
\begin{multline*}
\textstyle{\sum_{\tau\in S_r}}\sgn{\sigma+\upsilon+\tau}
 \langle f_{\sigma(1)},y_1\rangle\cdots\langle f_{\sigma(k)},y_k\rangle
  \langle f_{\upsilon(1)},x_{\tau(1)}\rangle\cdots\langle 
     f_{\upsilon(r)},x_{\tau(r)}\rangle\;\cdot\\
\cdot f_{\sigma(k+1)}\smwedge\cdots\smwedge f_{\sigma(m)}\smwedge
 f_{\upsilon(r+1)}\smwedge\cdots\smwedge f_{\upsilon(m)}=
\end{multline*}
\begin{multline*}
\textstyle{\sum_{\tau\in S_r}}\sgn{\sigma+
\upsilon(\phi\times id_{m-r})+\tau\phi}\\
 \langle f_{\sigma(1)},y_1\rangle\cdots\langle f_{\sigma(k)},y_k\rangle
  \langle f_{\upsilon\phi(1)},x_{\tau\phi(1)}\rangle\cdots\langle 
     f_{\upsilon\phi(r)},x_{\tau\phi(r)}\rangle\;\cdot\\
\cdot f_{\sigma(k+1)}\smwedge\cdots\smwedge f_{\sigma(m)}\smwedge
 f_{\upsilon(r+1)}\smwedge\cdots\smwedge f_{\upsilon(m)}=
\end{multline*}
\begin{multline*}
\textstyle{\sum_{\tau\in S_r}}\sgn{\sigma+\upsilon+\tau}
 \langle f_{\sigma(1)},y_1\rangle\cdots\langle f_{\sigma(k)},y_k\rangle
  \langle f_{\sigma(k+1)},x_{\tau(1)}\rangle\cdots\langle 
     f_{\sigma(m)},x_{\tau(m-k)}\rangle\cdot\\
     \langle f_{\upsilon\phi(m-k+1)},x_{\tau(m-k+1)}\rangle\cdots\langle 
     f_{\upsilon\phi(r)},x_{\tau(r)}\rangle\cdot{}\\
\cdot f_{\sigma(k+1)}\smwedge\cdots\smwedge f_{\sigma(m)}\smwedge
 f_{\upsilon(r+1)}\smwedge\cdots\smwedge f_{\upsilon(m)}
\end{multline*}
The number of ways  
$\{\sigma(k+1),\ldots,\sigma(m)\}$ can be a subset of 
$\{\upsilon(1),\ldots,\upsilon(r)\}$ is the
same as counting the number of ways to put $(m-k)$ labeled balls into 
$r$ labeled slots with exclusion, which is computed by the
\textit{falling factorial}
$$
P(r,m-k)=r\cdot(r-1)\cdots (r-(m-k)+1)=\textstyle\frac{r!}{(r+k-m)!}
$$
and if we define another index set 
$$J_1(m,k,r)=
\{(\sigma,\upsilon)\in S_m\times S_m\;|\;
\upsilon(1)=\sigma(k+1),\ldots, \upsilon(m-k)=\sigma(m)\}$$
we finally can rewrite (\ref{right_side-3}) into
\begin{multline}\label{right_side_4}
\textstyle\frac{1}{(m-k)!(m-r)!(r+k-m)!}
 \sum_{\sigma,\upsilon\in J_1(m,r,k)}
  \sum_{\tau\in S_r}\sgn{\sigma+\upsilon+\tau}\cdot{}\\
 \langle f_{\sigma(1)},y_1\rangle\cdots\langle f_{\sigma(k)},y_k\rangle
  \langle f_{\sigma(k+1)},x_{\tau(1)}\rangle\cdots\langle 
     f_{\sigma(m)},x_{\tau(m-k)}\rangle\cdot\\
     \langle f_{\upsilon(m-k+1)},x_{\tau(m-k+1)}\rangle\cdots\langle 
     f_{\upsilon(r)},x_{\tau(r)}\rangle\cdot{}\\
\cdot f_{\sigma(k+1)}\smwedge\cdots\smwedge f_{\sigma(m)}\smwedge
 f_{\upsilon(r+1)}\smwedge\cdots\smwedge f_{\upsilon(m)}\;.
\end{multline}
In the next step we transform the left side of equation 
(\ref{eq_contraction}), exactly into (\ref{right_side_4}). We start by 
making the nested contraction $i_{j_{i_y f}x}f$ explicit using (\ref{right_contrac_1}) and (\ref{left_contrac_1}).
For $k>m$, the contraction $i_yf$ vanishes and we assume $k\leq m$ to compute  
\begin{multline}\label{eq_1}
i_{y_{1}\wedge\cdots\wedge y_{k}}
 f_{1}\smwedge\cdots\smwedge f_{m}=\\
\textstyle \frac{1}{(m-k)!}\sum_{\sigma\in S_{m}}\sgn{\sigma}
 \langle f_{\sigma(1)},y_{1}\rangle\cdots\langle f_{\sigma(k)},y_{k}\rangle
  f_{\sigma(k+1)}\smwedge\cdots\smwedge f_{\sigma(m)}\;.
\end{multline}
For $r+k<m$, the double contraction $j_{i_y f}x$ is zero and therefore we assume 
$r+k\geq m$ to compute
\begin{multline*}
j_{(i_{y_{1}\wedge\cdots\wedge y_{k}}
 f_{1}\wedge\cdots\wedge f_{m})}
  x_{1}\smwedge\cdots\smwedge x_{r}=\\
\textstyle\frac{1}{(m-k)!}\sum_{\sigma\in S_{m}}\sgn{\sigma}
 \langle f_{\sigma(1)},y_{1}\rangle\cdots\langle f_{\sigma(k)},y_{k}\rangle
  j_{f_{\sigma(k+1)}\wedge\cdots\wedge f_{\sigma(m)}}x_{1}\smwedge\cdots\smwedge x_{r}=
\end{multline*}
\begin{multline*}
\textstyle\frac{1}{(m-k)!(r+k-m)!}
 \sum_{\sigma\in S_{m}}\sum_{\tau\in S_{r}}\sgn{\sigma+\tau}\langle f_{\sigma(1)},y_{1}\rangle\cdots
  \langle f_{\sigma(k)},y_{k}\rangle\;\cdot\\
   \cdot\langle f_{\sigma(k+1)},x_{\tau(r+k-m+1)}\rangle\cdots\langle f_{\sigma(m)},x_{\tau(r)}\rangle 
    x_{\tau(1)}\smwedge\cdots\smwedge x_{\tau(r+k-m)}\;.
\end{multline*}
For $r+k>2m$, the contraction $i_{j_{(i_y f)}x}f$ is zero and therefore we assume 
$r+k\leq 2m$ to compute
\begin{comment}

\begin{multline*}
\textstyle\frac{1}{(m-k)!(r+k-m)!}
 \sum_{\sigma\in S_{m}}\sum_{t\in S_{r}}\sgn{\sigma}\sgn{\tau}\langle f_{\sigma(1)},y_{1}\rangle\cdots
  \langle f_{\sigma(k)},y_{k}\rangle\;\cdot\\
\cdot\langle f_{\sigma(k+1)},x_{\tau(r+k-m+1)}\rangle\cdots\langle f_{\sigma(m)},x_{\tau(r)}\rangle 
 i_{x_{\tau(1)}\wedge\cdots\wedge x_{\tau(r+k-m)}}f_{1}\smwedge\cdots\smwedge f_{m}=
\end{multline*}

\end{comment}
\begin{multline}\label{left_side-1}
\textstyle\frac{1}{(m-k)!(r+k-m)!(2m-r-k)!}
 \sum_{\sigma,\upsilon\in S_{m}}\sum_{\tau\in S_{r}}
 \sgn{\sigma+\tau+\upsilon}\;\cdot\\
\cdot\langle f_{\sigma(1)},y_{1}\rangle\cdots\langle f_{\sigma(k)},y_{k}\rangle\cdot
 \langle f_{\sigma(k+1)},x_{\tau(r+k-m+1)}\rangle\cdots 
  \langle f_{\sigma(m)},x_{\tau(r)}\rangle\;\cdot\\
\cdot\langle f_{\upsilon(1)},x_{\tau(1)}\rangle
 \cdots\langle f_{\upsilon(r+k-m)},x_{\tau(r+k-m)}\rangle 
  f_{\upsilon(r+k-m+1)}\smwedge\cdots\smwedge f_{\upsilon(m)}\;.
\end{multline}
To simplify this, we show that
for arbitrary but fixed permutations $\sigma,\upsilon\in S_m$ with 
$\{\sigma(k+1),\ldots,\sigma(m)\}\cap
\{\upsilon(1),\ldots,\upsilon(r+k-m)\} \neq \emptyset$
the summation over all permutations from $S_r$ vanishes in 
(\ref{left_side-1}).

To see that, let $\sigma,\upsilon\in S_m$ with 
$\{\sigma(k+1),\ldots,\sigma(m)\}\cap
\{\upsilon(1),\ldots,\upsilon(r+k-m)\} \neq \emptyset$.
Then there is a $k+1\leq j \leq m$ and a $1\leq i \leq r-m+k$, such that 
$\sigma(j)=\upsilon(i)$ and for any permutation $\tau\in S_r$ precisely one other 
permutation $\tau'\in S_r$, which is equal to $\tau$ but with 
$\tau(i)$ and $\tau(r-m+j)$ transposed. Moreover,
since $i\neq r-m+j$ this transposition is not the identity.

Then
$\langle f_{\sigma(j)},x_{\tau(r-m+j)}\rangle=
\langle f_{\upsilon(i)},x_{\tau'(i)}\rangle$ 
and $\langle f_{\sigma(j)},x_{\tau'(r-m+j)}\rangle=
\langle f_{\upsilon(i)},x_{\tau(i)}\rangle$
and since $\sgn{\tau'}=-\sgn{\tau}$, the term for $\tau$ cancel against the 
term for $\tau'$.
 
Consequently we can restrict the summation in expression (\ref{left_side-1})
to those $\sigma,\upsilon \in S_m$ satisfying
$\{\sigma(k+1),\ldots,\sigma(m)\}\subset
\{\upsilon(r+k-m+1),\ldots,\upsilon(m)\}$. 

From this follows that the left side is non zero only if
$m-k\leq 2m-r-k$, which means $r\leq m$. Summarizing this,
we can say that the left side of equation () is only nontrivial
in a narrow range of tensor degrees, namely for $r + k \geq m$ 
and $r \leq m$ as well as $k \leq m$, which coincides with the triviality of
the right side.

To proceed we assume, $r + k \geq m$ and $r \leq m$ as well as $k \leq m$
and write
$$I_2(m,k,r)=\{(\sigma,\upsilon)\in S_m \times S_m\;|\; 
\{\sigma(k+1),\ldots,\sigma(m)\}\subset\{\upsilon(r+k-m+1),\ldots,\upsilon(m)\}\}$$
for the appropriate set of restricted permutations, 
(\ref{left_side-1}) can equivalently be written as
\begin{multline}\label{left_side-2}
\textstyle\frac{1}{(m-k)!(r+k-m)!(2m-r-k)!}
 \sum_{\sigma,\upsilon\in I_2(m,k,r)}\sum_{\tau\in S_{r}}
 \sgn{\sigma+\tau+\upsilon}\;\cdot\\
\cdot\langle f_{\sigma(1)},y_{1}\rangle\cdots\langle f_{\sigma(k)},y_{k}\rangle\cdot
 \langle f_{\sigma(k+1)},x_{\tau(r+k-m+1)}\rangle\cdots 
  \langle f_{\sigma(m)},x_{\tau(r)}\rangle\;\cdot\\
\cdot\langle f_{\upsilon(1)},x_{\tau(1)}\rangle
 \cdots\langle f_{\upsilon(r+k-m)},x_{\tau(r+k-m)}\rangle 
  f_{\upsilon(r+k-m+1)}\smwedge\cdots\smwedge f_{\upsilon(m)}\;.
\end{multline}
To simplify this further, observe that for any
$(\sigma,\upsilon)\in I_2(m,k,r)$ there is a unique $\phi\in S_{2m-r-k}$,
which 'unshuffles $\sigma(k+1),\ldots, \sigma(m)$ in 
$\{\upsilon(r+k-m+1),\ldots, \upsilon(m)\}$ to the left', but leaves the
order in the (ordered) set 
$\{\upsilon(r+k-m+1),\ldots,\upsilon(m)\} \backslash 
\{\sigma(k+1),\ldots,\sigma(m)\}$ unchanged. In particular it acts by
\begin{multline*}
(id_{r+k-m}\times\phi)\upsilon(\{1,\ldots,m\})=\\
\{\upsilon(1),\ldots,\upsilon(r+k-m),
\sigma(k+1),\ldots,\sigma(m),\phi\upsilon(r+1),\ldots,\phi\upsilon(m)\}
\end{multline*}
Writing 
$$J_2(m,k,r)=\{(s,u)\in S_m\times S_m\;|\;
\upsilon(r+k-m+1)=\sigma(k+1),\ldots, \upsilon(r)=\sigma(m)\}$$
this 'unshuffling' gives a surjection of sets 
$\pi: I_2(m,k,r)\to J_2(m,k,r)$, with 
fiber over $\pi^{-1}(\sigma,\upsilon)$ the set of all 
$(\sigma,\delta)\in J_2(m,k,r)$,
such that the $\sigma(k+1),\ldots, \sigma(m)$ are somehow 
inside $\{\upsilon(r+k-m+1),\ldots, \upsilon(m)\}$.

The number of ways  
$\{\sigma(k+1),\ldots,\sigma(m)\}$ can be a subset of the set
$\{u(r+k-m+1),\ldots,u(m)\}$, is the
same as counting the number of ways to put $(m-k)$ labeled balls into 
$2m-(r+k)$ labeled slots with exclusion, which is computed by the
\textit{falling factorial}
$$
P(2m-(r+k),m-k)=
(2m-(r+k))\cdot(2m-(r+k)-1)\cdots (m-r+1)=
\textstyle\frac{(2m-(r+k))!}{(m-r)!}\;.
$$
Now the point is, that all terms in (\ref{left_side-2}), indexed by elements of 
the same fiber are in fact equal. To see that, 
let $(\sigma,\delta)\in \pi^{-1}(\sigma,\upsilon)$ for some 
$(\sigma,\upsilon)\in J_2(m,k,r)$. Starting with
\begin{multline*}
\sgn{\sigma+\upsilon+\tau}\cdot
 \langle f_{\sigma(1)},y_{1}\rangle\cdots\langle f_{\sigma(k)},y_{k}\rangle\cdot
 \langle f_{\sigma(k+1)},x_{\tau(r+k-m+1)}\rangle\cdots 
  \langle f_{\sigma(m)},x_{\tau(r)}\rangle\;\cdot\\
\cdot\langle f_{\upsilon(1)},x_{\tau(1)}\rangle
 \cdots\langle f_{\upsilon(r+k-m)},x_{\tau(r+k-m)}\rangle\;\cdot\\
  \cdot f_{\sigma(k+1)}\smwedge\cdots\smwedge f_{\sigma(m)}\smwedge
   f_{\upsilon(r+1)}\smwedge\cdots\smwedge f_{\upsilon(m)}
\end{multline*}
observe that there is a unique 
permutation $\phi\in S_{2m-r-k}$ with
$(id_{r+k-m}\times \phi)\circ \delta =\upsilon$. We can reorder the cotensor 
factors in the previous expression according to the inverse of $\phi$.
 Since the exterior product is antisymmetric we get
\begin{multline*}
\sgn{\sigma+\upsilon+\tau +\phi^{-1}}\cdot
 \langle f_{\sigma(1)},y_{1}\rangle\cdots\langle f_{\sigma(k)},y_{k}\rangle\cdot
 \langle f_{\sigma(k+1)},x_{\tau(r+k-m+1)}\rangle\cdots 
  \langle f_{\sigma(m)},x_{\tau(r)}\rangle\;\cdot\\
\cdot\langle f_{\upsilon(1)},x_{\tau(1)}\rangle
 \cdots\langle f_{\upsilon(r+k-m)},x_{\tau(r+k-m)}\rangle\;\cdot\\
  \cdot f_{\phi^{-1}\sigma(k+1)}\smwedge\cdots\smwedge f_{\phi^{-1}\sigma(m)}
   \smwedge f_{\phi^{-1}\upsilon(r+1)}\smwedge\cdots\smwedge
    f_{\phi^{-1}\upsilon(m)}
\end{multline*}
which we can write as
\begin{multline*}
\sgn{\sigma+(id_{r+k-m}\times \phi)^{-1}\upsilon+\tau }\cdot
 \langle f_{\sigma(1)},y_{1}\rangle\cdots\langle f_{\sigma(k)},y_{k}\rangle\cdot{}\\
 \langle f_{\sigma(k+1)},x_{\tau(r+k-m+1)}\rangle\cdots 
  \langle f_{\sigma(m)},x_{\tau(r)}\rangle
\cdot\langle f_{\upsilon(1)},x_{\tau(1)}\rangle
 \cdots\langle f_{\upsilon(r+k-m)},x_{\tau(r+k-m)}\rangle\;\cdot\\
  \cdot f_{\phi^{-1}\sigma(k+1)}\smwedge\cdots\smwedge f_{\phi^{-1}\sigma(m)}
   \smwedge f_{\phi^{-1}\upsilon(r+1)}\smwedge\cdots\smwedge
    f_{\phi^{-1}\upsilon(m)}=
\end{multline*}
\begin{multline*}
\sgn{\sigma+\delta+\tau}\cdot
 \langle f_{\sigma(1)},y_{1}\rangle\cdots\langle f_{\sigma(k)},y_{k}\rangle\cdot
 \langle f_{\sigma(k+1)},x_{\tau(r+k-m+1)}\rangle\cdots 
  \langle f_{\sigma(m)},x_{\tau(r)}\rangle\;\cdot\\
\cdot\langle f_{\delta(1)},x_{\tau(1)}\rangle
 \cdots\langle f_{\delta(r+k-m)},x_{\tau(r+k-m)}\rangle 
  f_{\delta(r+k-m+1)}\smwedge\cdots\smwedge f_{\delta(m)}\;.
\end{multline*}
Since we can do this for any element of the same fiber, they really
describe the same term in (\ref{left_side-2}).

Using this we can rewrite (\ref{left_side-2}) as a sum over the index
set $J_2(m,r,k)$. This gives
\begin{multline*}
\textstyle\frac{1}{(m-k)!(m-r)!(r+k-m)!}
 \sum_{\sigma,\upsilon\in J_2(m,k,r)}
  \sum_{\tau\in S_{r}}\sgn{\sigma+\upsilon+\tau}\;\cdot\\
\cdot\langle f_{\sigma(1)},y_{1}\rangle\cdots\langle 
 f_{\sigma(k)},y_{k}\rangle\cdot
 \langle f_{\sigma(k+1)},x_{\tau(r+k-m+1)}\rangle\cdots 
  \langle f_{\sigma(m)},x_{\tau(r)}\rangle\;\cdot\\
\cdot\langle f_{\upsilon(1)},x_{\tau(1)}\rangle
 \cdots\langle f_{\upsilon(r+k-m)},x_{\tau(r+k-m)}\rangle \\
  f_{\sigma(k+1)}\smwedge\cdots\smwedge f_{\sigma(m)}\smwedge
   f_{\upsilon(r+1)}\smwedge\cdots\smwedge f_{\upsilon(m)}\;.
\end{multline*}
And by reordering the summations over $\tau$ we can rewrite this further into 
\begin{multline}\label{left_side-4}
\textstyle\frac{1}{(m-k)!(m-r)!(r+k-m)!}
 \sum_{\sigma,\upsilon\in J_2(m,k,r)}
  \sum_{\tau\in S_{r}}\sgn{\sigma+\upsilon+\tau+(m-k)(r+k-m)}\;\cdot\\
\cdot\langle f_{\sigma(1)},y_{1}\rangle\cdots\langle 
 f_{\sigma(k)},y_{k}\rangle\cdot
 \langle f_{\sigma(k+1)},x_{\tau(1)}\rangle\cdots 
  \langle f_{\sigma(m)},x_{\tau(m-k)}\rangle\;\cdot\\
\cdot\langle f_{\upsilon(1)},x_{\tau(m-k+1)}\rangle
 \cdots\langle f_{\upsilon(r+k-m)},x_{\tau(r)}\rangle \\
  f_{\sigma(k+1)}\smwedge\cdots\smwedge f_{\sigma(m)}\smwedge
   f_{\upsilon(r+1)}\smwedge\cdots\smwedge f_{\upsilon(m)}\;.
\end{multline}
To compare this expression with (\ref{right_side_4}), we have to change to a
summation over the index set $J_1(m,r,k)$. Therefore consider the map 
$\Phi: J_1(m,r,k)\to J_2(m,r,k)$ defined by 
$\Phi(\sigma,\upsilon)=(\sigma,\upsilon\kappa)$
where $\kappa\in S_m$ is given by the parmutation that maps
$(r+k-m+1,\ldots,r,1,\ldots,r+k-m,r+1,\ldots,m)$ to $(1,\ldots,m)$. This permutation
can be realized by $(m-k)(r+k-m)$ transposition and 
since $|J_1(m,r,k)| = |J_2(m,r,k)|$ the map $\Phi
$ is a bijection. In conclusion (\ref{right_side_4})
can be written as
\begin{multline}\label{left_side-5}
\textstyle\frac{1}{(m-k)!(m-r)!(r+k-m)!}
 \sum_{\sigma,\upsilon\in J_2(m,k,r)}
  \sum_{\tau\in S_{r}}\sgn{\sigma+\upsilon+\tau}\;\cdot\\
\cdot\langle f_{\sigma(1)},y_{1}\rangle\cdots\langle 
 f_{\sigma(k)},y_{k}\rangle\cdot
 \langle f_{\sigma(k+1)},x_{\tau(1)}\rangle\cdots 
  \langle f_{\sigma(m)},x_{\tau(m-k)}\rangle\;\cdot\\
\cdot\langle f_{\upsilon(m-k+1)},x_{\tau(m-k+1)}\rangle
 \cdots\langle f_{\upsilon(r)},x_{\tau(r)}\rangle \\
  f_{\sigma(k+1)}\smwedge\cdots\smwedge f_{\sigma(m)}\smwedge
   f_{\upsilon(r+1)}\smwedge\cdots\smwedge f_{\upsilon(m)}\;.
\end{multline}
which is to to the right side (\ref{right_side_4}), which 
proofs equation () for simple and therefore on arbitrary
arguments.
\end{proof}
\begin{corollary}
Equation (\ref{multi_rules-5}) is only non zero for 
$|x|,|y|\leq |f|$ and $|x|+|y|\geq |f|$.
\end{corollary}
\section{The n-plectic homotopy Jacobi equation}\label{homotopy_Lie_algebra}
In this section, we proof the
\textbf{n-plectic homotopy Jacobi equation} in dimension $k$:
\begin{multline}
\textstyle\sum_{j=1}^k\sum_{\sigma\in Sh(j,k-j)}\sgn{\sigma+j(k-j)}
e(\sigma;sx_1,\ldots,sx_k)\cdot{}\\\label{shJ_1}
\cdot\{\{v_{\sigma(1)},\ldots, v_{\sigma(j)}\},
 v_{\sigma(j+1)},\ldots, v_{\sigma(k)}\}=0\;.
\end{multline}
Here the appearance of the $(n-1)$-fold shifting is due to the degree relation
$|s^{n-1}v|=|sx|$ between Poisson cotensors and associated Hamilton tensors.

We proof these equations by induction on $k$. Sine () and (),
they are satisfied for all $k\leq 3$, which serves as the induction base.

In a nutshell the proof is roughly as follows: 

The following proposition gives the key observation, necessary for
the  induction to work. Loosely speaking it is the reflection of the 
structure equation on the
level of associated Hamilton tensors.
\begin{proposition}\label{_mvf_sh_J}Suppose $k\in\N$ and that 
the $n$-plectic Jacobi equation (\ref{shJ_1}) holds for all $j\leq(k-1)$. 
Then the exterior tensor
\begin{multline*}
\textstyle\sum_{j=1}^{k-2}\sum_{\sigma\in Sh(j,(k-1)-j)}\sgn{\sigma+j((k-1)-j)}
 e(s;sx_1,\ldots,sx_{k-1})\cdot{}\\
\cdot x_{\{\{v_{\sigma(1)},\ldots,v_{\sigma(j)}\},
  v_{\sigma(j+1)},\ldots,v_{\sigma(k-1)}\}}		
\end{multline*}
is an element of the kernel $\ker(\omega)$ for all Poisson cotensors 
$v_1,\ldots,v_k\in V$ and associated Hamilton tensors $x_1,\ldots,x_k$.
\end{proposition}
\begin{proof}
Apply the de Rham differential $d$ to the $n$-plectic Jacobi 
equation (\ref{shJ_1}) in dimension $(k-1)$. Since $d^2=0$ we get 
\begin{multline*}
\textstyle\sum_{j=1}^{k-2}
 \sum_{\sigma\in Sh(j,(k-1)-j)}\sgn{\sigma+j((k-1)-j)}
  e(\sigma;sx_1,\ldots,sx_{k-1})\cdot{}\\
\cdot d\{\{v_{\sigma(1)},\ldots,v_{\sigma(j)}\},
    v_{\sigma(j+1)},\ldots,v_{\sigma(k-1)}\}=0\;.
\end{multline*}
From () we know that if the homotopy Jacobi equation holds in dimension $j$, then
the images of $j$-ary Poisson bracket are Poisson cotensors. 
Therefore we can use the fundamental pairing () to find Hamilton tensors
associated to the nested brackets and transform 
the last equation into
\begin{multline*}
\textstyle\sum_{j=1}^{k-2}
 \sum_{\sigma\in Sh(j,(k-1)-j)}\sgn{\sigma+j((k-1)-j)}
  e(\sigma;sx_1,\ldots,sx_{k-1})\cdot{}\\
\cdot i_{x_{\{\{v_{\sigma(1)},\ldots,v_{\sigma(j)}\},
    v_{\sigma(j+1)},\ldots,v_{\sigma(k-1)}\}}}\omega=0\;.
\end{multline*}
\end{proof}
The following theorem computes the induction step. 
\begin{theorem}[Induction Step]Let $k\in\N$ be any natural number and
assume that the $n$-plectic Jacobi equations (\ref{shJ_1})
hold for all $j\leq (k-1)$. Then the same equation is
also satisfied for $k$.
\end{theorem}
\begin{proof}
Consider the general $n$-plectic Jacobi equation 
(\ref{shJ_1}), assume that each argument $v_i\in V$ is 
homogeneous and choose appropriate Hamilton tensors $x_i$.

Since all $k$-ary brackets are $(n-1)$-fold shifted graded antisymmetric, we can 
rewrite equation (\ref{shJ_1}) as a sum over arbitrary permutations:
\begin{multline}
\textstyle\frac{1}{j!(k-j)!}\sum_{j=1}^k\sum_{\sigma\in S_k}\sgn{\sigma+j(k-j)}
e(\sigma;sx_1,\ldots,sx_k)\cdot{}\\\label{shJ_2}
\cdot\{\{v_{\sigma(1)},\ldots, v_{\sigma(j)}\},
 v_{\sigma(j+1)},\ldots, v_{\sigma(k)}\}=0\;.
\end{multline}
Now recall that the homotopy Poisson $1$-bracket $\{\cdot\}$ is nothing
but the de Rham differential. To exploit this, rephrase
(\ref{shJ_2}) by separating all terms, which contain the $1$-bracket:
\begin{align}\label{shJ_3_1}
&d\{v_{1},\ldots, v_{k}\}\\\label{shJ_3_2}
&+\textstyle\frac{1}{(k-1)!}\sum_{\sigma\in S_k}\sgn{\sigma+(k-1)}
e(\sigma;sx_1,\ldots,sx_k)
\{dv_{\sigma(1)},v_{\sigma(2)},\ldots, v_{\sigma(k)}\}\\\label{shJ_3_3}
&\!\begin{multlined}
+\textstyle\frac{1}{j!(k-j)!}\sum_{j=2}^{k-1}
 \sum_{\sigma\in S_k}\sgn{\sigma+j(k-j)}
e(\sigma;sx_1,\ldots,sx_k)\cdot{}\\
\phantom{mmmmmmmmmmmmmm}\cdot\{\{v_{\sigma(1)},\ldots, v_{\sigma(j)}\},
 v_{\sigma(j+1)},\ldots, v_{\sigma(k)}\}=0\;.
\end{multlined} 
\end{align}
According to () as well as (), we know that this equation is satisfied 
for all $k\leq 3$ and therefore we assume $k>3$. In this case 
the homotopy Poisson $k$-bracket () is given by recursion, which can 
be used to replace the bracket in the cocycle (\ref{shJ_3_1}) by its definition:
\begin{multline}
d\{v_{1},\ldots,v_{k}\}=\\\label{shJ_4}
\textstyle \frac{1}{(k-1)!}\sum_{\sigma\in S_{k}}(-1)^{\sigma}
 e(\sigma;sx_{1},\ldots,sx_{k})
  (-1)^{(k-1)}di_{x_{\{v_{\sigma(1)},\ldots,v_{\sigma(k-1)}\}}}v_{\sigma(k)}\;.
\end{multline}
To rewrite the $k$-ary bracket in (\ref{shJ_3_2}), observe that its 
definition can be broken down into two parts. One that fixes the first argument 
and another one which shuffles the first argument to the $k$-th position. After simplification we get 
\begin{multline}
\textstyle \frac{1}{(k-1)!}\sum_{\sigma\in S_k}
 (-1)^{\sigma+1(k-1)}e(\sigma;v_{1},\ldots,v_{k})
  \{dv_{\sigma(1)},v_{\sigma(2)},\ldots,v_{\sigma(k)}\}=\\\label{shJ_5}
\textstyle \frac{1}{(k-2)!}\sum_{\sigma\in S_{k}}
 (-1)^{\sigma}e(\sigma;sx_{1},\ldots,sx_{k})
  i_{x_{\{dv_{\sigma(1)},v_{\sigma(2)},\ldots,v_{\sigma(k-1)}\}}}v_{\sigma(k)}\\
-\textstyle \frac{1}{(k-1)!}\sum_{\sigma\in S_{k}}
(-1)^{\sigma+(k-1)}e(\sigma;sx_{1},\ldots,sx_{k})\cdot{}\\
 \cdot(-1)^{|x_{\{v_{\sigma(1)},\ldots,v_{\sigma(k-1)}\}}|}
  i_{x_{\{v_{\sigma(1)},\ldots,v_{\sigma(k-1)}\}}}dv_{\sigma(k)}\;.
\end{multline}
Now insert (\ref{shJ_4}) and (\ref{shJ_5}) back into (\ref{shJ_3_1}) and 
(\ref{shJ_3_2}) and apply Cartans infinitesimal homotopy formula ().
This transform (\ref{shJ_2}) into
\begin{align}
&\nonumber
 \textstyle \frac{1}{(k-2)!}\sum_{\sigma\in S_{k}}
 (-1)^{\sigma}e(\sigma;sx_{1},\ldots,sx_{k})
  i_{x_{\{dv_{\sigma(1)},v_{\sigma(2)},\ldots,v_{\sigma(k-1)}\}}}v_{\sigma(k)}\\
&\label{shJ_6}
+\textstyle \frac{1}{(k-1)!}\sum_{\sigma\in S_{k}}
 (-1)^{\sigma+(k-1)}e(\sigma;sx_{1},\ldots,sx_{k})
  L_{x_{\{\sigma(1),\ldots,\sigma(k-1)\}}}v_{\sigma(k)}\\
&\nonumber
\!\begin{multlined}
+\textstyle\frac{1}{j!(k-j)!}\sum_{j=2}^{k-1}
 \sum_{\sigma\in S_k}\sgn{\sigma+j(k-j)}
e(\sigma;sx_1,\ldots,sx_k)\cdot{}\\
\phantom{mmmmmmmmmmmmmmmmmm}\cdot\{\{v_{\sigma(1)},\ldots, v_{\sigma(j)}\},
 v_{\sigma(j+1)},\ldots, v_{\sigma(k)}\}=0\;.
\end{multlined} 
\end{align}
To proceed we have to consider the cases $k=4$ and $k=5$ separately, since the
last terms in (\ref{shJ_6}) are different in those situations.
For $k=4$ equation (\ref{shJ_6}) is equal to
\begin{align}
&\nonumber
\textstyle \frac{1}{2}\sum_{\sigma\in S_{4}}
 (-1)^{\sigma}e(\sigma;sx_{1},\ldots,sx_{4})
  i_{x_{\{dv_{\sigma(1)},v_{\sigma(2)},v_{\sigma(3)}\}}}v_{\sigma(4)}\\
&\nonumber
-\textstyle \frac{1}{6}\sum_{\sigma\in S_{4}}
 (-1)^{\sigma}e(\sigma;sx_{1},\ldots,sx_{4})
  L_{x_{\{v_{\sigma(1)},v_{\sigma(2)},v_{\sigma(3)}\}}}v_{\sigma(4)}\\
&\label{shJ_7}
+\textstyle\frac{1}{4}\sum_{\sigma\in S_4}(-1)^{\sigma}
 e(\sigma;sx_1,\ldots,sx_4)
  \{\{v_{\sigma(1)}, v_{\sigma(2)}\},v_{\sigma(3)}, v_{\sigma(4)}\}\\
&\nonumber
 -\textstyle\frac{1}{6}\sum_{\sigma\in S_4}(-1)^{\sigma}
 e(\sigma;sx_1,\ldots,sx_4)
 	\{\{v_{\sigma(1)}, v_{\sigma(2)}, v_{\sigma(3)}\},v_{\sigma(4)}\}=0\;.
\end{align}
To see that (\ref{shJ_7}) is correct, expand the remaining nested 
brackets according to their definition. After 
simplification this gives
\begin{align}
&\label{shJ_8_1}
\textstyle \frac{1}{2}\sum_{\sigma\in S_{4}}
 (-1)^{\sigma}e(\sigma;sx_{1},\ldots,sx_{4})
  i_{x_{\{dv_{\sigma(1)},v_{\sigma(2)},v_{\sigma(3)}\}}}v_{\sigma(4)}\\
&\label{shJ_8_2}
-\textstyle \frac{1}{6}\sum_{\sigma\in S_{4}}
 (-1)^{\sigma}e(\sigma;sx_{1},\ldots,sx_{4})
  L_{x_{\{v_{\sigma(1)},v_{\sigma(2)},v_{\sigma(3)}\}}}v_{\sigma(4)}\\
&\label{shJ_8_3}
-\textstyle \sum_{\sigma\in S_{4}}(-1)^{\sigma}
 e(\sigma;sx_{1},\ldots,sx_{4})
  i_{[x_{\sigma(3)},[x_{\sigma(2)},x_{\sigma(1)}]]}v_{\sigma(4)}\\
&\label{shJ_8_4}
-\textstyle \frac{1}{2}\sum_{\sigma\in S_{4}}
 (-1)^{\sigma}e(\sigma;sx_{1},\ldots,sx_{4})
  i_{[x_{\sigma(2)},x_{\sigma(1)}]}
   L_{x_{\sigma(3)}}v_{\sigma(4)}\\
&\label{shJ_8_5}
+\textstyle \frac{1}{6}\sum_{\sigma\in S_{4}}
 (-1)^{\sigma}e(\sigma;sx_{1},\ldots,sx_{4})
  L_{x_{\{v_{\sigma(1)},v_{\sigma(2)},v_{\sigma(3)}\}}}v_{\sigma(4)}\\
&\label{shJ_8_6}
-\textstyle \frac{1}{2}\sum_{\sigma\in S_{4}}
 (-1)^{\sigma}e(\sigma;sx_{1},\ldots,sx_{4})
  (-1)^{|x_{\sigma(1)}|}
   L_{x_{\sigma(1)}}i_{[x_{\sigma(3)},x_{\sigma(2)}]}v_{\sigma(4)}=0\;.
\end{align}
Now (\ref{shJ_8_2}) cancels againt (\ref{shJ_8_5}) and using equation (), 
the terms (\ref{shJ_8_4}) and (\ref{shJ_8_6}) can be combined into
$\textstyle \sum_{\sigma\in S_{4}}(-1)^{\sigma}
 e(\sigma;sx_{1},\ldots,sx_{4})
  i_{[x_{\sigma(3)},[x_{\sigma(2)},x_{\sigma(1)}]]}v_{\sigma(4)}$.
This expression vanishes together with  (\ref{shJ_8_3}),
since the Schouten-Nijenhuis brackets satisfies the strict Jacobi equation of a
Gerstenhaber algebra. The vanishing of (\ref{shJ_8_1}) follows from 
theorem (), proposition () and the kernel property, since 
equation (\ref{shJ_1}) holds for $k=3$.
This proofs (\ref{shJ_1}) for $k=4$.

Now suppose $k\geq 5$. To proceed we separate all terms 
in equation (\ref{shJ_6}), that contains the homotopy Poisson $2$-bracket.

\begin{align}
&\nonumber
\textstyle \frac{1}{(k-2)!}\sum_{\sigma\in S_{k}}
 (-1)^{\sigma}e(\sigma;sx_{1},\ldots,sx_{k})
  i_{x_{\{dv_{\sigma(1)},v_{\sigma(2)},\ldots,v_{\sigma(k-1)}\}}}v_{\sigma(k)}\\
&\nonumber
+\textstyle \frac{1}{(k-1)!}\sum_{\sigma\in S_{k}}
 (-1)^{\sigma+(k-1)}e(\sigma;sx_{1},\ldots,sx_{k})
  L_{x_{\{\sigma(1),\ldots,\sigma(k-1)\}}}v_{\sigma(k)}\\
&\label{shJ_9_1}
+\textstyle \frac{1}{2(k-2)!}\sum_{\sigma\in S_k}
 (-1)^{\sigma}e(\sigma;sx_{1},\ldots,sx_{k})
  \{\{v_{\sigma(1)},v_{\sigma(2)}\},v_{\sigma(3)},\ldots,v_{\sigma(k)}\}\\
&\label{shJ_9_2}
+\textstyle \frac{1}{(k-1)!}\sum_{\sigma\in S_k}
 (-1)^{\sigma+(k-1)}e(\sigma;sx_{1},\ldots,sx_{k})
  \{\{v_{\sigma(1)},\ldots,v_{\sigma(k-1)}\},v_{\sigma(k)}\}\\  
&\nonumber
\!\begin{multlined}
+\textstyle\sum_{j=3}^{k-2}\frac{1}{j!(k-j)!}\sum_{\sigma\in Sh(j,k-j)}
 \sgn{\sigma+j(k-j)}e(\sigma;sx_1,\ldots,sx_k)\cdot{}\\
\phantom{mmmmmmmmmmmmmmmmmm}\cdot\{\{v_{\sigma(1)},\ldots, v_{\sigma(j)}\},
 v_{\sigma(j+1)},\ldots, v_{\sigma(k)}\}=0\;.
\end{multlined} 
\end{align}
Now consider (\ref{shJ_9_1}) and (\ref{shJ_9_2}) and express the nested brackets
in terms of their definitions. Since $k\geq 5$ this leads to 
the following identities:
\begin{multline*}
\textstyle \frac{1}{2(k-2)!}\sum_{\sigma\in S_k}
 (-1)^{\sigma}e(\sigma;sx_{1},\ldots,sx_{k})
  \{\{v_{\sigma(1)},v_{\sigma(2)}\},v_{\sigma(3)},\ldots,v_{\sigma(k)}\}=\\
\textstyle \frac{1}{2(k-3)!}\sum_{\sigma\in S_{k}}
 (-1)^{\sigma+(k-2)}e(\sigma;sx_{1},\ldots,sx_{k})
   i_{x_{\{\{v_{\sigma(1)},v_{\sigma(2)}\},
    v_{\sigma(3)},\ldots,v_{\sigma(k-1)}\}}}v_{\sigma(k)}\\
-\textstyle \frac{1}{(k-2)!}\sum_{\sigma\in S_{k}}
 (-1)^{\sigma}e(\sigma;sx_{1},\ldots,sx_{k})
  i_{x_{\{v_{\sigma(1)},\ldots,v_{\sigma(k-2)}\}}}
   L_{x_{\sigma(k-1)}}v_{\sigma(k)} \;.
\end{multline*}
\begin{multline*}
\textstyle \frac{1}{(k-1)!}\sum_{\sigma\in S_k}
 (-1)^{\sigma+(k-1)}e(\sigma;sx_{1},\ldots,sx_{k})
  \{\{v_{\sigma(1)},\ldots,v_{\sigma(k-1)}\},v_{\sigma(k)}\}=\\
\textstyle -\frac{1}{(k-2)!}\sum_{\sigma\in S_{k}}
 (-1)^{\sigma+(k-1)|x_{\sigma(1)}|}
  e(\sigma;sx_{1},\ldots,sx_{k})
   L_{x_{\sigma(1)}}i_{x_{\{v_{\sigma(2)},\ldots,v_{\sigma(k-1)}\}}}v_{\sigma(k)}\\
-\textstyle \frac{1}{(k-1)!}\sum_{\sigma\in S_{k}}
 (-1)^{\sigma+(k-1)}e(\sigma;sx_{1},\ldots,sx_{k})
  L_{x_{\{v_{\sigma(1)},\ldots,v_{\sigma(k-1)}\}}}v_{\sigma(k)}\;.    
\end{multline*}
Now using $[x_j,x_i]=\frac{1}{2}x_{\{v_i,v_j\}}$ and 
equation (), we have the following additional identity
\begin{multline*}
\textstyle -\frac{1}{(k-2)!}\sum_{\sigma\in S_{k}}
 (-1)^{\sigma+(k-1)|x_{\sigma(1)}|}e(\sigma;sx_{1},\ldots,sx_{k})
  L_{x_{\sigma(1)}}i_{x_{\{v_{\sigma(2)},\ldots,v_{\sigma(k-1)}\}}}v_{\sigma(k)}\\
-\textstyle \frac{1}{(k-2)!}\sum_{\sigma\in S_{k}}(-1)^{\sigma}e(\sigma;sx_{1},\ldots,sx_{k})i_{x_{\{v_{\sigma(1)},\ldots,v_{\sigma(k-2)}\}}}L_{x_{\sigma(k-1)}}v_{\sigma(k)}=\\
\phantom{mmmm}\textstyle \frac{1}{(k-2)!}\sum_{\sigma\in S_{k}}
 (-1)^{\sigma}e(\sigma;sx_{1},\ldots,sx_{k})
  i_{[x_{\sigma(k-2)},x_{\{v_{\sigma(1)},\ldots,
   v_{\sigma(k-2)}\}}]}v_{\sigma(k)}=\\
\textstyle \frac{1}{2(k-2)!}\sum_{\sigma\in S_{k}}
 (-1)^{\sigma}e(\sigma;sx_{1},\ldots,sx_{k})
  i_{x_{\{\{v_{\sigma(1)},\ldots,
   v_{\sigma(k-2)}\},v_{\sigma(k-1)}\}}}v_{\sigma(k)}   
\end{multline*}
Substituting this back into expression (\ref{shJ_9_1}) and
(\ref{shJ_9_2}) we can further rewrite the $n$-plectic Jacobi equation 
as
\begin{align}
&\nonumber
\textstyle \frac{1}{(k-2)!}\sum_{\sigma\in S_{k}}
 (-1)^{\sigma}e(\sigma;sx_{1},\ldots,sx_{k})
  i_{x_{\{dv_{\sigma(1)},v_{\sigma(2)},\ldots,v_{\sigma(k-1)}\}}}v_{\sigma(k)}\\
&\nonumber
+\textstyle \frac{1}{2(k-3)!}\sum_{\sigma\in S_{k}}
 (-1)^{\sigma+(k-2)}e(\sigma;sx_{1},\ldots,sx_{k})
   i_{x_{\{\{v_{\sigma(1)},v_{\sigma(2)}\},
    v_{\sigma(3)},\ldots,v_{\sigma(k-1)}\}}}v_{\sigma(k)}\\
&\label{shJ_10}
+\textstyle \frac{1}{2(k-2)!}\sum_{\sigma\in S_{k}}
 (-1)^{\sigma}e(\sigma;sx_{1},\ldots,sx_{k})
  i_{x_{\{\{v_{\sigma(1)},\ldots,
   v_{\sigma(k-2)}\},v_{\sigma(k-1)}\}}}v_{\sigma(k)}\\   
&\nonumber
\!\begin{multlined}
+\textstyle\frac{1}{j!(k-j)!}\sum_{j=3}^{k-2}
 \sum_{\sigma\in Sh(j,k-j)}\sgn{\sigma+j(k-j)}
e(\sigma;sx_1,\ldots,sx_k)\cdot{}\\
\phantom{mmmmmmmmmmmmmmmmmm}\cdot\{\{v_{\sigma(1)},\ldots, v_{\sigma(j)}\},
 v_{\sigma(j+1)},\ldots, v_{\sigma(k)}\}=0\;.
\end{multlined} 
\end{align}
Since this equation behaves differently for $k=5$ and $k> 5$,
first assume $k=5$. In this case the last row of (\ref{shJ_10})
can be computed by applying the definition of the nested Poisson $k$-brackets.
After simplification this gives  

\begin{multline}
\label{shJ_11}
\textstyle \frac{1}{3!2!}\sum_{\sigma\in S_5}(-1)^{\sigma+3\cdot 2}
 e(\sigma;sx_{1},\ldots,sx_{5})
 \{\{v_{\sigma(1)},v_{\sigma(2)},v_{\sigma(3)}\},
  v_{\sigma(4)},v_{\sigma(5)}\}=\\
\textstyle \frac{1}{4}\sum_{\sigma\in S_{5}}
 (-1)^{\sigma}e(\sigma;sx_{1},\ldots,sx_{5})
  (-1)^{|x_{\sigma(1)}|+|x_{\sigma(2)}|}i_{[x_{\sigma(2)},x_{\sigma(1)}]}
   i_{[x_{\sigma(4)},x_{\sigma(3)}]}v_{\sigma(5)}\\
+\textstyle \frac{1}{6}\sum_{\sigma\in S_{5}}
 (-1)^{\sigma}e(\sigma;sx_{1},\ldots,sx_{5})
  i_{[x_{\sigma(4)},x_{\{v_{\sigma(1)},v_{\sigma(2)},v_{\sigma(3)}\}}]}
   v_{\sigma(5)}
\end{multline}
Since $(-1)^{|x_{1}|+|x_{2}|}i_{[x_{2},x_{1}]}
 i_{[x_{4},x_{3}]}(\cdot)=-(-1)^{|x_{3}|+|x_{4}|}i_{[x_{4},x_{3}]}
  i_{[x_{2},x_{1}]}(\cdot)$ the terms in the second row vanish. Using $[x_j,x_i]=\frac{1}{2}x_{\{v_i,v_j\}}$
we can substitute (\ref{shJ_11}) back into (\ref{shJ_10}) to rewrite
the $n$-plectic Jacobi expression for $k=5$ into
\begin{align*}
&-\textstyle \frac{1}{3!}\sum_{\sigma\in S_{5}}
 (-1)^{\sigma+1\cdot 3}e(\sigma;sx_{1},\ldots,sx_{5})
  i_{x_{\{dv_{\sigma(1)},v_{\sigma(2)},v_{\sigma(3)},v_{\sigma(4)}\}}}
   v_{\sigma(5)}\\
&-\textstyle \frac{1}{2\cdot 2}\sum_{\sigma\in S_{5}}
 (-1)^{\sigma+2\cdot 2}e(\sigma;sx_{1},\ldots,sx_{5})
   i_{x_{\{\{v_{\sigma(1)},v_{\sigma(2)}\},
    v_{\sigma(3)},v_{\sigma(4)}\}}}v_{\sigma(5)}\\
&-\textstyle \frac{1}{3!}\sum_{\sigma\in S_{5}}
 (-1)^{\sigma+3\cdot 1}e(\sigma;sx_{1},\ldots,sx_{5})
  i_{x_{\{\{v_{\sigma(1)},v_{\sigma(2)},
   v_{\sigma(3)}\},v_{\sigma(4)}\}}}v_{\sigma(5)}\;.
\end{align*}
Since the $n$-plectic Jacobi equation is satisfied for $k=4$, 
proposition () and the kernel property (), imply that this expression 
vanishes and therefore that the $n$-plectic Jacobi equation is satisfied
for $k=5$.

Now suppose $k\geq 6$ and separate all terms that contain the
homotopy Poisson 3-bracket from the last row of equation (\ref{shJ_10}). This gives
\begin{align}
&\nonumber
\textstyle \frac{1}{(k-2)!}\sum_{\sigma\in S_{k}}
 (-1)^{\sigma}e(\sigma;sx_{1},\ldots,sx_{k})
  i_{x_{\{dv_{\sigma(1)},v_{\sigma(2)},\ldots,v_{\sigma(k-1)}\}}}v_{\sigma(k)}\\
&\nonumber
+\textstyle \frac{1}{2(k-3)!}\sum_{\sigma\in S_{k}}
 (-1)^{\sigma+(k-2)}e(\sigma;sx_{1},\ldots,sx_{k})
   i_{x_{\{\{v_{\sigma(1)},v_{\sigma(2)}\},
    v_{\sigma(3)},\ldots,v_{\sigma(k-1)}\}}}v_{\sigma(k)}\\
&\label{shJ_12}
+\textstyle \frac{1}{2(k-2)!}\sum_{\sigma\in S_{k}}
 (-1)^{\sigma}e(\sigma;sx_{1},\ldots,sx_{k})
  i_{x_{\{\{v_{\sigma(1)},\ldots,
   v_{\sigma(k-2)}\},v_{\sigma(k-1)}\}}}v_{\sigma(k)}\\
&\nonumber
+\textstyle \frac{1}{3!(k-3)!}\sum_{\sigma\in S_k}
 (-1)^{\sigma+3(k-3)}e(\sigma;sx_{1},\ldots,sx_{k})
  \{\{v_{\sigma(1)},v_{\sigma(2)},v_{\sigma(3)}\},
    v_{\sigma(4)},\ldots,v_{\sigma(k)}\}\\
&\nonumber
+\textstyle \frac{1}{2(k-2)!}\sum_{\sigma\in S_k}
 (-1)^{\sigma+(k-2)(2)}e(\sigma;sx_{1},\ldots,sx_{k})\{\{v_{\sigma(1)},\ldots,v_{\sigma(k-2)}\},v_{\sigma(k-1)},v_{\sigma(k)}\}\\      
&\nonumber
\!\begin{multlined}
+\textstyle\frac{1}{j!(k-j)!}\sum_{j=4}^{k-3}
 \sum_{\sigma\in Sh(j,k-j)}\sgn{\sigma+j(k-j)}
e(\sigma;sx_1,\ldots,sx_k)\cdot{}\\
\phantom{mmmmmmmmmmmmmmmmmm}\cdot\{\{v_{\sigma(1)},\ldots, v_{\sigma(j)}\},
 v_{\sigma(j+1)},\ldots, v_{\sigma(k)}\}=0\;.
\end{multlined} 
\end{align}
Next consider the two rows which contain the homotopy Poisson $3$-bracket
and apply the definition of all involved nested brackets. After simplification
this gives the following two identities:
\begin{align}
&\nonumber
\!\begin{multlined}
\textstyle \frac{1}{3!(k-3)!}\sum_{\sigma\in S_k}
 (-1)^{\sigma+3(k-3)}e(\sigma;sx_{1},\ldots,sx_{k})\cdot{}\\
\phantom{mmmmmmmmmmmmmmm}\cdot\{\{v_{\sigma(1)},v_{\sigma(2)},v_{\sigma(3)}\},
    v_{\sigma(4)},\ldots,v_{\sigma(k)}\}=
\end{multlined}\\
&\textstyle \frac{1}{3!(k-4)!}\sum_{\sigma\in S_{k}}(-1)^{\sigma}e(\sigma;sx_{1},\ldots,sx_{k})i_{x_{\{\{v_{\sigma(1)},v_{\sigma(2)},v_{\sigma(3)}\},v_{\sigma(4)},\ldots,v_{\sigma(k-1)}\}}}v_{\sigma(k)}
\end{align}
\begin{multline}\label{shJ_13}
\textstyle +\frac{1}{2!(k-3)!}\sum_{\sigma\in S_{k}}
 (-1)^{\sigma+3(k-3)}e(\sigma;sx_{1},\ldots,sx_{k})\cdot{}\\
\cdot(-1)^{\sum_{i=1}^{k-3}|x_{\sigma(i)}|}
   i_{\{v_{\sigma(1)},...,v_{\sigma(k-3)}\}}
    i_{[x_{\sigma(k-1)},x_{\sigma(k-2)}]}v_{\sigma(k)}
\end{multline}
\begin{multline*}
\textstyle \frac{1}{2(k-2)!}\sum_{\sigma\in S_k}
 (-1)^{\sigma}e(\sigma;sx_{1},\ldots,sx_{k})
  \{\{v_{\sigma(1)},\ldots,v_{\sigma(k-2)}\},
   v_{\sigma(k-1)},v_{\sigma(k)}\}=\\
+\textstyle \frac{1}{(k-2)!}\sum_{\sigma\in S_{k}}
 (-1)^{\sigma}e(\sigma;sx_{1},\ldots,sx_{k})
  i_{[x_{\sigma(k-1)},x_{\{v_{\sigma(1)},\ldots,v_{\sigma(k-2)}\}}]}v_{\sigma(k)}
\end{multline*}
\begin{multline}\label{shJ_14}
+\textstyle \frac{1}{(k-3)!2!}\sum_{\sigma\in S_{k}}
 (-1)^{\sigma+(k-3)}e(\sigma;sx_{1},\ldots,sx_{k})\cdot{}\\
 \cdot(-1)^{(k-2)(|x_{\sigma(1)}|+|x_{\sigma(2)}|)}
  i_{[x_{\sigma(2)},x_{\sigma(1)}]}
   i_{x_{\{v_{\sigma(3)},\ldots,v_{\sigma(k-1)}\}}}v_{\sigma(k)}
\end{multline}
From $i_xi_{x'}=(-1)^{|x||x'|}i_{x'}i_x$ follows that (\ref{shJ_13}) cancels
against (\ref{shJ_14}). Insert the remaining terms back into (\ref{shJ_12}) and
use $[x_j,x_i]=\frac{1}{2}x_{\{v_i,v_j\}}$ gives the following expression for
the $n$-plectic Jacobi identity: 
\begin{align}
&\nonumber
\textstyle \frac{1}{(k-2)!}\sum_{\sigma\in S_{k}}
 (-1)^{\sigma}e(\sigma;sx_{1},\ldots,sx_{k})
  i_{x_{\{dv_{\sigma(1)},v_{\sigma(2)},\ldots,v_{\sigma(k-1)}\}}}v_{\sigma(k)}\\
&\nonumber
+\textstyle \frac{1}{2(k-3)!}\sum_{\sigma\in S_{k}}
 (-1)^{\sigma+k}e(\sigma;sx_{1},\ldots,sx_{k})
   i_{x_{\{\{v_{\sigma(1)},v_{\sigma(2)}\},
    v_{\sigma(3)},\ldots,v_{\sigma(k-1)}\}}}v_{\sigma(k)}\\
&\nonumber
+\textstyle \frac{1}{3!(k-4)!}\sum_{\sigma\in S_{k}}(-1)^{\sigma}e(\sigma;sx_{1},\ldots,sx_{k})i_{x_{\{\{v_{\sigma(1)},v_{\sigma(2)},v_{\sigma(3)}\},v_{\sigma(4)},\ldots,v_{\sigma(k-1)}\}}}v_{\sigma(k)}\\
&\label{shJ_15}
+\textstyle \frac{1}{(k-2)!}\sum_{\sigma\in S_{k}}
 (-1)^{\sigma}e(\sigma;sx_{1},\ldots,sx_{k})
  i_{x_{\{\{v_{\sigma(1)},\ldots,
   v_{\sigma(k-2)}\},v_{\sigma(k-1)}\}}}v_{\sigma(k)}\\     
&\nonumber
\!\begin{multlined}
+\textstyle\frac{1}{j!(k-j)!}\sum_{j=4}^{k-3}
 \sum_{\sigma\in S_k}\sgn{\sigma+j(k-j)}
e(\sigma;sx_1,\ldots,sx_k)\cdot{}\\
\phantom{mmmmmmmmmmmmmmmmmm}\cdot\{\{v_{\sigma(1)},\ldots, v_{\sigma(j)}\},
 v_{\sigma(j+1)},\ldots, v_{\sigma(k)}\}=0\;.
\end{multlined} 
\end{align}
To rewrite the last row in (\ref{shJ_15}), suppose $4\leq j \leq k-3$ is fixed and
rephrase the outer Poisson bracket in terms of its definition.
\begin{multline*}
\textstyle \sum_{\sigma\in Sh(j,k-j)}
 (-1)^{\sigma+j(k-j)}e(\sigma;sx_{1},\ldots,sx_{k})\cdot{}\\
\cdot\{\{f_{\sigma(1)},\ldots,f_{\sigma(j)}\},
 f_{\sigma(j+1)},\ldots,f_{\sigma(k)}\}=
\end{multline*}
\begin{multline*}
\textstyle \frac{1}{j!(k-j-1)!}\sum_{\sigma\in S_{k}}
 (-1)^{\sigma+(j+1)(k-j)}e(\sigma;sx_{1},\ldots,sx_{k})\cdot{}\\
\cdot i_{x_{\{\{f_{\sigma(1)},\ldots,f_{\sigma(j)}\},
   f_{\sigma(j+1)},\ldots,f_{\sigma(k-1)}\}}}f_{\sigma(k)}
\end{multline*}
\begin{multline}\label{shJ_16}
\textstyle +\frac{1}{j!(k-j)!}\sum_{\sigma\in S_{k}}
 (-1)^{\sigma+j(k-j)}e(\sigma;sx_{1},\ldots,sx_{k})\cdot{}\\
\cdot(-1)^{j\sum_{i_{1}=1}^{k-j}|x_{\sigma(i_{1})}|}i_{x_{\{f_{\sigma(1)},\ldots,f_{\sigma(k-j)}\}}}\{f_{\sigma(k-j+1)},\ldots,f_{\sigma(k)}\}\;.
\end{multline}
Expending the remaining Poisson bracket in (\ref{shJ_16}), using 
$i_xi_{x'}=(-1)^{|x||x'|}i_{x'}i_x$ leads to the following identity
\begin{multline*}
\textstyle +\frac{1}{j!(k-j)!}\sum_{\sigma\in S_{k}}(-1)^{\sigma+j(k-j)}e(\sigma;sx_{1},\ldots,sx_{k})\\
 (-1)^{j\sum_{i=1}^{k-j}|x_{\sigma(i)}|}i_{x_{\{f_{\sigma(1)},\ldots,f_{\sigma(k-j)}\}}}\{f_{\sigma(k-j+1)},\ldots,f_{\sigma(k)}\}=
\end{multline*}
\begin{multline*}
\textstyle -\frac{1}{(k-j)!j!}\sum_{\sigma\in S_{k}}
 (-1)^{\sigma+(k-j)j}e(\sigma;sx_{1},\ldots,sx_{k})\\
 (-1)^{(k-j)\sum_{i_{1}=1}^{j}|x_{\sigma(i_{1})}|}
  i_{x_{\{f_{\sigma(1)},\ldots,f_{\sigma(j)}\}}}
   \{f_{\sigma(j+1)},...,f_{\sigma(k)}\}
\end{multline*}
and it follows that for given $j$ each term (\ref{shJ_16}) cancel against 
the same term for $(k-j)$. Taking this
into account the left side of the $n$-plectic Jacobi equation (\ref{shJ_2}) 
can finally be written as
\begin{multline*}
\textstyle \sum_{j=1}^{k-1}\frac{1}{j!((k-1)-j)!}\sum_{\sigma\in S_{k}}
 (-1)^{\sigma+j((k-1)-j)}e(\sigma;sx_{1},\ldots,sx_{k})\cdot{}\\
\phantom{mmmmmmmmmmmmmmmmmmm}\cdot i_{x_{\{\{f_{\sigma(1)},\ldots,f_{\sigma(j)}\},
   f_{\sigma(j+1)},\ldots,f_{\sigma(k-1)}\}}}v_{\sigma(k)}
\end{multline*}
By the induction hypothesis, equation (\ref{shJ_2}) is satisfied for 
$(k-1)$ and proposition () as well as the kernel property () imply that 
the previous expression vanishes and therefore that equation 
(\ref{shJ_2}) is satisfied for $k$, too. This completes the induction.
\end{proof}
\end{appendix}

\end{document}